\documentclass[review,hidelinksonefignum,onetabnum]{siamart220329}
\usepackage{bm}
\usepackage{graphicx}
\usepackage{epstopdf}
\usepackage{epsfig}
\usepackage{subcaption}
\usepackage{amsmath,amssymb, amsfonts,mathtools}

\newtheorem{thm}{Theorem}[section]

\newtheorem{dfn}[thm]{Definition}

\newtheorem{pro}[thm]{Proposition}

\nolinenumbers

\usepackage{lipsum}
\usepackage{amsfonts}
\usepackage{graphicx}
\usepackage{epstopdf}

\usepackage{algorithmic}
\ifpdf
  \DeclareGraphicsExtensions{.eps,.pdf,.png,.jpg}
\else
  \DeclareGraphicsExtensions{.eps}
\fi

\newsiamremark{remark}{Remark}
\newsiamremark{hypothesis}{Hypothesis}
\crefname{hypothesis}{Hypothesis}{Hypotheses}
\newsiamthm{claim}{Claim}

\headers{A convergent interacting particle method}{T. Zhang, Z. Wang, J. Xin and Z. Zhang}

\title{A convergent interacting particle method for computing KPP front speeds in random flows\thanks{To appear in SIAM/ASA J Uncertainty Quantification.
}}

\author{Tan Zhang\thanks{Department of Mathematics, The University of Hong Kong, Pokfulam Road, Hong Kong SAR, P.R.China.
  (\email{thta@connect.hku.hk}).}
\and Zhongjian Wang\thanks{Division of Mathematical Sciences, School of Physical and Mathematical Sciences, Nanyang  Technological University, Singapore 637371.
  (\email{zhongjian.wang@ntu.edu.sg}.)}
\and Jack Xin\thanks{Department of Mathematics, University of California at Irvine, Irvine, CA 92697, USA.
  (\email{jxin@math.uci.edu}.)}
\and Zhiwen Zhang\thanks{Corresponding author. Department of Mathematics, The University of Hong Kong, Pokfulam Road, Hong Kong SAR, P.R.China. Materials Innovation Institute for Life Sciences and Energy (MILES), HKU-SIRI, Shenzhen, P.R. China. (\email{zhangzw@hku.hk}).} }

\usepackage{amsopn}

\ifpdf
\hypersetup{
  pdftitle={A convergent interacting particle
  method for computing KPP front speeds in random flows},
  pdfauthor={T. Zhang, Z. Wang, J. Xin and Z. Zhang}
}
\fi

\begin{document}

\maketitle

\begin{abstract}
\noindent
This paper aims to efficiently compute the spreading speeds of reaction-diffusion-advection fronts in divergence-free random flows under the Kolmogorov-Petrovsky-Piskunov nonlinearity. We develop a stochastic interacting particle method (IPM) for the reduced principal eigenvalue (Lyapunov exponent) problem of an associated linear advection-diffusion operator with spatially random coefficients. The Fourier representation of the random advection field and the Feynman-Kac formula of the principal eigenvalue (Lyapunov exponent) form the foundation of our method, which is implemented as a genetic evolution algorithm. The particles undergo advection-diffusion and mutation/selection through a fitness function originated in the Feynman-Kac semigroup. We analyze the convergence of the algorithm based on operator splitting and present numerical results on representative flows such as 2D cellular flow and 3D Arnold-Beltrami-Childress (ABC) flow under random perturbations. The 2D examples serve as a consistency check with semi-Lagrangian computation. The 3D results demonstrate that IPM, being mesh-free and self-adaptive,  is easy to implement and efficient for computing front spreading speeds in the advection-dominated regime for high-dimensional random flows on unbounded domains where no truncation is needed.
\end{abstract}



\section{Introduction}\label{sec:intro}
\noindent Front propagation in complex fluid flows is a nonlinear phenomenon commonly found in many scientific and technological areas such as chemical reaction fronts in liquid, transport in porous media, and turbulence combustion \cite{xin2009}. If the fluid velocity field is simulated by stochastic processes with known statistics, 
one challenging mathematical problem is to characterize and compute the propagation velocity of the front (the dynamic asymptotic stability of the front in the large time limit) by utilizing the governing partial differential equations. 
An extensively studied model is the reaction-diffusion-advection (RDA) equation with Kolmogorov-Petrovsky-Piskunov (KPP) nonlinearity \cite{kolmogorov1937}:
\begin{equation}
    u_t = \kappa \Delta_{\bm x} u + (\bm v \cdot \nabla_{\bm x}) u + f(u), \quad t \in \mathbb{R}^+,\quad \bm x = (x_1,...,x_d)^T \in \mathbb{R}^d.
    \label{KPPOriginalEquation}
\end{equation}
Here in equation (\ref{KPPOriginalEquation}), $u$ is the concentration of reactant or population, $\kappa$ is the diffusion constant, and $\bm v$ is a prescribed incompressible velocity field. We focus on the case with KPP reaction term $f(u) = u(1-u)$ satisfying $f(u) \leq u f'(0)$. In the following analysis and numerical examples, we will fix $\kappa = 1$, while changing the magnitude of the velocity field $\bm v$, or equivalently varying P\'{e}clet number $Pe$ which measures the relative strength of advection and diffusion.  

The propagation of reaction-diffusion fronts in fluid flows has been an actively researched field for decades \cite{gartner1979,xin1992existence,majda1994large,xin2000front,berestycki2005speed,Nolen2008,NolenJack2009,xin2009,Nolen2012existence}. The dynamic asymptotic stability of the system \eqref{KPPOriginalEquation} means that if the initial data are prescribed in the form: $u_0(\bm x) = U_{c}(\bm x) + \tilde{u}(\bm x)$, where $U_{c}(\bm x)$ is a front profile corresponding to the wave speed $c$ with propagation direction $\bm z$ and $\tilde{u}(\bm x)$ is a smooth and spatially decaying perturbation, then $u(\bm x, t)$ converges to $U_{c}(\bm x - c t \bm z+ \bm \xi_0 )$ in a suitably weighted Banach space as $t\to \infty$ for some constant vector $\bm \xi_0$. Due to the spatial translation-invariance of the original equation, it may have a family of traveling fronts for each allowable wave speed $c$. We shall call the front moves with the slowest speed in absolute value as the critical front and its corresponding wave speed $c = c^*$ as the front speed \cite{Nolen2007,xin2000front}. The linear corrector equation and the variational formula aided the computation of KPP front speeds. For spatially periodic velocity field $\bm v = \bm v(\bm x)$, the minimal front speed formula  was known  \cite{Freidlin1985,gartner1979} around 1980s: $c^*(\bm z) = \inf_{(\bm z, \lambda \bm e) > 0} \mu(\lambda \bm e) / (\bm z, \lambda \bm e)$. Here $\bm z \in \mathbb{R}^d$ is the front propagation direction,  $\mu(\lambda \bm e)$ is the principal eigenvalue of the linear advection-reaction-diffusion operator  $\mathcal{A}^{\lambda}$ parameterized by dual unit vector
$\bm e \in \mathbb{R}^d$ and $\lambda \in \mathbb{R}$:
\begin{equation}
    \mathcal{A}^{\lambda} \equiv \kappa \Delta_{\bm x} \cdot + 
    (-2\kappa \lambda \bm e  + \bm v) \cdot \nabla_{\bm x} \cdot + (\kappa  \lambda^2 - \lambda \bm v \cdot \bm e + f'(0)),
    \label{peig}
\end{equation}   
on a $d$-dimensional torus $\mathbb{T}^d$. When $\bm v$ is space-time ergodic,  the $c^*$ variational formula generalizes \cite{NolenJack2009} with $\mu$ being the principal Lyapunov exponent and  $\mathbb{R}^d$ in place of  $\mathbb{T}^d$. Existing works include the adaptive finite element method (FEM) for steady flows  \cite{shen2013finite2d, shen2013finite3d}; edge-averaged FEM with algebraic multigrid acceleration  \cite{zu2015} for two dimensional (2D) time-periodic flows; semi-Lagrangian (SL) method for 2D random flows \cite{Nolen2008}. 
When flow is random with large amplitude and spatially 3D, the FEM and SL calculations of $\mu$ become incredibly expensive. An alternative is the (fully Lagrangian, mesh-free, and self-adaptive) interacting particle method (IPM, \cite{Junlong2022}) successfully studied for cellular and chaotic flows. This motivates us to develop IPM for computing KPP front speeds in random flows in this paper.

We first combine operator splitting and random Fourier methods to approximate the underlying linear advection-diffusion operator (Eq.\eqref{linearCorrEqu}), then leverage the Feynman-Kac (FK)  representation of the parabolic evolution and develop an IPM to compute $c^*$. Since directly approximating the FK formula is unstable,  we study a normalized version for $\mu$ based on the corresponding FK semigroups. Using perturbation theory (Proposition \ref{EigExistforrandomperturbation1}), we guarantee the existence of $\mu$ when a relatively small random field with an equispaced discrete spectrum is added to a deterministic flow
on $\mathcal{D}$, where $\mathcal{D} = [0, \frac{1}{\Delta k}]^d$ is a $d$-dimensional torus and $\Delta k$ is the frequency step of $\bm v$. The random flow $\bm v$ with the equispaced discrete spectrum is realizable by the random Fourier method and approximated by the series $\{\bm v_j\}$ along with the convergence of their expectations in $L^2(\mathcal{D})$ ( i.e. as $j \to \infty, \mathbb{E}[||\bm v -\bm v_j||_{L^2(\mathcal{D})}] \to 0$,  Eq.\eqref{ErrestforEwvr_vjL2} ) and $\mu$ approximations (Theorem \ref{Thm::ConvOfPrinEigofRFM}). Next, we estimate the approximation generated by the random Fourier method associated with the solution operator of non-autonomous evolution equations (Theorem \ref{Thm::Errofe3ande}). It demonstrates that the solution operator approximated by the random Fourier method has a convergence rate of $O(\sqrt{(\Delta t)^2 +||\bm v - \bm v_j||^2_{L^2(\mathcal{D})} })$, where $\Delta t$ denotes the discrete time step.
We also provide an error estimate of the IPM on $c^*$ in random flows on $\mathcal{D}$. It indicates that: \textit{the IPM has the following convergence result:} 
\begin{equation}
    \mu_{IPM}(\lambda) = \mu(\lambda) + O\big((1-\frac{\theta}{\vartheta})^M\big) + O\big((\Delta t)^{\frac{1}{2}}\big) + O( ||\bm v - \tilde{\bm v}||_{L^2(\mathcal{D})}),
\end{equation}
\textit{where $\mu, \mu_{IPM}$ are the corresponding principle eigenvalues of the original system and the IPM. $\tilde{\bm v}$ is an approximation to $\bm v$ and is generated by the random Fourier method. $0<\theta <\vartheta$ are parameters and $M$ is the iteration number.} (Theorem \ref{Thm::TotalErrEstforprinEig38})  

Numerically, we observe the convergence of $c^*$ in the large volume limit of $\mathcal{D}$ to 
 the values from a direct IPM computation on $\mathbb{R}^d$, encouraging a future study along the line of periodization of stochastic homogenization (see \cite{Otto2012,Otto2015,Otto2021} for elliptic problems). Flexible to both $\mathcal{D}$ and $\mathbb{R}^d$, 
IPM is advantageous for high-dimensional and stochastic flows where the issue of unbounded domain arises in (\ref{peig}) and is expensive for traditional Eulerian methods. For time complexity, let $N_x$ denote the number of grid points in each dimension of Eulerian methods and $N$ represent the Monte Carlo number in IPM. With the same iteration number $n$ and time step $\Delta t$, the time complexity of traditional Eulerian methods (e.g., Semi-Lagrangian + Crank-Nicksion method in \cite{Nolen2008}) is $O((N_x)^d)$, while the IPM has a time complexity of $O(N d)$. Comparing the two approaches, IPM is {\it self-adaptive and free from domain truncation}. For particle methods on related  effective diffusivity problems in periodic and stochastic flows, see \cite{PavliotisStuart:05,PavliotisStuart:07, PSZ_2009,
WangXinZhang:18,SharpMMS_21,IPM_2020,TDChaoticFlow_22}.

The rest of the paper is organized as follows. In Section 2, we review the background of the random Fourier method and present an IPM for computing $c^*$ in random flows. In Section 3, we provide the error estimate of the approximation to the solution operator and convergence rate. In Section 4, we validate IPM with a semi-Lagrangian method on 2D examples and carry out numerical experiments to show its accuracy in spatially periodic flows in agreement with known theoretical results. Meanwhile, experiments on the time complexity of the semi-Lagrangian method in a 3D periodic flow have been conducted. Then we investigate the influence of random perturbations on $c^*$ and show the distribution of particles with different magnitudes of perturbation. Most notably the randomly perturbed ABC flows whose vortical structures are increasingly smeared as the random perturbations are strengthened, and $c^*$ is reduced accordingly. We discuss a shift normalization of IPM on $\mathbb{R}^d$ to realize its convergence to an invariant measure. The concluding remarks are in section 5, with most of the proofs left in the Appendix.  

\section{Efficient IPM for KPP front speeds}\label{sec:LagrangianMethods}
\noindent We develop the Lagrangian IPM for computing $c^*$ in random flows. 

\subsection{Random fields realization}\label{sec:RanFieRealization}
\noindent
Several earlier works on calculating the KPP front speeds of random fields required certain compactness of the domain or regularity of the random fields. In \cite{Nolen2008}, an infinite channel with a finite size cross-section and a random shear flow imposed on the cross-section is considered. In \cite{NolenJack2009}, the author assumes the random velocity field is uniformly bounded in the spatial domain $\mathbb{R}^d$. Considering the compactness of the functional space, we would like to consider a kind of random velocity field with an equispaced discrete spectrum in this paper. For a random velocity field, $v(x,\omega)$, with the equispaced discrete spectrum, it can be represented as a sum of uncorrelated harmonic oscillations, i.e.  
\begin{equation}
    v(x,\omega) = \sum_{j = -\infty}^{\infty} e^{-2\pi i j \Delta k x} E^{1/2}(|j|\Delta k) \Delta\Tilde{W}_j,
    \label{DefofVdisspec}
\end{equation}
where $E(k)$ is a given spectral density, $\omega$ is an element of the probability space describing all possible environments, and the complex random variables $\Delta \tilde{W}_j$ are defined in terms of the complex white noise process: 
\begin{equation}
    \Delta \tilde{W}_j = \int_{j - \Delta k /2}^{j + \Delta k /2} d \tilde{W}(k),
\end{equation}
where $d\Tilde{W}(k)$ represents a complex Gaussian white noise. Of course, stationary random processes with a discrete spectrum are far from the only kind of stationary processes. However, it turns out that any stationary process can be obtained as the limit of a sequence of processes with discrete spectra \cite{Yaglom1987}. 

To approximate the random velocity field $v(x,\omega)$, we use random Fourier methods \cite{Majda:99} for producing a random velocity field $\tilde{v}(x,\omega)$. For ease of implementation, we would require $\tilde{v}(x,\omega)$ can be described by combinations between a finite number of random variables.  The random Fourier method can be verified by computing the average of mean-square displacement among $N_r$ independent realizations $\{\tilde{v}^j\}_{j=1}^{N_r}$ of the approximate velocity fields which are generated by successive calls to the random number generator (different random seeds). By Riemann sum approximation over a finite partition of $2N_F+1$ equal width intervals of width $\Delta k$, we define $k_{max} = (N_F + \frac{1}{2}) \Delta k$. This partition extends over a finite interval $[-k_{max}, k_{max}]$. We can get the following Riemann-Stieltjes sum for the approximating velocity field:
\begin{equation}
    \tilde{v}(x,\omega) = \sum_{j=-N_F}^{N_F} e^{-2\pi i j \Delta k x} E^{1/2}(|j|\Delta k) \Delta \tilde{W}_j.
    \label{vapp}
\end{equation}
where $\{\Delta \Tilde{W}_j\}_{j=1}^{N_F}$ are statistically independent complex Gaussian variables and $\Delta \Tilde{W}_0$ is a real Gaussian variable with mean zero and variance $\Delta k$ independent from all other variables. Therefore we can transform the equation (\ref{vapp}) into:
\begin{equation}
    \tilde{v}(x,\omega) = E^{1/2}(0) \Delta \Tilde{W}_0 + 2 \textbf{Re} \sum_{j=1}^{N_F} e^{-2\pi i j \Delta k x} E^{1/2}(|j|\Delta k) \Delta \tilde{W}_j ,
\end{equation}
which shares the same distribution. Here \textbf{Re} denotes the real part of the following expression. By expanding the complex random variable $\Delta \Tilde{W}_j$ into real and imaginary parts, we obtain a concise expression of the approximate velocity field as a discrete sum of real Fourier modes. The approximate velocity field is:
\begin{equation}
    \tilde{v}(x,\omega) = \sum_{j = 0}^{N_F} \sqrt{2E(k_j)\Delta k_j} [\zeta_j \cos(2\pi k_j x) + \eta_j \sin(2\pi k_j x)],
    \label{VfourGen}
\end{equation}
where $k_j = j\Delta k$ denote the locations of equally spaced grid points and $\Delta k_j = \Delta k$ for $j = 1,2,..., N_F$ and $\Delta k_0 = \frac{1}{2} \Delta k$. The $\{\zeta_j\}_{j=0}^{N_F}$ and $\{\eta_j\}_{j=0}^{N_F}$ are independent Gaussian variables with mean zero and variance one. 

The exact correlation function of random field $v(x,\omega)$ with given $E(k)$  is:
\begin{equation}
    R(x) = \langle v(x_0,\omega) v(x_0 + x,\omega) \rangle = 2 \sum_{j=0}^{\infty} E(k_j) \Delta k_j \cos(2\pi k_j x) .
    \label{TrueCorrFct}
\end{equation}

As $\Delta k \to 0$, \eqref{TrueCorrFct} will be an approximation to the integral $ \int_{0}^{\infty} \cos (2\pi kx) E(k) dk$.
Meanwhile, the correlation function of the Fourier generated random field $\tilde{v}(x,\omega)$ is 
\begin{align}
    \tilde{R}(x)  = & \sum_{j=0}^{N_F} (\sqrt{2E(k_j)\Delta k_j})^2 [\cos(2\pi k_j x_0) \cos(2\pi k_j (x + x_0))\ + \notag \\
    & \sin(2\pi k_j x_0) \sin(2\pi k_j (x + x_0))] = \sum_{j = 0}^{N_F} 2E(k_j)\Delta k_j \cos (2\pi k_j x).
    \label{AppCorrFct}
\end{align}

By comparing these two correlation functions (\ref{TrueCorrFct}) and (\ref{AppCorrFct}), we find that $\tilde{R}(x)$ is just the finite truncation of $R(x)$. As the number of Fourier modes $N_F$ gets larger, this approximation will become more accurate. In our numerical experiments, we select these two parameters according to the correlation length of energy spectral density. 

\subsection{Computing $\mu$ by FK semigroup}\label{sec:FKformula}
\noindent
Let us consider  $\bm v  
= \bm v(\bm x,\omega)$, 
a mean zero and divergence-free flow field,  
to be constructed by a sum of a deterministic space term with period $1$ and a stochastic term.
The stochastic term is a mean zero, stationary, and ergodic random field defined by \eqref{DefofVdisspec}. We assume that the realizations of $\bm v(\bm x,\omega)$ are almost surely divergence-free. To compute the KPP front speed $c^*(\bm z)$ along direction $\bm z$, we denote $q$ as the solution to the linearized corrector equation with $(\bm z, \bm e) > 0$:
\begin{equation}
    q_t = \mathcal{A}\, q \overset{\text{def}}{=} \kappa \Delta_{\bm x} q + (-2\kappa \lambda \bm e + \bm v) \cdot \nabla_{\bm x} q + (\kappa \lambda^2 - \lambda \bm v \cdot \bm e + f'(0))q  ,
    \label{linearCorrEqu}
\end{equation}
with compactly supported non-negative initial condition $q(\bm x, 0)\not \equiv 0$. Recall that the minimal front speed in $\bm z$ is: $c^* (\bm z) = \inf_{(\bm z, \lambda\bm e) > 0} \mu(\lambda \bm e) / (\bm z, \lambda\bm e)$, where  
\begin{equation}
    \mu(\lambda \bm e) = \lim_{t \to +\infty} \frac{1}{t} \ln \int_{\mathbb{R}^d} q(\bm x, t) d\bm x,
    \label{muComp}
\end{equation}  
also known as the principal Lyapunov exponent of  (\ref{linearCorrEqu}). We define $c(\bm x, \omega) = (\kappa \lambda^2 - \lambda \bm v(\bm x,\omega) \cdot \bm e + f'(0) )$.  To design the IPM, we split the operator $\mathcal{A}$ into $\mathcal{A} = \mathcal{L} + \mathcal{C}$, where $\mathcal{L}$ and $\mathcal{C}$ are defined by:
\begin{equation}
    \mathcal{L}q \overset{\text{def}}{=} \kappa \Delta_{\bm x}q + (-2\kappa \lambda \bm e + \bm v) \cdot \nabla_{\bm x}q,
    \label{DefL}
\end{equation}
\begin{equation}
    \mathcal{C}q \overset{\text{def}}{=}  (\kappa \lambda^2 - \lambda \bm v \cdot \bm e + f'(0) ) q.
    \label{DefC}
\end{equation}

Then for each random realization, to approximate the operator $\mathcal{L}$, we define the following SDE system:
\begin{equation}
    d X_s^{t_2, t_1, \bm x} = \bm b(X_s^{t_2, t_1, \bm x}) \, ds + \sqrt{2 \kappa}\, d\, \textbf{w}(s),\quad X_{t_1}^{t_2, t_1, \bm x} = \bm x, \quad t_2 \geq s \geq t_1,
    \label{SDEsys}
\end{equation}
where $\bm b = -2\kappa \lambda \bm e + \bm v$ is the drift term, 
$\textbf{w}(t)$ is a $d$-dimensional Brownian motion independent of $\omega$. The FK formula is:
\begin{equation}
    \mu(\lambda \bm e) = \lim_{t \to +\infty} \frac{1}{t} \ln \mathbb{E} \Big(\exp \big(\int_0^t c(X_{s}^{t,0,\bm x}) ds\big) \Big).
    \label{muCompFeyKac}
\end{equation}
Here the expectation is over the probability space induced by the $d$-dimensional Brownian motion $\textbf{w}(t)$. For a given $\bm e$ and its corresponding operator $\mathcal{A}$ in the following analysis, we denote $\mu(\lambda) = \mu(\lambda \bm e)$ for simplicity. Only during the computation of front speed $c^*$, we search $\bm e$ to minimize $\mu(\lambda \bm e) / (\bm z, \lambda \bm e)$, subject to $(\bm z, \bm e) > 0$.

Comparing the computation in formulas (\ref{muComp}) and (\ref{muCompFeyKac}), we see that the FK formula provides an alternative strategy to design (Lagrangian) particle methods to compute $\mu(\lambda)$ and overcome the drawbacks in the traditional Eulerian methods, especially in the cases when the magnitude of $\bm v$ is large or the dimensional $d$ is high (e.g. $d = 3$).  Directly applying the FK formula (\ref{muCompFeyKac}) is unstable as the sample paths visiting the maximal points of the potential $c$ make the main contribution to the expectation $\mathbb{E}$ in (\ref{muCompFeyKac}). This may lead to inaccurate or even divergent results. A more accurate strategy for the computation of $\mu(\lambda)$ is to study the convergence of the FK group associated with the above SDEs (\ref{SDEsys}) and the potential $c$. From \eqref{DefofVdisspec}, $\bm v(\bm x, \omega)$ will have a period $1/\Delta k$ in each spatial dimension if we set $1/\Delta k \in \mathbb{N}^+$. Specifically, let $\mathcal{D} = [0, \frac{1}{\Delta k}]^d$ denote the $d$-dimensional torus with the periodic boundary condition, $\mathcal{P}(\mathcal{D})$ denote the set of probability measures over $\mathcal{D}$ and $S = \mathcal{C}^{\infty} (\mathcal{D})$. Unless otherwise specified, the norms taken in the rest of this paper are all on the domain $\mathcal{D}$. Then let $Q_{\cdot,\cdot}$ denote the evolution operator which is associated with the process $X_{s}^{t_2, t_1, \bm x}$ in the SDE system i.e. for any measure $m \in P(\mathcal{D})$, function $\phi \in S$, we have 
\begin{equation}
    m(Q_{t_2, t_1}\phi) = \mathbb{E}_{\bm x \sim m} \big(\phi(X_{t_2}^{t_2, t_1, \bm x}) \big).
\end{equation}

Similarly, we define the associated weighted counterpart as
\begin{equation}
    m(Q_{t_2, t_1}^c \phi) = \mathbb{E}_{\bm x \sim m} \big(\phi(X_{t_2}^{t_2, t_1, \bm x}) \exp(\int_{t_1}^{t_2} c(X_s^{t_2, t_1, \bm x}) ds) \big).
    \label{WeightCouPart}
\end{equation}

The $\mathcal{L}(t_1)$ and $\mathcal{A}(t_1)$ are the infinitesimal generators of two evolution operator $Q_{t_2, t_1}$ and $Q_{t_2, t_1}^c$. By the definition of $Q_{t_2,t_1}^c$ and as in \cite{ferre2019error,cerou2011nonasymptotic}, we define the FK operator $\Phi_{t_2,t_1}^c$ by:
\begin{equation}
    \Phi_{t_2, t_1}^c(m)(\phi) = \frac{m(Q_{t_2,t_1}^c \phi)}{ m(Q_{t_2, t_1}^c 1)} = \frac{\mathbb{E}_{\bm x \sim m} \big(\phi(X_{t_2}^{t_2, t_1, \bm x}) \exp(\int_{t_1}^{t_2} c(X_s^{t_2, t_1, \bm x}) ds ) \big)}{\mathbb{E}_{\bm x \sim m} \big(\exp(\int_{t_1}^{t_2} c(X_s^{t_2, t_1, \bm x}) ds ) \big)}.
    \label{FeyKacSemiOperator}
\end{equation}

Due to steady velocity fields here, we denote $Q_{t,0}^c = Q_{t}^c$ and $\Phi_{t,0}^c = \Phi_{t}^c$. Let $m_c$ denote the invariant measure of $\Phi_{t}^c$ which satisfies $\Phi_{t}^c(m_c) = m_c$. It is straightforward to verify the following property of the FK semigroup $\Phi_{t}^c$ \cite{Junlong2022}.

\begin{pro}
For any $m \in \mathcal{P}(\mathcal{D})$ and $\phi \in S$, there exists $C_{\mathcal{A}} > 0$ such that 
\begin{equation}
\Big|\Phi^{c}_{t}(m)(\phi)-\int_{\mathcal{D}}\phi d m_{c}\Big|\leq C_{\mathcal{A}}||\phi||\exp(-\delta_{c}t), 
\end{equation}  
where $\delta_{c}=\inf\{\mu(\lambda)-\Re(z):z\in \sigma(\mathcal{A}) \setminus\{\mu(\lambda)\}\}>0$ is the spectral gap of the operator $\mathcal{A}$, $m_{c}$ is the invariant measure of $\Phi^{c}_{t}$.
\end{pro}

The exponential decay property of the FK operator $\Phi_{t}^c$ in the above proposition allows us to establish the existence of invariant measure $m_c$ of $\Phi_{t}^c$ from a simple initial measure $m$ (uniform or Gaussian) after long enough evolution. Since $m_c$ is the invariant measure of $\Phi_{t}^c$, for any $\phi \in S$, we have
\begin{equation}
    \int_{\mathcal{D}} \phi dm_c = (\int_{\mathcal{D}} Q^c_t 1 dm_c)^{-1} \int_{\mathcal{D}} Q^c_t \phi dm_c.
\end{equation}
This provides us a practical approach to compute the principal eigenvalue of $Q^c_t$ by evaluating $\int_{\mathcal{D}} Q^c_t 1 d m_c$.

\subsection{Numerical discretization and resampling techniques}\label{sec:NumApprResampling}
\noindent
For the time interval $[0, M T]$, let $M$ denote the iteration number, $H$ denote the number of time discretization intervals in each iteration, and let $\Delta t = T / H$. At time $t_i = i \Delta t$ in one iteration, we can define the transform of random variables by the following formula:
\begin{equation}
    {\bf Y}_i = {\bf X}_i + 
\textbf{b}({\bf X}_i)\Delta t+  \sqrt{2\kappa \Delta t} \bm\omega_i, 
\end{equation}
where $\{\omega_i\}$ are i.i.d d-dimensional standard Gaussian random variables which are independent of $X_i$. This transform represents one step Euler-Maruyama discretization scheme for process ${\bf X}_s^{T,0,\bm x}$ at time $s = T - i\Delta t$ which follows the SDE system (\ref{SDEsys}). Then the one-step evolution operator 
$Q^{\Delta t}_i$ will be defined as follows:
\begin{equation}
    Q^{\Delta t}_i \phi(\bm x) = \mathbb{E}(\phi({\bf Y_i})| {{\bf X}_i = \bm x}),\quad \phi \in S .
\end{equation}

The evolution operator $Q^{\Delta t}_i$ describes how the values of a given function evolve in $L_2$ sense over one time step $\Delta t$. One can easily verify that  
\begin{equation} \label{MarkovProcessApproximation}
\big|\big|Q^{\Delta t}_{i} - e^{\Delta t \mathcal{L}}\big|\big|_{L^2} \le C_Q (\Delta t)^2,
\end{equation}
where the operator $\mathcal{L}$ refers to \eqref{DefL} and $C_Q$ is a positive constant \cite{MilsteinTransition}. Specially, when $\bm b = 0$,  $Q^{\Delta t}_{t_i} = e^{\Delta t \mathcal{L}}$ for all $i$. 
In addition, we can define the approximation operator for $Q^c_t$ in \eqref{WeightCouPart}. For instance, if we choose the right-point rectangular rule, we obtain  that for any $m \in \mathcal{P}(\mathcal{D})$ and $\phi \in S$
\begin{equation}
m(Q_i^{\Delta t} e^{\Delta t \mathcal{C}}\phi) = \mathbb{E}\Big(\phi({\bf Y}_{i})\exp\big(
c({\bf Y}_{i})\Delta t\big)\big|{\bf X}_{i} \sim m \Big), \quad i=1,2,...,H.
\label{DiscreteOperatorPct}
\end{equation}

The time discretization for the FK semigroup \eqref{FeyKacSemiOperator} reads:
\begin{equation}
\Phi_{i}^{\mathcal{C},\Delta t} (m)(\phi) = \frac{ m(Q_i^{\Delta t} e^{\Delta t \mathcal{C}}\phi)}{m(Q_i^{\Delta t} e^{\Delta t \mathcal{C}}1)},
\quad i=1,2,...,H.
\label{DiscreteFeynmanKacsemigroup} 
\end{equation}

It is not practical to compute the closed-form solution of the evolution of the probability measure in \eqref{DiscreteFeynmanKacsemigroup}. So we try to approximate the evolution of the probability measure in another way by an $N$-interacting particle system ($N$-IPS) \cite{del2000branching}. Now we denote $\mathcal{K}^{\Delta t,i} = Q_i^{\Delta t} e^{\Delta t \mathcal{C}(t_i)}$ and define $\mathcal{K}^{\Delta t} = \mathcal{K}^{\Delta t,H} \mathcal{K}^{\Delta t,H-1}\cdots \mathcal{K}^{\Delta t,1}$.
\begin{align}
\Phi^{\mathcal{K}^{\Delta t,i}} (m)(\phi) = \frac{ m(\mathcal{K}^{\Delta t,i}\phi)}{m(\mathcal{K}^{\Delta t,i}1)},
\quad i=1,2,...,H, 
\label{operatorPhiKiDeltat}
\end{align}
is the FK semigroup associated with the operator $\mathcal{K}^{\Delta t,i}$. 

\begin{pro} [Lemma 3.6 in \cite{Junlong2022}] For any operators $\mathcal{A}$, $\mathcal{B}$ in $\mathcal{L}(L^2(\mathcal{D}))$, $\Phi^{\mathcal{A}\mathcal{B}} = \Phi^{\mathcal{B}} \Phi^{\mathcal{A}}$. 
\label{Prop2.2}
\end{pro}

According to Proposition \ref{Prop2.2}, we have that  
\begin{equation}
		\Phi^{\mathcal{K}^{\Delta t}} = \prod_{i=0}^{H-1} \Phi^{\mathcal{K}^{\Delta t,H-i}}=\Phi^{\mathcal{K}^{\Delta t,1}}\Phi^{\mathcal{K}^{\Delta t,2}}\cdots\Phi^{\mathcal{K}^{\Delta t,H}}.
		\label{operatorPhiKDeltat}
\end{equation}

Suppose a Markov process $({\Theta}, (\mathcal{F}_M)_{M\ge 0}, (\bm\xi^M)_{M\ge 0}, \mathbb{P})$ 
is defined in the product space $(\mathcal{D})^N$. By an $N$-IPS, we can approximate any initial probability measure $m_0 = m$ as:
\begin{equation}
	P(\bm\xi^0 \in d\bm z) = \prod_{l=1}^N m_0(dz^l).
	\label{initialNparticle}
\end{equation}

Then according to \cite{del2000branching} and denote $\bm \xi^{M-1} = \bm x$, we evolve the $N$-IPS according to 
\begin{equation}\label{evolveNparticle}
		P(\bm\xi^M \in d\bm z|\bm\xi^{M-1}) = \prod_{l=1}^N \Phi^{\mathcal{K}^{\Delta t}}(\frac{\sum_{i=1}^{N}\delta_{x^i}}{N} )(dz^l)=\prod_{l=1}^N (\prod_{j=0}^{H-1} \Phi^{\mathcal{K}^{\Delta t,H-j}})(\frac{\sum_{i=1}^{N}\delta_{x^i}}{N}) (dz^l),
\end{equation}
where $N$ denotes the number of particles, vector $\bm x=(x^1,...,x^N)^T$, and $M$ denotes the iteration number of the probability measure's evolution by the FK semigroup defined in \eqref{DiscreteFeynmanKacsemigroup}. Following equation \eqref{evolveNparticle}, we compute each iteration of the evolution of $N$-IPS step by step from $\bm \xi^0$. And each iteration step will be divided into $H$ small steps and we denote $\bm \xi_i^M = (\xi_i^{1,M},...,\xi_i^{N,M})$ as $\bm \xi^M$ after $i$th small steps ($i = 0,1,...,H$). During each iteration stage, we evolve the particles from $s=0$ to $s=T$ in the SDE \eqref{SDEsys} by the evolution operator $\{Q^{\Delta t}_{i}\}_{i=0}^{H-1}$ and then re-sample these particles according to their weights which are determined by the potential function. Specifically, each particle in the system is updated by:
\begin{equation}
	\widetilde{\xi}_{i}^{l,M-1} = \xi_{i}^{l,M-1} + \bm b(\xi_{i}^{l,M-1})\Delta t+  \sqrt{2\kappa \Delta t} \bm\omega_i^{l,M-1}, \quad l = 1,2,...,N, \label{NMSDEevolveparticles}
\end{equation}
where $\{\bm\omega_i^{l,M-1}\}$ are i.i.d standard $d$-dimensional Gaussian random variables. Then by following the FK semigroup, we compute each particle's weight $w_i^{l, M-1}$ by \eqref{MultinomialResampleWeight} and re-sample them according to the multinomial distribution of their weights and then obtain $\bm\xi_{i+1}^{M-1}$.
\begin{equation}
	{w}_{i}^{l,M-1}=\frac{\exp{\big(c(\widetilde\xi_{i}^{l,M-1})\Delta t\big)}}{\sum_{l=1}^{N}\exp{\big(c(\widetilde\xi_{i}^{l,M-1})\Delta t\big)}}, \quad  l=1,2,...,N.
	\label{MultinomialResampleWeight}
\end{equation} 	 

The whole evolution process of $N$-IPS from $0$ to $MT$ can be seen from Algorithm \ref{Algo::ForTimeindependent}. The restriction step in this algorithm is the positional restriction back for the particles that are about to escape from the range of the domain during the evolution process. For example, under the periodic velocity field condition in $\mathcal{D} = [0, \frac{1}{\Delta k}]^d$, we update the particle position in each dimension by $x \leftarrow mod(x, \frac{1}{\Delta k})$.

After obtaining the particles', $\bm\xi_{0}^{M}$, empirical distribution, we  compute $\mu^M_{\Delta t}(\lambda)$. At the $M$-th iteration stage, we first define the change of the mass as follows
\begin{equation}
	e_{i,M}^N = N^{-1}\sum_{l=1}^{N} \exp(c(\widetilde\xi_{i}^{l,M})\Delta t).
\end{equation}

Then, we compute the $\mu$ approximation:
\begin{equation}
	\mu^M_{\Delta t}(\lambda) = (H\Delta t)^{-1}\sum_{i=0}^{H-1} \log\Big( N^{-1}\sum_{l=1}^{N} \exp(c(\widetilde\xi_{i}^{l,M})\Delta t)\Big). 
	\label{NIPS-PrincipalEig}
\end{equation} 

We know that the particles', $\bm\xi_{0}^{M}$, empirical distribution  will weakly converge to the distribution $\Phi_n^{\mathcal{K}^{\Delta t}} (m_0)$ as $N$ approaches infinity. Therefore, we can use $\mu^M_{\Delta t}(\lambda)$ as an approximation to $\mu(\lambda)$. Equation \eqref{NIPS-PrincipalEig} is the discrete version of the following formula
\begin{equation}
    \mu^{M} (\lambda) = \frac{1}{T} \int_{(M-1)T}^{MT} \ln \mathbb{E} [\exp(c(X_s^{MT,0,\bm x}))] ds.
\end{equation}

Compared to using the FK formula \eqref{muCompFeyKac} directly, the first thing we need to realize is that the limits of \eqref{muCompFeyKac} and \eqref{NIPS-PrincipalEig} are equivalent due to the ergodicity. In \eqref{NIPS-PrincipalEig}, only the states during the final iteration, i.e., $M$-th iteration, are taken into account in the computation of $\mu^M_{\Delta t}$. In contrast, directly applying the FK formula \eqref{muCompFeyKac} requires accounting for the whole evolution process, i.e. from $t = 0$ to $t = MT$, which is numerically unstable in the long-time computation. 
The following algorithm \ref{Algo::ForTimeindependent} shows the entire evolution process. 

\begin{algorithm}[h]
\caption{Genetic algorithm for computing 
$\mu_{\Delta t}^{n}(\lambda)$}
\label{Algo::ForTimeindependent}
\begin{algorithmic}[1]  
\REQUIRE Input:
velocity field $\bm v(\bm x, \omega)$, potential $c(\bm x, \omega)$, initial probability measure $m_0$, Monte Carlo number $N$, iteration number $M$, time step $\Delta t$, generation life span $T = H\Delta t$ and $t_i = i\Delta t,\ 0\leq i \leq H$ (mutation number). 

\STATE Generate $N$ i.i.d. $m_0$-distributed random variables $\bm \xi_0^0 = (\xi_0^{1,0}, \cdots, \xi_0^{N,0})$ on $\mathcal{D}$ 
\FOR{$j = 1:M$}
\FOR{$i = 0:H-1$}
\STATE Generate $N$ i.i.d. standard Gaussian random variables. 

Compute $ \tilde{\bm \xi}_i^{j-1} = (\tilde\xi_i^{1, j-1} ,\cdots, \tilde\xi_i^{N,j-1})$ from $\bm \xi_i^{j-1}$ by \eqref{NMSDEevolveparticles}. 

Restrict $\bm \xi_i^{j-1}$ on $\mathcal{D}$ by imposing boundary condition (not needed on $\mathbb{R}^d$).
\STATE Compute point-wise fitness array $\bm S \overset{\text{def}}{=} (\exp(c^1 \Delta t)$ $,\cdots,  \exp(c^N\Delta t) )$, where $c^k = c(\tilde\xi_i^{k,j-1} , \omega),\ 1\leq k \leq N$. 
\STATE Define weight array $\bm w = (w^{1,j-1}_i, \cdots, w^{N,j-1}_i)\overset{\text{def}}{=} \bm S / sum(\bm S)$ via fitness, and compute mutation-wise mean particle fitness growth rate (PFGR): 

$\bm E_{j,i} \overset{\text{def}}{=} \log({\text mean}(\bm S)) / \Delta t$
\STATE Resample $\tilde{\bm \xi}_i^{j-1}$ by ${\bm w}$-parameterized  multinomial distribution to get $\bm \xi_{i+1}^{j-1}$.
\ENDFOR
\STATE Generation-wise mean PFGR: $\mu_{\Delta t}^j (\lambda) \overset{\text{def}}{=} H^{-1} \sum_{i=0}^{H-1} (\bm E_{j,i})$ and set $\bm \xi_{0}^{j} = \bm \xi_{H}^{j-1}$.
\ENDFOR
\ENSURE Output $\mu_{\Delta t}^M (\lambda)$ and approximately time-invariantly distributed $\bm \xi_0^M$.
\end{algorithmic}

\end{algorithm}

\section{Convergence analysis}\label{sec:ConvergenceAnalysis}
\noindent
In this section, we demonstrate the convergence of the IPM in computing the KPP front speed of the stationary and isotropic Gaussian random fields with discrete spectrum defined as \eqref{DefofVdisspec}. If we consider the system in $\mathbb{R}^d$ directly, our particle approach is easy to implement for computation. However, due to the lack of spectral gap conditions and bounded settings, convergence analysis is difficult to carry out directly. Since the random velocity field $\bm v(\bm x, \omega)$ in this setting always has period $1/\Delta k$, the motion of the particle can be considered in the torus $\mathcal{D} = [0, 1/\Delta k]^d$.  In \cite{Junlong2022}, the authors have given the convergence analysis of the cases with deterministic velocity fields. Due to each random realization, the realized velocity field will perform as a deterministic case, so we divide the analysis into four parts. The first part studies the approximation of the random field with a discrete spectrum by the random Fourier method. The second part studies the error estimate of an operator splitting method to approximate the evolution of parabolic operators. The third part studies the error estimate of the IPM in $\mu$ computation.
The fourth part studies the error estimate in computing $\mu$ via the random Fourier method. 

\subsection{Approximating Gaussian random field with a prescribed discrete spectrum by random Fourier method} \label{sec:AppofRanFieldusingRFM}
\noindent
In this subsection, we first represent the linearized corrector KPP equation \eqref{linearCorrEqu} into a non-autonomous parabolic equation in each iteration:
\begin{equation}
    q_t = \kappa\Delta_{\bm x} q + \bm b(\bm x,\omega)\cdot\nabla_{\bm x} q + c(\bm x,\omega)q, \quad \bm x=(x_1,...,x_d)^{T}\in \mathcal{D}, \quad t \in[0,T],
    \label{nonAutoParaEq}
\end{equation}
 with the initial condition $q(\bm x, 0) = q_0$, $\bm b(\bm x,\omega)=-2\kappa\lambda\bm e+\bm v$, $c(\bm x,\omega)=\kappa \lambda^2-\lambda\bm v\cdot \bm e + f^{\prime}(0)$.
 For Gaussian random field $\bm v(\bm x, \omega)$ defined by \eqref{DefofVdisspec}, it is fully characterized by its mean and covariance function.
 For each realization of the random field, our strategy is to generate an approximation $\tilde{\bm v}(\bm x)$ of the random velocity field, $\bm v(\bm x,\omega)$, with long-range correlation according to the correlation functions. Generated by the random Fourier method, $\tilde{\bm v}$ can be used with confidence to predict some important scaling laws and can simulate the spreading of particles accurately and efficiently. 
 We use the simplest tracer statistic of the sort to measure the spreading of the particles. 
In the SDE system \eqref{SDEsys}
\begin{equation}
    d \bm X_s = \bm b(\bm X_s,\omega) ds + \sqrt{2\kappa} d\textbf{w} (s),\quad \bm X_0 = \bm x',
\end{equation}
where $\textbf{w}(s)$ is $d$-dimensional standard Brownian motion and its variance at time $s$ in each dimensional is $2\kappa s$. The stochastic drift term $\bm b(\bm X_s,\omega)$ is independent of $\textbf{w}(s)$, so their variances are additional. 
We denote the deterministic part of $\bm b(\bm X_s,\omega)$ as $\bm b_{0}(\bm X_s)$.
Since $\bm v$ is mean zero and stationary, we have $\langle \bm b(\bm X_s, \omega)\rangle = \bm b_{0}(\bm X_s)$ and the covariance function between two points $\bm X_s, \bm X_s + \bm x$ can be computed by: 
\begin{align}
    \langle (\bm b(\bm X_s,\omega) -\bm b_{0}(\bm X_s))\cdot(\bm b(\bm X_s + \bm x, \omega)-\bm b_{0}(\bm X_s + \bm x))\rangle 
    &= \langle \bm v (\bm X_s,\omega)\cdot \bm v(\bm X_s + \bm x,\omega)\rangle \notag\\ 
    &= R(\bm x).
\end{align}

 Here $\langle\cdot\rangle$ denotes the average over probability space induced by $\omega$ and $R(\bm x)$ denotes the two-point correlation function of $\bm v$. Since $\bm v$ is isotropic, the correlation function $R(\bm x)$ admits the following representations. 
\begin{align}
    &R(\bm x) = R(x) =  \left \langle \bm v(x_0,\omega) \cdot \bm v(x_0 + x,\omega) \right \rangle = 2\sum_{j=0}^{+ \infty} \Delta k \cos(2\pi k_j r ) E(k_j) ,\ (\text{1-D}) \\
    &R(\bm x) = \left \langle \bm v(\bm x_0,\omega) \cdot \bm v(\bm x_0 +\bm x,\omega) \right \rangle = 2\sum_{j = 0}^{+ \infty} \Delta k J_0(2\pi k_j r) E(k_j) ,\ (\text{2-D}) \\
    &R(\bm x) = \left \langle \bm v(\bm x_0,\omega) \cdot \bm v(\bm x_0 +\bm x,\omega) \right \rangle = 2\sum_{j = 0}^{+ \infty} \Delta k \frac{\sin(2\pi k_j r)}{2\pi k_j r} E(k_j) ,\ (\text{3-D}) 
\end{align}
where $E(k)$ is a given energy spectral density, $k_j = j \Delta k$, $J_0(\cdot)$ is the Bessel function of the first kind, and $r = ||\bm x||_2$. We know that the mean square displacement is also completely determined by the correlation function $R(\bm x)$. So the important thing is that the generated velocity field's correlation function $\tilde{R}(\bm x)$ should be an approximation to the true random velocity field's correlation function $R(\bm x)$.
Following the random Fourier method mentioned in Section \ref{sec:RanFieRealization}, to make it accurate at low frequency, we would require $E(k)$ to satisfy the following properties. 
\begin{enumerate}
    \item[P1.] $E(k)$ is integrable and has finite moments. Since $E(k)$ is non-negative, this also implies $E(k) \in L^1(\mathbb{R})$.
    \item[P2.] $E(k)$ is finite at zero point, i.e. $E(0) < +\infty$.
    \item[P3.] $E(k), E'(k), E''(k)$ are all bounded a.e.
    \item[P4.] Let $g(y) = 2\sum_{j = y}^{+\infty} E(k_j) \Delta k$. For $\forall C_g > 0$, $\exists y_0 \in \mathbb{N}^+$, $\forall y > y_0\ \text{and}\ y \in \mathbb{N}^+, g(y) < C_g/ y^2$, i.e. $g(y)$ decay to $0$ at least as $O(y^{-2})$ as $y$ goes to infinity.   
\end{enumerate}

From properties $P1$ to $P4$, we can also imply that the correlation function $R(\bm x)$ is always finite. Then by comparing these two correlation functions (\ref{TrueCorrFct}) and (\ref{AppCorrFct}), the error between the two velocity fields' correlation function is the truncation error. In other words, the error is related to the number of Fourier mode $N_F$ (or the length of truncated frequency interval $k_{max} = (N_F+\frac{1}{2}) \Delta k$).
 For a specific $N_F$ and $\Delta k$, the truncation error $e_{R}$ in 1-D cases can be represented by:
 \begin{align}
     e_{R}(N_F) &= \big|2 \sum_{j = 0}^{\infty} \cos (2\pi k_j x) E(k_j) \Delta k - 2 \sum_{j =0}^{N_F} \cos (2\pi k_j x) E(k_j) \Delta k\big| \notag \\
     &= 2 \sum_{j = N_F + 1}^{\infty} \big|\cos (2\pi k_j x) E(k_j) \Delta k\big| < g(N_F + 1).
     \label{ErrEstRx}
 \end{align} 



\begin{lemma}
    The error, $e_R$, between the corresponding correlation functions of two velocity fields, $\bm v$ and $\tilde{\bm v}$, decays to zero at a convergence rate of $O((N_F)^{-2})$, i.e., for any $C_g > 0$,  $\exists N_{F_0} \in N^+, \forall N_F > N_{F_0}$ and $N_F \in N^+$, s.t.
    \begin{equation}
        e_R (N_F) 
        \leq \frac{C_g}{(N_F)^2}.
    \end{equation}
    
\end{lemma}
\begin{proof}
    For the truncation error, $e_{R}$, will be related to the performance of $E(k)$ at large frequency. From the properties $P1$ to $P4$ we mentioned above, for any $C_g > 0$, with sufficient large $N_F$, we have
    \begin{equation}
        e_{R}(N_F) \leq g(N_F + 1) < \frac{C_g}{(N_F + 1)^2} < \frac{C_g}{(N_F)^2}. 
    \end{equation}
    
\end{proof} 

For 2-D and 3-D cases, we can follow the same idea and get similar estimates.
In \cite{Junlong2022}, the author assumed that the deterministic part $\bm b_{0}(\bm x)$ and $c_{0}(\bm x)$ satisfies $||\bm b_{0}(\bm x)||_{L^2} \leq M_1$, $|c_{0}(\bm x)| \leq M_2$, $||\nabla_{\bm x} c_{0}(\bm x)||_{L^2} \leq M_3$ and $||\Delta_{\bm x} c_{0}(\bm x)||_{L^2} \leq M_4$. Now we analyze these assumptions for the random velocity field $\bm v$ in a torus $\mathcal{D}$, using $\tilde{\bm v}$ as an approximation of $\bm v$ in practical experiments.
\begin{lemma}
    For a probability $p$ close to 1, the term $||\bm b(\bm x, \omega)||_{L^2}$ corresponding to the random velocity field mentioned above has its upper bound $\Tilde{M}_{1}$, so that
    \begin{equation}
        P\big(||\bm b(\bm x,\omega)||_{L^2} < \Tilde{M}_{1}\big) > p.  
    \end{equation}
    
    The terms, $|c(\bm x)|, ||\nabla_{\bm x} c(\bm x)||_{L^2}, ||\Delta_{\bm x} c(\bm x)||_{L^2}$, can also find their corrsponding upper bounds $\tilde{M}_{2},\tilde{M}_{3},\tilde{M}_{4}$.
    \label{Lemma3_2}
\end{lemma}
The Plancherel theorem and Markov inequality have been used in our proof. The detailed proof can be found in Appendix \ref{sec::AppendixUpperBound}.

\subsection{Approximating parabolic evolution by operator splitting}\label{sec:ErrorEstimateOperatorSplit}
\noindent
We recall the non-autonomous parabolic equation \eqref{nonAutoParaEq} and
use $\tilde{\bm v}$ to be an approximation to $\bm v$. From the notation defined in \eqref{DefL} and \eqref{DefC}, we can define 
 \begin{equation}
     \mathcal{A}(\bm x, \omega) = \mathcal{L}(\bm x,\omega) + \mathcal{C}(\bm x,\omega),
     \label{defAt}
 \end{equation}
 where $\mathcal{L}(\bm x,\omega) = \kappa \Delta_{\bm x} + \bm b(\bm x,\omega)\cdot\nabla_{\bm x}$ and $\mathcal{C}(\bm x,\omega) = c(\bm x,\omega)$. Similarly, we can define $\tilde{\mathcal{A}}, \tilde{\mathcal{L}}$ and $\tilde{\mathcal{C}}$ to be the corresponding operators for $\tilde{\bm v}$. We know that in each realization, the velocity field $\tilde{\bm v}$ is realized as a time-independent and space-periodic field. Then the operator $\tilde{\mathcal{A}}(\bm x,\omega_i) = \tilde{\mathcal{A}}(\bm x)$ has a real isolated principle eigenvalue \cite{hess1991periodic}. Here we treat $\mathcal{A}(\bm x, \omega)$ as the deterministic part $\tilde{\mathcal{A}}$ with a small random perturbation (caused by the remainder of the velocity field $\bm v - \tilde{\bm v}$). Following the perturbation theory in \cite{kato2013perturbation}, for the spectrum $\Sigma(\tilde{\mathcal{A}})$ of the deterministic infinite system $\tilde{\mathcal{A}}$ mentioned above, using the spectral gap condition, we can first separate it into $\Sigma'(\tilde{\mathcal{A}})$ and $\Sigma''(\tilde{\mathcal{A}})$ where $\Sigma'(\tilde{\mathcal{A}})$ is a finite system.

\begin{pro}
     The change in a finite system $\Sigma'(\tilde{\mathcal{A}})$ of the eigenvalues of a closed operator $\tilde{\mathcal{A}}$ is small (in the sense of \ref{PerbCtsEigTx}) when $\tilde{\mathcal{A}}$ is subjected to a small perturbation in the sense of generalized convergence. 
    \label{EigExistforrandomperturbation1}
\end{pro}

Proposition \ref{EigExistforrandomperturbation1} guarantees the existence of $\mu(\lambda)$ of $\mathcal{A}$ when it is treated as the original deterministic operator $\tilde{\mathcal{A}}$ subject to a relatively small random perturbation in each realization. The detailed explanation of Proposition \ref{EigExistforrandomperturbation1} can be found in Appendix \ref{Sec:AppendixE3proof}. Then we define a solution operator $\mathcal{U}(t,q): \mathcal{U}(t,s) q(\cdot, s) = q(\cdot, t)$ corresponding to the non-autonomous parabolic equation \eqref{nonAutoParaEq}. This solution operator $\mathcal{U}(t,s)$ maps the solution of equation \eqref{nonAutoParaEq} at time point $s$ to the solution at time point $t$. It is easy to verify:
\begin{enumerate}
	\item $\mathcal{U}(s,s) = Id$, $\forall\ s \geq 0$;
	\item $\mathcal{U}(t,r)\circ \mathcal{U}(r,s) = \mathcal{U}(t,s)$,  $ \forall\ t \ge r \ge s \ge 0$;
	\item $\frac{d}{dt} \mathcal{U}(t,s) q_0 = \mathcal{A}(\bm x,\omega)\ \mathcal{U}(t,s) q_0$, $\forall\ t\ge s\ge 0, q_0 \in L^2(\mathcal{D})$.
\end{enumerate} 

The solution operator provides us an approach to study the parabolic equation \eqref{nonAutoParaEq}. The principal eigenvalue of the parabolic operator $\mathcal{A}$ can be given by the principal eigenvalue of the solution operator $\mathcal{U}(T,0)$. Although the principal eigenvalue of $\mathcal{U}(T,0)$ exists and is real \cite{hess1991periodic}, it is difficult to have it in closed form for this operator. Therefore, we adopt an operator splitting method to approximate $\mathcal{U}(T,0)$.

We divide the time interval into $H$ equal sub-intervals and set $\Delta t = \frac{T}{H}$, $t_i = i \Delta t$, and consider the following parabolic equation during a time sub-interval
\begin{equation}
    q_t = \kappa\Delta_{\bm x}  q + \bm b(\bm x,\omega)\cdot\nabla_{\bm x} q + c(\bm x,\omega)q, \quad t_i< t \le t_{i+1}, \quad i\geq 0.
    \label{freetime1}
\end{equation}

Then we can represent the corresponding solution operator as:
\begin{equation}
q(t) = e^{(t-t_i)(\mathcal{L}+ \mathcal{C})}\prod_{h = 0}^{i-1}e^{\Delta t_h(\mathcal{L}+ \mathcal{C})}q_0,\quad t_i \le t < t_{i+1}. \label{SoluOperatorfreetime1}
\end{equation}
Here $\Delta t_h = \Delta t = \frac{T}{H}$, just to emphasize that it corresponds to the $h$-th time sub-interval. Furthermore, after applying the first-order Lie-Trotter operator splitting method to 
get an approximation to the solution operator defined in \eqref{SoluOperatorfreetime1}, we obtain 
\begin{equation}
q(t) = e^{(t-t_i)\mathcal{L}}e^{(t-t_i) \mathcal{C}}\prod_{h = 0}^{i-1}e^{\Delta t_h\mathcal{L}}e^{\Delta t_h \mathcal{C}}q_0, \quad t_i \le t < t_{i+1}. 
\label{LieTrotterAPPSoluOperator}
\end{equation} 

It has been proved that the solution operator $\prod_{h = 0}^{H-1}e^{\Delta t_h\mathcal{L}}e^{\Delta t_h \mathcal{C}}$ obtained by the first-order Lie-Trotter splitting method converges to the corresponding solution operator $\mathcal{U}(T,0)$ as $\Delta t \to 0$ in certain operator norm (Theorem 3.4 in \cite{Junlong2022}). If $\bm b(\bm x,\omega)$ and $c(\bm x, \omega)$ are uniformly Lipschitz in $\bm x$, then the authors have also provided the error estimate, which shows that the error of the principal eigenvalue obtained by the Lie-Trotter operator splitting method is at least $O((\Delta t)^{\frac{1}{2}})$.

\subsection{Error estimate of particle method}\label{sec:AnalysisLagrangianMethod}
\noindent
For operator $\mathcal{A}$ defined in \eqref{defAt}, we consider its corresponding FK operator $\Phi^{\mathcal{A}}$ with a probability measure $m$ is defined by:
\begin{equation}
 \Phi^{\mathcal{A}} (m)(\phi) = \frac{ m(\mathcal{A}\phi)}{m(\mathcal{A}1)}, \quad
\forall \phi\in L^2(\mathcal{D}). \label{DefineFKsemigroup}   
\end{equation}
Moreover, we denote $\Phi^{\mathcal{A}}_M = (\Phi^{\mathcal{A}})^M$. It is easy to verify that the FK semigroup $\Phi^{\cdot}$ satisfies the property:
$\Phi^{\mathcal{A}\mathcal{B}} = \Phi^{\mathcal{B}} \Phi^{\mathcal{A}},\ \forall\ \mathcal{A},\mathcal{B} \in \mathcal{L}(L^2(\mathcal{D}))$.

Recall that the operator $\Phi^{\mathcal{K}^{\Delta t}}$ defined in \eqref{operatorPhiKDeltat} is constructed from the FK semigroup $\Phi^{\mathcal{K}^{\Delta t, i}}$ defined in \eqref{operatorPhiKiDeltat} which corresponding to the operator $\mathcal{K}^{\Delta t, i}$. We show several important results to ensure the existence of an invariant measure for the discretized FK dynamics and give an error estimate for computing the principal eigenvalue of the parabolic operator $\tilde{\mathcal{A}}(\bm x)$ defined as \eqref{defAt}. 

\begin{pro}[Theorem 3.7 in \cite{Junlong2022}]\label{bam}
There exists a probability measure $m^{\mathcal{K}}$ so that the operator $\mathcal{K}^{\Delta t}$ satisfies a uniform minorization and boundedness condition as follows

\begin{equation} 
\theta m^{\mathcal{K}}(\phi) \le \mathcal{K}^{\Delta t} (\phi) (\bm x) \le \vartheta m^{\mathcal{K}}(\phi), 
\quad \forall\ \bm x \in \mathcal{D}, \quad \forall\ \phi\in L^2(\mathcal{D}),
\label{MinorizationBoundedCondition}
\end{equation}
where $0<\theta<\vartheta $ are independent of $\Delta t$ and $\phi$.
Moreover, when $\Delta t \to 0$, the limit operator is the exact solution operator $\mathcal{U}(T,0)$, which also satisfies this uniform minorization and boundedness condition.
\end{pro}

\begin{pro}[Theorem 3.8 in \cite{Junlong2022}]\label{thm:existenceinvariantmeasure}
Suppose the minorization and boundedness conditions \eqref{MinorizationBoundedCondition} hold. Then,  $\Phi_M^{\mathcal{K}^{\Delta t}}$ admits an invariant measure $m_{\Delta t}$, whose density function is the eigenfunction of the operator $(\mathcal{K}^{\Delta t})^\star$, the adjoint operator of the solution operator $\mathcal{K}^{\Delta t}$. Moreover, for any initial distribution  $m_0\in \mathcal{P}(\mathcal{D})$, we have 
\begin{equation}
\big|\big|\Phi_M^{\mathcal{K}^{\Delta t}} (m_0) - m_{\Delta t}\big|\big|_{TV} \le 2(1-\frac{\theta}{\vartheta} )^M,
\label{CovInvariantMeasure}
\end{equation}
where $||\cdot||_{TV}$ is the total variation norm and $0<\theta<\vartheta $ are the parameters defined in the minorization and boundedness conditions in \eqref{MinorizationBoundedCondition}. The estimate \eqref{CovInvariantMeasure} is also true when changing $\mathcal{K}^{\Delta t}$ to the exact solution operator $\mathcal{U}(T,0)$.
\end{pro}

Proposition \ref{bam} and Proposition \ref{thm:existenceinvariantmeasure} the operator $\mathcal{K}^{\Delta t}$ satisfies a uniform minorization and boundedness condition and then can guarantee that $\Phi^{\mathcal{K}^{\Delta t}}_M$ admits an invariant measure. The details of the proof can be found in \cite{Junlong2022} and \cite{Meyn:92}. Then the principal eigenvalue of operator $\mathcal{K}^{\Delta t}$ satisfies the following relation:
\begin{equation}
    e^{\mu_{\Delta t}(\lambda)T} = m_{\Delta t} \mathcal{K}^{\Delta t} 1 = \Phi_M^{\mathcal{K}^{\Delta t}} (m_0) \mathcal{K}^{\Delta t} 1 + \rho_M,
    \label{PrinEigRel}
\end{equation}
where $m_{\Delta t}$ and $m_0$ are just the probability measure mentioned in Proposition \ref{thm:existenceinvariantmeasure}, $T$ is the period in time and $\rho_M$ is $O(1-\frac{\theta}{\vartheta})^M$.
Denote $m_{\Delta t}^h = \prod_{i=0}^{h-1} \Phi^{\mathcal{K}^{\Delta t,H-i}} m_{\Delta t}$, $1\leq h \leq H$. Let $e_h = (m_{\Delta t}^h) (\mathcal{K}^{\Delta t,H-h} 1)$ denotes the changing of mass. Then, we have 
\begin{equation} 
e^{\mu_{\Delta t}(\lambda)T} = \prod_{h=1}^{H } e_h, \quad \text{and} \quad  \mu_{\Delta t}(\lambda) = \frac{1}{H\Delta t} \sum_{h=0}^{H-1} \log(e_h). 
\label{formula4mu}
\end{equation}

Then we represent an important result that gives the error analysis of computing the principal eigenvalue of  $\mathcal{A}(t)$ using the particle method. Since the random fields we consider have an equispaced discrete spectrum, the following proposition remains valid. 

\begin{pro} [Theorem 3.11 in  \cite{Junlong2022}]
Suppose $\bm b(\bm x)$ and $c(\bm x)$ in $\mathcal{A}(\bm x)$ are bounded, smooth and periodic in each component of $\bm x$, and uniformly H\"{o}lder continuous in $t$. Let 
\begin{align}
    \mu^M_{\Delta t}(\lambda)=	(H\Delta t)^{-1}\sum_{h=0}^{H-1} \log\big( N^{-1}\sum_{l=1}^{N} 
\exp(c(\widetilde\xi_{h}^{l,M-1})\Delta t)\big),
\end{align} denote the approximate principal eigenvalue obtained by the $N$-IPS method, where $\widetilde\xi_{h}^{l,M-1}$, $h = 0,\cdots,H-1$, $l = 1,\cdots,N$, $M$ is the iteration number and $\Delta t = T/H$ is the time step. Let $\mu(\lambda)$ denote the principal eigenvalue of \eqref{linearCorrEqu} defined in Eq.\eqref{muCompFeyKac}. Then, we have the following convergence result
\begin{equation}
\lim_{N \to \infty} \frac{1}{H\Delta t}\sum_{h=0}^{H-1} \log\Big( \frac{1}{N}\sum_{l=1}^{N} 
\exp(c(\widetilde\xi_{h}^{l,M-1})\Delta t)\Big) = \mu(\lambda) + O\big((1-\frac{\theta}{\vartheta})^M\big) + O\big((\Delta t)^{\frac{1}{2}}\big),
\label{mainresult-NIPSPincipaleig}
\end{equation}
where $0<\theta<\vartheta $ are the parameters defined in the minorization and boundedness conditions in \eqref{MinorizationBoundedCondition}.
\label{Prop3.4}
\end{pro}

The standing assumption in \cite{Junlong2022} is a deterministic periodic flow while in this paper it is a random flow with an equispaced discrete spectrum. Since every random field realization at this point, still has a period, the Proposition \ref{Prop3.4} is still valid. As discussed in \eqref{NIPS-PrincipalEig} we substitute the long-time limit \eqref{muCompFeyKac} with the final state limit. This is due to the ergodicity nature of both the original system and the discrete system. Then the final error estimate in Proposition \ref{Prop3.4} consists of the difference of ergodic limits between continuous and discrete systems (generally a cell problem $O(\Delta t)$) and the difference between final state and ergodic limit in the discrete system (geometrically decreasing).  
\subsection{Error estimate in computing $\mu$ via random Fourier method}\label{sec:ErrorEstRanFourMethod}
\noindent
Here we will give some statistical error analysis on the solution operator. By the approximation to random velocity field and the operator splitting method, it has the following error in approximating the solution operator $\mathcal{U}(T, 0)$ in $L^2$ operator norm
\begin{equation}
    e_{op} \overset{\text{def}}{=} \big|\big|\mathcal{U} (T,0)q - \prod_{h=1}^H e^{\Delta t_h \tilde{\mathcal{L}} (\bm x,\omega)} e^{\Delta t_h \tilde{\mathcal{C}} (\bm x, \omega)}q \big|\big|_{L^2}.
\end{equation}
Here $\tilde{\mathcal{L}}$ and $\tilde{\mathcal{C}}$ denote the corresponding operators in \eqref{defAt} of $\tilde{\bm v}$. We can separate this error into three parts and define them by
\begin{equation}
    e_1 \overset{\text{def}}{=}  \big|\big|\mathcal{U} (T,0)q - \prod_{h=1}^H e^{\Delta t_h \mathcal{A} (\bm x,\omega)}q \big|\big|_{L^2},
\end{equation}
\begin{equation}
    e_2 \overset{\text{def}}{=} \big|\big|\prod_{h=1}^H e^{\Delta t_h \mathcal{A} (\bm x,\omega)}q - \prod_{h=1}^H e^{\Delta t_h \mathcal{L} (\bm x,\omega)} e^{\Delta t_h \mathcal{C} (\bm x,\omega)} q\big|\big|_{L^2},
\end{equation}
\begin{equation}
    e_3 \overset{\text{def}}{=} \big|\big| \prod_{h=1}^H e^{\Delta t_h \mathcal{L} (\bm x,\omega)} e^{\Delta t_h \mathcal{C} (\bm x,\omega)}q - \prod_{h=1}^H e^{\Delta t_h \tilde{\mathcal{L}} (\bm x,\omega)} e^{\Delta t_h \tilde{\mathcal{C}} (\bm x,\omega)} q \big|\big|_{L^2}.
    \label{Error3_Def}    
\end{equation}
Here $e_1, e_2$ and $e_3$ correspond to the error caused by solution operator \eqref{SoluOperatorfreetime1}, the Lie-Trotter operator splitting method, and using the random Fourier method. Then by triangle inequality, we can get that
\begin{equation}
    e_{op} \leq e_1 + e_2 + e_3 .
\end{equation}

\begin{pro}
\label{Prop3_5}
    Suppose $\bm b(\bm x, \omega)$ and $c(\bm x, \omega)$ in the operator $\mathcal{A} = \mathcal{L} +\mathcal{C}$ are bounded and smooth in a bounded domain $\mathcal{D}$, then there exists $\gamma_1$ large enough and $\gamma_2 > 0$, s.t.
    \begin{equation}
        ||\kappa \Delta_{\bm x} + (c(\bm x, \omega) - \gamma_1) \bm v||_{L^2} \geq \hat{C}_1(||\Delta_{\bm x} \bm v||_{L^2} + ||\bm v||_{L^2}),
    \end{equation}
    \begin{equation}
        ||[\mathcal{L}, \mathcal{C}] \bm v||_{L^2} \leq \hat{C}_2 ||(\mathcal{L} - \gamma_2) \bm v||_{L^2}^{\frac{1}{2}} ||\bm v||_{L^2}^{\frac{1}{2}},
    \end{equation}
    where $\hat{C}_1$ and $\hat{C}_2$ are constants and independent of $\bm x$ and $\omega$.
\end{pro}

 A detailed explanation can be found in \cite{Junlong2022} Lemma 3.2 and Lemma 3.3. From Lemma \ref{Lemma3_2}, we can derive the upper bounds for $||\bm b(\bm x, \omega)||_{L^2}$, $|c(\bm x,\omega)|$, $||\nabla_{\bm x}c(\bm x,$ $\omega)||_{L^2}$, $||\Delta_{\bm x} c(\bm x, \omega)||_{L^2}$ for a probability $p$ close to 1. Following the same approach, we can obtain the corresponding upper bounds for $||\tilde{\bm b}(\bm x, \omega)||_{L^2}$, $|\tilde{c}(\bm x,\omega)|$, $||\nabla_{\bm x}\tilde{c}(\bm x,$ $\omega)||_{L^2}$, $||\Delta_{\bm x} \tilde{c}(\bm x, \omega)||_{L^2}$.
\begin{lemma}
\label{Lemmma:Defpandupperbound}

We can construct $\gamma$ by these upper bounds, and the operator $\mathcal{A}-\gamma, \tilde{\mathcal{A}} -\gamma$ will both satisfy the requirements in Proposition \ref{Prop3_5} with probability greater than $p$.
\end{lemma}
\begin{proof}

From the analysis in Appendix \ref{sec::AppendixUpperBound}, we have obtained the upper bounds corresponding to the terms comprising $\bm b(\bm x,\omega)$ and $c(\bm x,\omega)$. So for any probability $p$ close to 1, we can find the number $\tilde{M}_{1}$, $\tilde{M}_{1}'$, s.t. 
\begin{equation}
    P(||\bm b(\bm x,\omega)||_{L^2} < \tilde{M}_{1}) > p \quad and \quad P(||\tilde{\bm b}(\bm x,\omega)||_{L^2} < \tilde{M}_{1}') > p.
    \label{CDFsense}
\end{equation}

Then we take $M_1^* = \max(\tilde{M}_{1}, \tilde{M}_{1}')$. Following the same idea, we can get the upper bounds $M_2^*, M_3^*$ and $M_4^*$ which correspond to $|c(\bm x,\omega)|, ||\nabla_{\bm x} c(\bm x,\omega)||_{L^2}$ and $||\Delta_{\bm x} c(\bm x,\omega)||_{L^2}$. Following the idea in \cite{Junlong2022} Lemma 3.2 and Lemma 3.3, for $\big(\bm b(\bm x,\omega)$, $c(\bm x,\omega)\big)$ and $\big(\tilde{\bm b}(\bm x,\omega)$, $\tilde{c}(\bm x,\omega)\big)$ , we can respectively find $(\gamma_1, \gamma_2)$ and $(\tilde{\gamma}_{1}, \tilde{\gamma}_{2})$. We take $\tilde\gamma = \max(\gamma_1, \gamma_{2}, \tilde{\gamma}_{1}, \tilde{\gamma}_{2})$ and take $\gamma = \tilde \gamma + M_2^*$ where $|c(\bm x,\omega)| < M_{2}^*$ and $|\tilde{c}(\bm x,\omega)| < M_{2}^*$ are satisfied with probability greater than $p$. Let $\mathcal{U}_{\gamma} (t, s) = e^{-\gamma(t-s)} \mathcal{U}(t,s)$ be the solution operator which corresponds to the parabolic equation \eqref{freetime1} with $\mathcal{A}_{\gamma} = \mathcal{A}_{\gamma} (\bm x,\omega) = \mathcal{A}(\bm x,\omega) - \gamma$, $\mathcal{L}_{\gamma}= \mathcal{L}_{\gamma}(\bm x,\omega) = \mathcal{L}(\bm x,\omega) - \tilde \gamma$ and $\mathcal{C}_{\gamma}=\mathcal{C}_{\gamma}(\bm x,\omega) = \mathcal{C}(\bm x,\omega) - M_2^*$. Similar for $\tilde{\mathcal{A}}(\bm x, \omega), \tilde{\mathcal{L}}(\bm x,\omega)$ and $\tilde{\mathcal{C}}(\bm x,\omega)$, we can obtain $\tilde{\mathcal{A}}_{\gamma}.\tilde{\mathcal{L}}_{\gamma}$ and $\tilde{\mathcal{C}}_{\gamma}$. 
\end{proof}

The error estimate of $e_1$ and $e_2$ has been proved in \cite{Junlong2022} and shown in Proposition \ref{Prop3.4}. Now we would like to give the error estimate of $e_3$.
\begin{thm}
\label{Thm::e3_5}
    For the probability $p$ defined in Lemma \ref{Lemmma:Defpandupperbound} and error $e_3$ from the random Fourier method and defined in \eqref{Error3_Def}, i.e.
    $$e_3 = \big|\big| \prod_{h=1}^H e^{\Delta t_h \mathcal{L} (\bm x,\omega)} e^{\Delta t_h \mathcal{C} (\bm x,\omega)}q - \prod_{h=1}^H e^{\Delta t_h \tilde{\mathcal{L}} (\bm x,\omega)} e^{\Delta t_h \tilde{\mathcal{C}} (\bm x,\omega)} q \big|\big|_{L^2}.$$ 
    
    Define event $B_1 = $ 
    \begin{equation}
    \Big\{ e_3 \leq Te^{\gamma T} \Big(||\bm v - \tilde{\bm v}||_{L^2}\cdot \big(||\nabla_{\bm x} q||_{L^2} + \lambda ||q||_{L^2}\big) + (\Delta t) \big(||\bm v - \tilde{\bm v}||_{L^2}\cdot||\nabla_{\bm x} q||_{L^2} + ||q||_{L^2}\big)\Big)\Big \},
    \label{Thm::e3_34}
    \end{equation}
    then $B_1$ holds with probability greater than $p$, i.e. $prob(B_1) > p$.
\end{thm}

 Detailed proof of Theorem \ref{Thm::e3_5} can be seen from Appendix \ref{Sec::AppendixTheorem3_5}. Now we look back at the total error estimate in a stochastic sense. From \eqref{b_rxL2normApp} to \eqref{UpperBoundforAss3Ass4}, we can get these terms' upper bounds distribution estimation. For any $p$ close to 1, we can always find their corresponding upper bounds' so that they are bounded by them with a probability greater than $p$. Following the selection method for $\gamma_1$ in \cite{Junlong2022}, we take $\gamma_1 = \frac{2 (\tilde{M}_{1})^2}{\kappa} + \tilde{M}_{2}$ and also require $\gamma_1$ should be large enough so that $4 (\frac{\gamma_1 - \tilde{M}_{2}}{3})^{\frac{3}{4}} \kappa^{\frac{1}{4}} > 2\kappa \tilde{M}_{3} C$. Here $C$ is defined by the fact that
\begin{equation}
    ||\nabla_{\bm x} q||_{L^2} \leq C ||\Delta_{\bm x} q||_{L^2}^{\frac{1}{2}} ||q||_{L^2}^{\frac{1}{2}}.
\end{equation}

For $\gamma_2$, it should be selected large enough so that 
\begin{equation}
    \big|\big|(\mathcal{L}(\bm x,\omega) - \gamma_2)q \big|\big|_{L^2} \geq \hat{C}(\kappa, \bm b) (||\Delta_{\bm x} q ||_{L^2} + ||q||_{L^2}),
\end{equation}
here $\hat{C}(\kappa, \bm b)$ should satisfy
$\hat{C}(\kappa, \bm b) \geq \max(\frac{16\kappa^2 C^2 \tilde{M}_{3}^2}{C_H^2}, \frac{4(\kappa \tilde{M}_{4} + \tilde{M}_{1}\tilde{M}_{3})^2}{C_H^2})$ and $C_H$ is the uniformly H\"{o}lder continuous coefficient. Similarly, we can select $\tilde{\gamma}_{1}$ and $\tilde{\gamma}_{2}$ according to $\tilde{\bm b}(\bm x,\omega)$ and $\tilde{c}(\bm x,\omega)$. Then we take $\tilde \gamma = \max(\gamma_1, \gamma_{2}, \tilde{\gamma}_{1}, \tilde{\gamma}_{2})$ and $\gamma = \tilde \gamma + M_2^*$. Then following the proof of $e_1$ in \cite{Junlong2022}, it can be proved that
\begin{equation}
    e_1 \leq Te^{\gamma T} \cdot \frac{C_s}{\beta + \frac{1}{2}} (\Delta t)^{\beta - \frac{1}{2}} ||q||_{L^2},
\end{equation}
here $C_s$ is the constant defined by the infinitesimal generator $\mathcal{A}$ and its corresponding analytical semigroup $e^{t \mathcal{A}}$ which satisfies
\begin{equation}
    ||\mathcal{A} e^{t\mathcal{A}} ||_{L^2} \leq \frac{C_s}{t},\quad t > 0.
\end{equation}

Then following the proof of $e_2$ in \cite{Junlong2022}, it can be proved that 
\begin{equation}
    e_2 \leq Te^{\gamma T} \cdot (\frac{4 C_s}{3} (\Delta t)^{\frac{1}{2}} + (M_2^*)^2 (\Delta t)) ||q||_{L^2}.
\end{equation}

Finally following the proof of $e_3$ mentioned above, we can obtain that
\begin{equation}
    e_3 \leq Te^{\gamma T} \cdot \Big(||\bm v - \tilde{\bm v}||_{L^2}\cdot \big(||\nabla_{\bm x} q||_{L^2} + \lambda ||q||_{L^2}\big) + (\Delta t) \big(||\bm v - \tilde{\bm v}||_{L^2}\cdot||\nabla_{\bm x} q||_{L^2} + ||q||_{L^2}\big)\Big).
\end{equation}

Now we consider one realization of $\tilde{\bm v}(\bm x,\omega_0) = \tilde{\bm v}_{0}(\bm x)$ with an initial Fourier mode number $N_{0}$ and frequency step $\Delta k$. Then we can construct a sequence $\{\bm v_{j}(\bm x)\}_{j=0}^{\infty}$ that each time we double the Fourier mode number, i.e. for $\bm v_{j}$, it will have a Fourier mode number $N_{j} = 2^j N_{0}$. During the construction process of each term in one realization, we keep the values of the random variables at the previously existing truncation points unchanged and add a subdivision based on them. Due to the exponential decay performance of $E(k)$, the amplitude of the high-frequency term will exponentially decay to zero. Then we treat $\bm v(\bm x) = \lim_{j\to \infty} \bm v_{ j}(\bm x)$ as the corresponding realization of real random velocity field $\bm v(\bm x)$. 
Now we look back to the error estimate of the correlation function of these series. From the analysis for \eqref{ErrEstRx}
, we know that the error of correlation function $e_{R}$ is $O((N_j)^{-2})$ in this case. Due to the construction strategy mentioned, $N_{j} = 2^j N_{0}$, then the estimate of $e_{R}$ will become $O(2^{-2j})$. From this, among several different realizations, we can further give the following estimate for $\mathbb{E}_{\omega}[||\bm v(\bm x, \omega) - \bm v_{j}(\bm x, \omega)||_{L^2}]$.
\begin{align}
    &\mathbb{E}_{\omega}[|| \bm v(\bm x, \omega) - \bm v_{j}(\bm x, \omega)||^2_{L^2}] = \mathbb{E}_{\omega}[\int_{\mathcal{D}} |\bm v(\bm x,\omega) - \bm v_{j} (\bm x,\omega)|^2 d\bm x] 
    \notag \\
    &= \int_{\mathcal{D}} \big(\mathbb{E}_{\omega}[\bm v(\bm x,\omega)^2] - \mathbb{E}_{\omega}[\bm v_{j}(\bm x,\omega)^2] -2 \mathbb{E}_{\omega}[(\bm v(\bm x,\omega) - \bm v_{j}(\bm x,\omega))\cdot \bm v_{j}(\bm x,\omega)] \big) d\bm x\notag \\
    & \leq \Big|\int_{\mathcal{D}}R(0) - \tilde{R}_{j}(0) - 2 \mathbb{E}_{\omega}[(\bm v(\bm x,\omega) - \bm v_{j}(\bm x,\omega))\cdot \bm v_{j}(\bm x,\omega)] d\bm x\Big|,
\end{align}
where $\tilde{R}_{j}(\cdot)$ is the correlation function of the generated velocity field $\bm v_{j}$. We define $ \bm v_{res}^{(j)} = \bm v -  \bm v_{j}$ and interpret $\bm v_{res}^{(j)}$ as $\bm v$ subtracting its components at the selected Fourier modes and the components at other Fourier modes will still form a mean zero random field which is independent of $\bm v_{j}$. Since $\mathcal{D}$ is a bounded domain, we use $S$ to denote its area and obtain that 
\begin{equation}
    \mathbb{E}_{\omega}[||\bm v(\bm x, \omega) - \bm v_{j}(\bm x, \omega)||^2_{L^2}] \leq S \cdot |R(0) - \tilde{R}_{j}(0)| = S \cdot e_{R,j},
    \label{ErrestforEwvr_vjL2}
\end{equation}
where $e_{R,j} \overset{\text{def}}{=} |R(0) - \tilde{R}_{j}(0)| $, is the error between two velocity fields' corresponding correlation functions. Then we can further know that $\mathbb{E}_{\omega}[||\bm v(\bm x, \omega) - \bm v_{j}(\bm x, \omega)||^2_{L^2}]$ will have a convergence rate $O(2^{-2j})$. It also means that $\mathbb{E}_{\omega}[||\bm v(\bm x, \omega) - \bm v_{j}(\bm x, \omega)||_{L^2}]$ will at least have a convergence rate $O(2^{-j}) = O((N_j)^{-1})$ and its variance $Var(||\bm v(\bm x,$ $ \omega) - \bm v_{j}(\bm x, \omega)||_{L^2})$ will also converge to zero with a convergence rate at least $O(2^{-2j}) = O((N_j)^{-2})$. In the above analysis, we did not make further requirements on realization, thus this convergence rate is available for arbitrary realizations. By following the perturbation theory in \cite{kato2013perturbation}, we can obtain that 

\begin{thm}
    In each realization, the corresponding velocity field $\{ \bm v_{j}\}_{j=0}^{\infty}$ and $\bm v$ will define their corresponding operator $\{e^{\Delta t \mathcal{L}_{\gamma}^{j} (\bm x)}e^{\Delta t \mathcal{C}_{\gamma}^{j} (\bm x)}  \}_{j=0}^{\infty}$ and $e^{\Delta t \mathcal{L}_{\gamma} (\bm x)}e^{\Delta t \mathcal{C}_{\gamma} (\bm x)}$, then for sufficiently large $j$, 
    $$e^{\Delta t \mathcal{L}_{\gamma}^{j} (\bm x)}e^{\Delta t \mathcal{C}_{\gamma}^{j} (\bm x)} \to e^{\Delta t \mathcal{L}_{\gamma} (\bm x)}e^{\Delta t \mathcal{C}_{\gamma} (\bm x)}\ in\ the\  generalized\ sense.$$ 
    
    Also for sufficiently large $j$, the principal eigenvalues of $e^{\Delta t \mathcal{L}_{\gamma}^{j} (\bm x)}e^{\Delta t \mathcal{C}_{\gamma}^{j} (\bm x)} $ and $ e^{\Delta t \mathcal{L}_{\gamma} (\bm x)}e^{\Delta t \mathcal{C}_{\gamma} (\bm x)}$ is close, their difference is small in the sense of \ref{PerbCtsEigTx}. 
    \label{Thm::ConvOfPrinEigofRFM}
\end{thm}
\begin{proof}
Here we define $\mathcal{I}(j) = e^{\Delta t \mathcal{L}_{\gamma}^{j} (\bm x)}e^{\Delta t \mathcal{C}_{\gamma}^{j} (\bm x)} - e^{\Delta t \mathcal{L}_{\gamma} (\bm x)}e^{\Delta t \mathcal{C}_{\gamma} (\bm x)}$. Then, following \eqref{erreste3partdeltat} we can obtain 
\begin{equation}
    ||\mathcal{I}(j) q||_{L^2} \leq \big(\lambda \Delta t ||\bm v - \bm v_{j}||_{L^2} + O((\Delta t)^2) \big)||q||_{L^2} + (\Delta t + (\Delta t)^2) \cdot ||\bm v - \bm v_{j}||_{L^2} ||\nabla_{\bm x} q||_{L^2},
\end{equation}
here the term $O((\Delta t)^2)$ actually should be infinitely termed' sum which all consist of the coefficient $||\bm v - \bm v_{j}||_{L^2}$. Due to we do not care so much about the term with a high $\Delta t$ order, here we still write $O((\Delta t)^2)$, but we know this term converges to zero as $j$ goes to infinity. then $||\nabla_{\bm x}q ||_{L^2}$ is bounded by $||(\mathcal{A} - \gamma)q||_{L^2}$, then we can get it is also bounded by $C_J||e^{\Delta t \mathcal{L}_{\gamma} (\bm x)}e^{\Delta t \mathcal{C}_{\gamma} (\bm x)}q||_{L^2}$. Then we can obtain that $\mathcal{I}(j), j = 0,1,2,...$ is $e^{\Delta t \mathcal{L}_{\gamma} (\bm x)}e^{\Delta t \mathcal{C}_{\gamma} (\bm x)}$-bounded and both the coefficient of $||q||_{L^2}$ term and $||\nabla_{\bm x} q||_{L^2}$ term converge to zero as $j$ goes to infinity. Following \ref{SuffCondGenCon}, we can get that  $e^{\Delta t \mathcal{L}_{\gamma}^{j} (\bm x)}e^{\Delta t \mathcal{C}_{\gamma}^{j} (\bm x)} \to e^{\Delta t \mathcal{L}_{\gamma} (\bm x)}e^{\Delta t \mathcal{C}_{\gamma} (\bm x)}$ in the generalized sense.

Then following the spectral gap assumption, \ref{PerbCtsEigTx} and \ref{RelaPerOpandEig}, we can obtain that when adding a small perturbation $\mathcal{I}(j)$ on the original operator $e^{\Delta t \mathcal{L}_{\gamma} (\bm x)}e^{\Delta t \mathcal{C}_{\gamma} (\bm x)}$, the change of the principal eigenvalue is also small in the sense that the principal eigenvalue is continuous at the coefficient of the perturbation equal to zero. Then we finish the proof and can extend this conclusion to the one iteration from $0$ to $T$. 
\end{proof}
\begin{thm}
    For sufficiently small $\Delta t$ and large $j$, we denote the following error estimate of the random Fourier method to approximate the  operator obtained by the Lie-Trotter operator splitting method \eqref{LieTrotterAPPSoluOperator} as event $B_2$,
    \begin{equation}
        B_2 = \Big\{ \big|\big|\prod_{h=1}^H e^{\Delta t \mathcal{L}} e^{\Delta t \mathcal{C} } - \prod_{h=1}^H e^{\Delta t \mathcal{L}_{j} } e^{\Delta t \mathcal{C}_{j} } \big|\big|_{L^2} \leq C_3(\lambda, p) \sqrt{||\bm v - \bm v_{j}||^2_{L^2} + (\Delta t)^2} \Big\},
        \label{Theorem3.6e3Con}
    \end{equation}
    and $B_2$ holds with a probability greater than $p$, i.e. $prob(B_2) > p$. Here $p$ is a probability which approaches to 1 (see Lemma \ref{Lemmma:Defpandupperbound}) and to be simplify, we use $\mathcal{L},\mathcal{C},\mathcal{L}_{j}, \mathcal{C}_{j}$ to denote $\mathcal{L}(\bm x,\omega)$, $\mathcal{C}(\bm x,\omega)$,$\mathcal{L}_{j}(\bm x,\omega)$, $\mathcal{C}_{j}(\bm x,\omega)$.  
    \label{Thm::Errofe3ande}
\end{thm}

Detailed proof of Theorem \ref{Thm::Errofe3ande} can be seen from Appendix \ref{sec::AppendixProofThm3_7}. Together with the lemma and theorem mentioned above, we can give an error estimate for the principal eigenvalue.

\begin{thm}
    Let $e^{\mu(\lambda) T}$ and $e^{\mu_{\Delta t}(\lambda) T}$ denote the principal eigenvalue of the solution operator $\mathcal{U}(T,0)$ and the approximated solution operator $\prod_{h=1}^H e^{\Delta t_h \tilde{\mathcal{L}} (\bm x,\omega)}$ $ e^{\Delta t_h \tilde{\mathcal{C}} (\bm x,\omega)}$ respectively. For given $E(k)$ satisfies P1 to P4, denoting the following error estimate as event $B_3$,
    \begin{equation}
        B_3 = \Big \{|e^{\mu(\lambda) T} - e^{\mu_{\Delta t}(\lambda) T}| \leq C_4(p)\cdot (\Delta t)^{\frac{1}{2}} + C_5(\lambda,p)\cdot ||\bm v - \tilde{\bm v}||_{L^2} \Big \},
        \label{ErrEstforemu}
    \end{equation}
 then $B_3$ holds with a probability greater than $p$ for sufficiently small $\Delta t$, and large $N_F$. Here $p$ is a probability which approaches to 1 (see Lemma \ref{Lemmma:Defpandupperbound}). $\mathbb{E}_{\omega}[||\bm v - \tilde{\bm v}||_{L^2}]$ has a convergence rate $O((N_F)^{-1})$ and $Var(||\bm v- \tilde{\bm v}||_{L^2})$ has a convergence rate $O((N_F)^{-2})$. Moreover, we can obtain that in this case, $|\mu(\lambda) - \mu_{\Delta t}(\lambda)| = O(C_4(p)(\Delta t)^{\frac{1}{2}} + C_5(\lambda, p)$ $||\bm v - \tilde{\bm v}||_{L^2})$. To simplify, we denote $C_4 = C_4(p),\ C_5 = C_5(\lambda,p)$ and
 $$ \mu_{IPM}(\lambda) = \lim_{N \to \infty} \frac{\sum_{h=0}^{H-1} \log\Big( N^{-1}\sum_{l=1}^{N} 
\exp(\tilde{c}(\widetilde\xi_{h}^{l,M-1})\Delta t)\Big) }{H\Delta t}.$$
 
 Combining with the error estimate of the Lagrangian (particle) method in Section \ref{sec:AnalysisLagrangianMethod}, we have the following convergence result: 
 Define event $B_4 = $
\begin{equation}
\Big \{\mu_{IPM}(\lambda) = \mu(\lambda) + O\big((1-\frac{\theta}{\vartheta})^M\big) + O\big(C_4(\Delta t)^{\frac{1}{2}}\big) + O(C_5 ||\bm v - \tilde{\bm v}||_{L^2}) \Big \},
   \label{ErrEstformu}
\end{equation}
\label{Thm::TotalErrEstforprinEig38}
then $B_4$ holds with probability greater than $p$. 
\end{thm}
\begin{proof}
    According to the standard spectral theorem in \cite{kato2013perturbation}, the principal eigenvalue of these two operators $\mathcal{U}(T,0)$ and $\prod_{h=1}^H e^{\Delta t_h \tilde{\mathcal{L}} (\bm x,\omega)} e^{\Delta t_h \tilde{\mathcal{C}} (\bm x,\omega)}$ will satisfy
    \begin{equation}
       |\mu(\lambda) - \mu_{\Delta t}(\lambda)| \leq C_{sp} \big|\big|\mathcal{U}(T,0) - \prod_{h=1}^H e^{\Delta t_h \tilde{\mathcal{L}} (\bm x,\omega)} e^{\Delta t_h \tilde{\mathcal{C}} (\bm x,\omega)} \big|\big|_{L^2} .
       \label{StandardSpectralThm}
    \end{equation}
    Here we first define $C_4(p) = \frac{8}{3} C_s T e^{\gamma T}$ and $C_5(\lambda, p) = 4 (C_D(p) + \lambda) T e^{\gamma(T+1)}$. Follow the error estimate we analysed for $e_1,e_2,e_3, ||\bm v - \tilde{\bm v}||_{L^2}$ and Theorem \ref{Thm::Errofe3ande}. By the triangle inequality, we can get \eqref{ErrEstforemu}. And then we can accordingly obtain the error estimate for $|\mu(\lambda) - \mu_{\Delta t}(\lambda)|$. Then due to the analysis in Section \ref{sec:AnalysisLagrangianMethod} and by triangle inequality, we get the convergence result \eqref{ErrEstformu}. 
    \end{proof}

\section{Numerical experiments}\label{sec:NumericalResults}
\noindent
We present several numerical examples computed by the IPM. We first show examples to verify the convergence and accuracy of the IPM method in computing $\mu$. Then we consider 2D and 3D random fields with the equispaced discrete spectrum where $\Delta k = \frac{1}{20\pi}$ and compute front speeds $c^*$ of these cases. Meanwhile, we investigate and analyze the relationship between  $c^*$ and the magnitude of the velocity fields on torus $\mathcal{D}$. We take domain size to infinity and study the finite domain effect to $c^*$ on $\mathbb{R}^d$ by a normalized/centered IPM. Such an unbounded domain problem is impossible to compute by mesh-based methods without a proper domain truncation. 

\subsection{Convergence test in 
computing $\mu$}\label{sec:ConAnainCompPrinEignVal}
\noindent
The convergence test of the operator splitting method and using the IPM method in computing principal eigenvalues of determined flows can be seen in \cite{Junlong2022}.
Different from the test in the deterministic velocity field, here we also need to verify the convergence between generated random fields (using the random Fourier method \cite{Majda:99}) and real random potential fields. Here we choose a specific $E(k)$ which satisfies the properties $P1$ to $P4$ we proposed in Section \ref{sec:AppofRanFieldusingRFM}   
\begin{equation}
    E(k) = |k|^{\frac{1}{2}} e^{-|k|}.
    \label{SpecEneDen05_1}
\end{equation}

We plot the performance of the coefficient of the generated random velocity field. From Figure \ref{CoeffPerGenRanFie} we see that as frequency $k$ goes to infinity, the coefficient of the corresponding Fourier mode term in the generated random velocity field goes to $0$. We also test convergence using the random Fourier method generated velocity in computing the principal eigenvalues for a 2D random shear flow $v(x, y) = \delta \cdot (0, \xi(x))$. Here $\delta = \frac{1}{40}$ and $\xi(x)$ is set as follows:
\begin{equation}
    \xi(x) = \sum_{j = 0}^{N_F} \sqrt{2 E(k_j) \Delta k_j}\cdot [\zeta_{j} \cos(2\pi k_j x) + \eta_{j} \sin (2\pi k_j x)].
    \label{xi(x)def}
\end{equation}

\begin{figure}[tbph]
\centering
     \begin{subfigure}{0.45\textwidth}
         \includegraphics[width = \linewidth, height = 4cm]{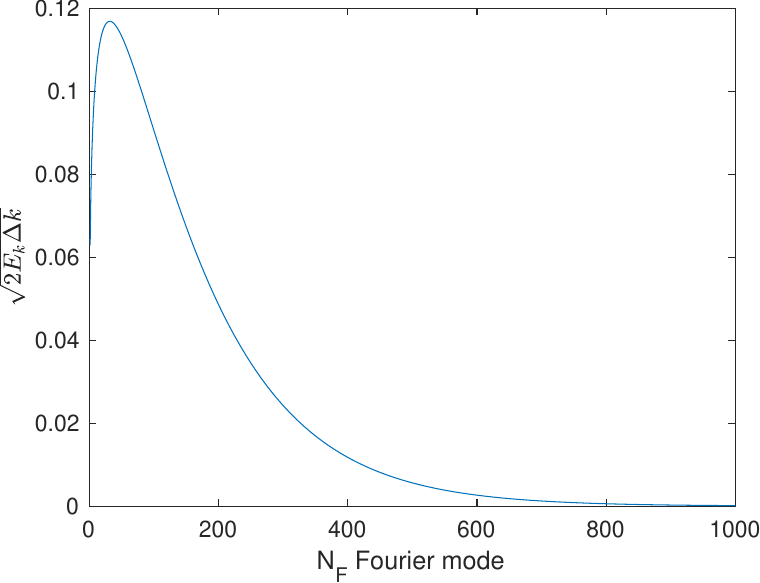}
         \caption{Coefficient value of Fourier mode}
         \label{CofValofFourMode}
     
     \end{subfigure}
     \begin{subfigure}{0.45\textwidth}
         \includegraphics[width = \linewidth, height = 4cm]{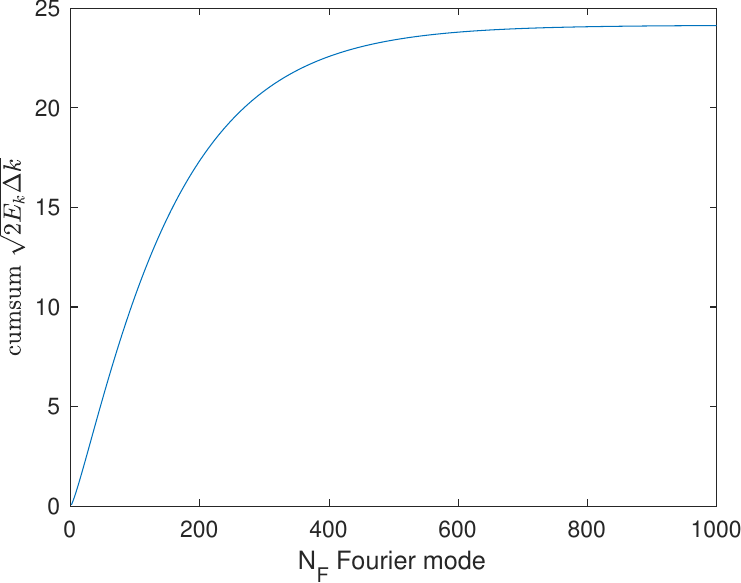}
         \caption{Cumsum of coefficient value}
         \label{CumSumofFourMode}    
     \end{subfigure}
     \caption{Performance of the coefficient of the generated random velocity field.}
     \label{CoeffPerGenRanFie}
\end{figure} 
\begin{figure}[h]
    \centering
    \includegraphics[width = 0.75\textwidth,height=5cm]
    {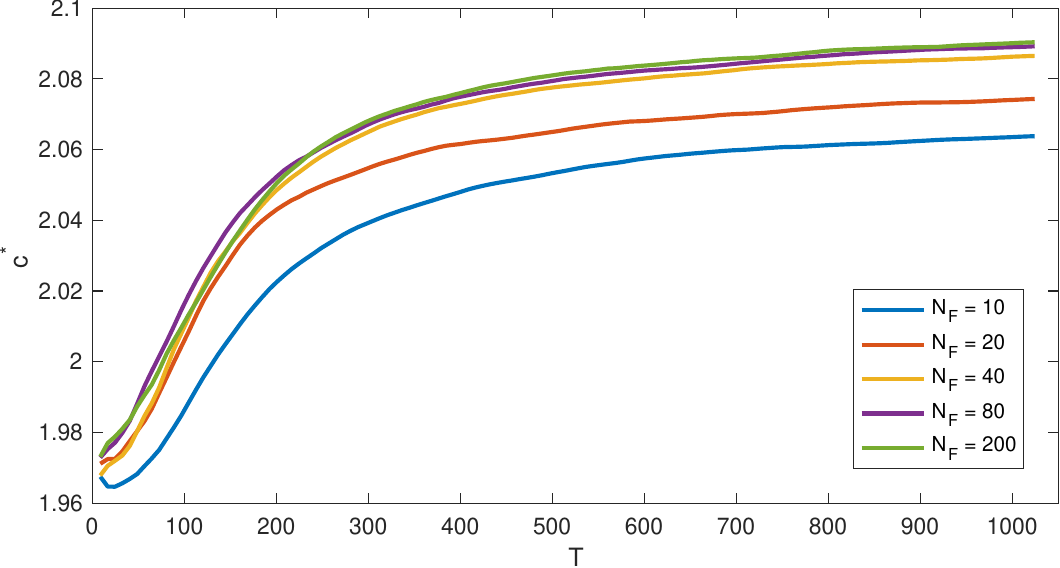}
    \caption{Front speed $c^*$ vs. $t$ of random shear flow with different Fourier modes $N_F$.}
    \label{NmodePrinEigCon}
\end{figure}

We can see from Figure \ref{NmodePrinEigCon} that in one random realization, as we increase the Fourier modes number, the principal eigenvalues we get are very close when the Fourier modes number $N_F$ is near or larger than 200 in this case.

\subsection{IPM vs. Semi-Lagrangian computation of KPP front speed}\label{sec:CompKPPbetIPMandSemiLagr}
\subsubsection{2D cellular flow}\label{sec:CompIPMandSemi_2DCell}
\noindent
We compare interacting particle method, spectral and semi-Lagrangian methods in computing KPP front speeds
in 2D cellular flow:
\begin{equation}
    v(x, y) = \delta\cdot\big(-\sin(x)\cos(y),\ \cos(x) \sin(y)\big) .
\end{equation}

We take the spectral method as a benchmark and select $\delta$ from $2^0$ to $2^4$. Figure \ref{2DCellComp} shows that these three methods give similar results. The log-log fitted curve has a slope of 0.257 so that $c^* = O(\delta^{0.257}) \approx O(\delta^{0.25})$, consistent with theory   \cite{Novikov2007}. In Figure \ref{2DCellComp} 
semi-Lagrangian and spectral methods are in agreement. Figure \ref{Searche} is an example for searching $\lambda \bm e$ to reach minimal front speed. Here $\bm z = (1,0)$. We do a grid search for the optimal vector $\lambda \bm e$ and determine $c^*(\bm z)$ by its variational formula.
We found that the optimal dual direction vector $\bm e$ is either equal to $\bm z$ or near $\bm z$. The former occurs for random flows and reduces the search to only the scalar variable $\lambda$ thereby saving computation time. The latter occurs for flows with ordered streamlines (e.g. cellular flows) for which a local search around $\bm z$ suffices instead of a time-consuming global research.  
\begin{figure}[h]
    \centering
    \begin{subfigure}{0.5\textwidth}
    \includegraphics[width = \linewidth,height=4cm]{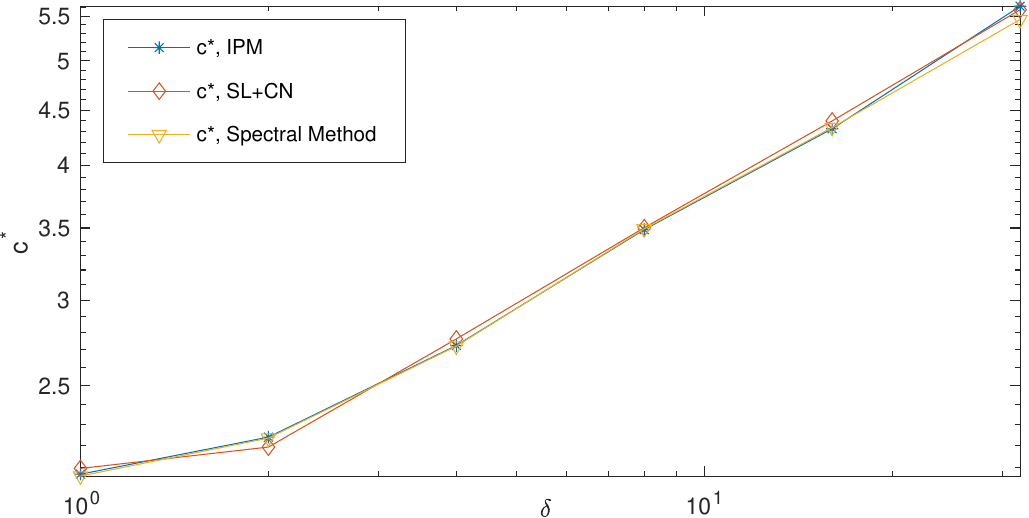}
    \caption{Comparison of $c^*$ in 2D cellular flow}
    \label{2DCellComp}
\end{subfigure}    
\begin{subfigure}{0.45\textwidth}
    \includegraphics[width = \linewidth,height=4cm]{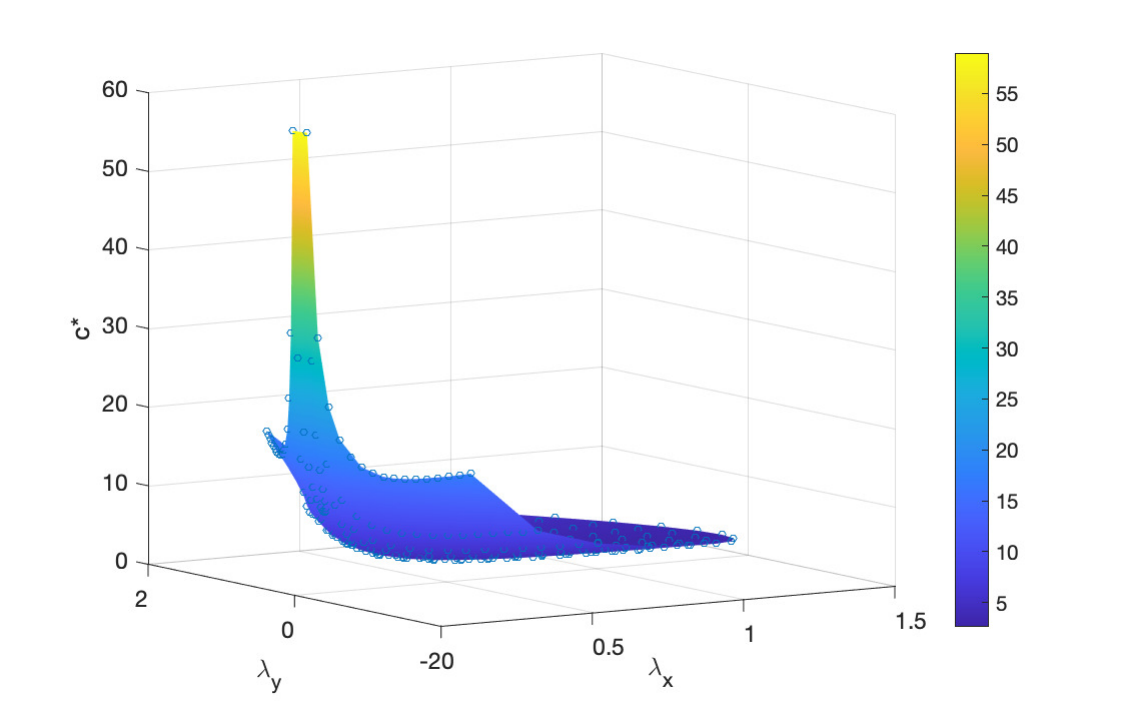}
    \caption{Search for $\lambda \bm e$, $\delta = 2^2$}
    \label{Searche}
    
\end{subfigure}    
\caption{Computation of $c^*$ in 2D cellular flow.}
\end{figure}

\subsubsection{2D random shear flow}\label{sec:CompIPMandSemi_2Dransh}
\noindent Now we adjust the velocity field with only the generated random shear field, i.e.
$v(x,y) = \delta \cdot (0, \xi(x))$. Here we would like to simulate the case in $\mathbb{R}^2$.  We can see that given sufficient Monte Carlo number $N$ and evolution time $n$ in the IPM, the Semi-Lagrangian method needs increasingly fine grids to get similar results. And during the process that we enlarge the domain to 'infinity', for the IPM, we just need to change the initial distribution and similarly compute the particle traces. This does not have any effect on the computation time. For the Semi-Lagrangian method, as we enlarge the domain, we will need to increase the number of times of the grids if we want to maintain the same partition accuracy. 
We compute the KPP front speed $c^*$ in different domain settings by these two methods with the same random realization. As in \cite{Nolen2008}, for the advection-reaction step, we use a semi-Lagrangian scheme; and for the diffusion step, we use an implicit Crank-Nicholson scheme in time and spectral method in space. We denote this method as SL+CN. In the SL+CN method, if we would like to simulate the situation with an unbounded domain, we can only enlarge the domain step by step. For example, we start with the domain $[0,2\pi]^2$ and enlarge it to $[0, 4\pi]^2,\ [0,8\pi]^2,\ ...\ $. In the Eulerian framework of the SL+CN method, we need to divide the whole domain into $N_x \times N_y$ grids. Here for simplicity, we just assume $N_x = N_y$ and if we want to maintain the mesh size during the process of enlarging the domain, the time cost will get more and more expensive. If we would like to maintain the time cost, then the accuracy of the SL+CN method will get worse and worse. The time complexity of the Semi-Lagrangian method is proportional to the number of grid points, i.e. if we set domain length as $L$ and maintain the mesh size, then the advection-reaction step's time complexity is $O(N_x^2) = O(L^2)$ in 2D cases. Benefiting from the FFT algorithm, the time complexity of the diffusion step is $O(L\log L)$. Time cost will become more expensive when we consider 3D cases by SL+CN. 
\begin{figure}[h]
\centering
\begin{subfigure}{0.48\textwidth}
    \includegraphics[width = \linewidth]{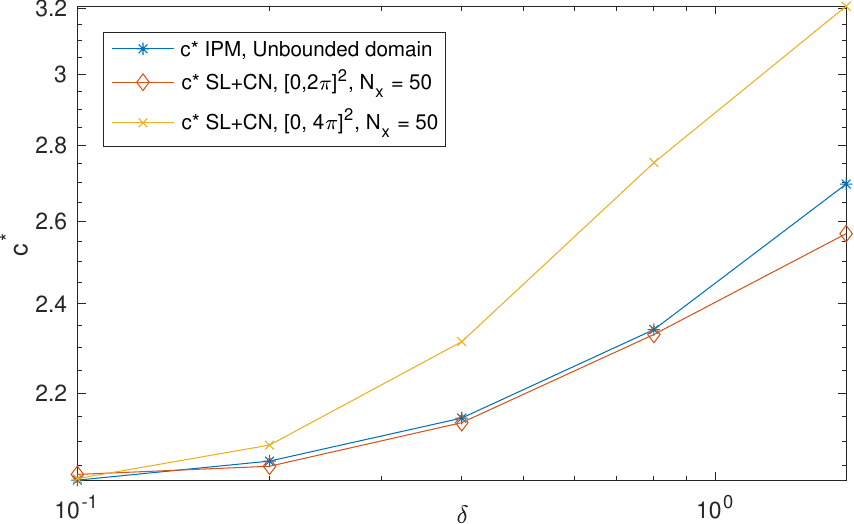}
    \caption{Same $N_x$, different domain lengths}
    \label{2DshycDLComplogc}
\end{subfigure}    
\begin{subfigure}{0.48\textwidth}
    \includegraphics[width = \linewidth]{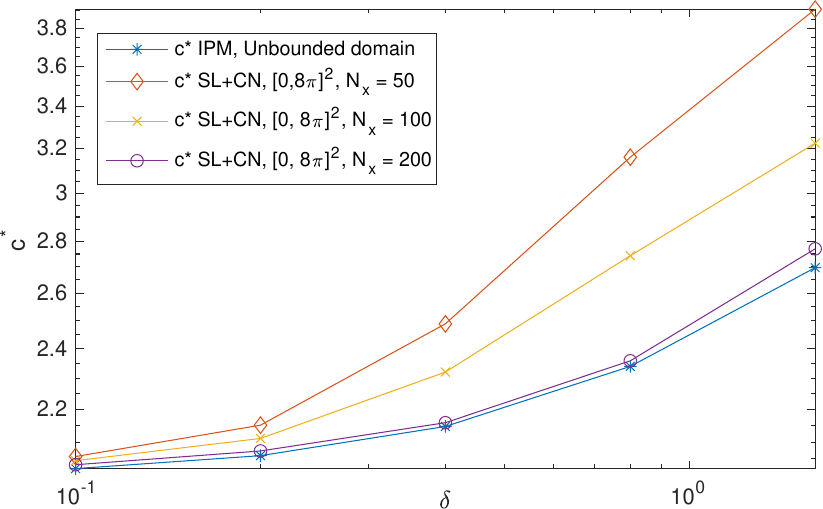}
    \caption{Same domain length, different $N_x$}
    \label{2DshycNxComplogc}
    
\end{subfigure}    
	\caption{Comparison of $c^*$ by IPM and SL+CN method.}
	\label{2DshycSLCNComp}
\end{figure} 

It is seen from Figure \ref{2DshycDLComplogc} and \ref{2DshycNxComplogc} that as the domain size $L$ gets larger, more grid points are needed to maintain the accuracy of SL+CN method. 
\subsection{3D cases' time complexity}Meanwhile,  we also test the time cost of SL+CN method in 3D cellular flow with evolution time $T = 100$ and varying meshgrid sizes. Where for the meshgrids $30\times 30 \times 30$, $35\times 35 \times 35$ and $40 \times 40 \times 40$, since it takes too long to evolve and the computational cost of each step of the evolution is similar, we only evolve one-tenth of the whole process, i.e. $T=10$ and estimate the computation time of the whole evolution process, i.e. $T = 100$. As illustrated in Figure \ref{timefor3DSLCN}, the computational time complexity is $O((N_x)^{7.90}) \approx O((N_x)^{8})$. This is due to the fact that the time complexity of an implicit format solver implemented using sparse matrices will be between $O((N_x)^6)$ and $O((N_x)^9)$ depending on the characteristics and structure of the sparse matrices. Such a time complexity is undoubtedly hard to accept when the SL+CN method would like to reach high resolution and the memory cost is also substantial. If we use the explicit scheme such as the Runge-Kutta 2 (RK2) scheme, since the entire format is implemented using sparse matrices, the time complexity illustrated in Figure \ref{3DSLEultime} is $O((N_x)^{2.90}) \approx O((N_x)^{3})$. Due to the CFL condition and other constraints, we have to shrink the time step $\Delta t$ when we increase the number of mesh grid points $N_x$. With adaptive time step $\Delta t$, the SL+RK2 method's time complexity shown in Figure \ref{3DSLEultime} is $O((N_x)^{4.90}) \approx O((N_x)^{5})$. We take the calculation with the mesh size of $80\times 80\times 80$ as a reference. To achieve the reference result of 0.01 tolerance, the SL+RK2 method requires $1.26 \times 10^4$ seconds, while our IPM takes only $7.62 \times 10^2$ seconds.  

\begin{figure}[ht]
    \centering
    \begin{minipage}{0.5\textwidth}
        \centering
        \includegraphics[width=\linewidth]{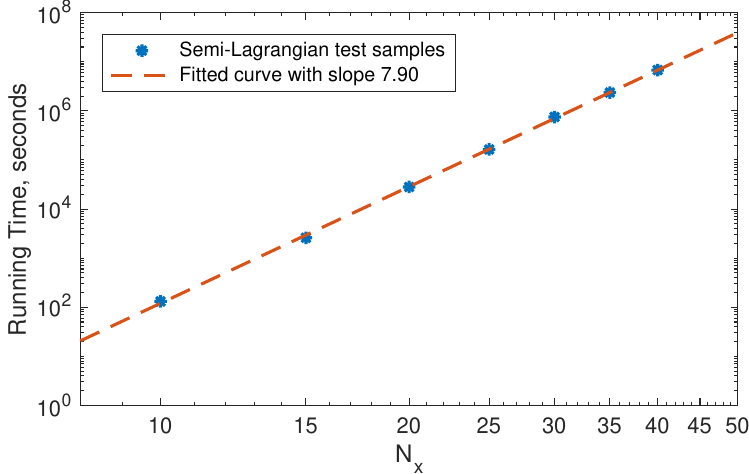}
        
    \end{minipage}%
    \begin{minipage}{0.45\textwidth}
        \centering
        \renewcommand{\arraystretch}{1.2} 
        \begin{tabular}{|c|c|}
            \hline
            Mesh size & Running time \\ \hline
    $10\times 10 \times 10$ &  $1.31\times 10^2$s \\ 
    $15\times 15 \times 15$ &  $2.58\times 10^3$s \\ 
    $20\times 20 \times 20$ &  $2.81\times 10^4$s \\ 
    $25\times 25 \times 25$ &  $1.63\times 10^5$s \\ 
    $30\times 30 \times 30$ &  $7.49\times 10^5$s * \\ 
    $35\times 35 \times 35$ &  $2.36\times 10^6$s *\\ 
    $40\times 40 \times 40$ &  $6.77\times 10^6$s *\\ \hline
        \end{tabular}
        
    \end{minipage}
    \caption{Loglog plot and table of mesh size $N_x$ vs. Running time of SL+CN method in 3D Cellular flow (* in the table means the time is estimated).}
    \label{timefor3DSLCN}
\end{figure}

\begin{figure}[h]
\centering
    \includegraphics[width=0.7\linewidth]{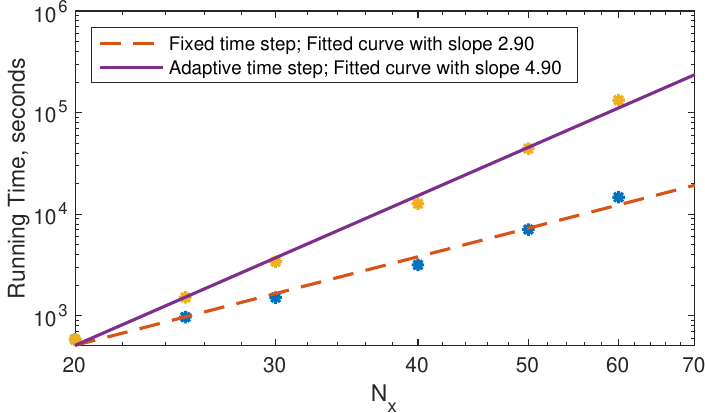}
    \caption{Loglog plot of Mesh size $N_x$ vs. Running time}
    \label{3DSLEultime}
    
\end{figure}

\subsection{IPM computation of KPP front speeds in random flows}\label{sec:CompKPPc*indiffRanFlows}
 
\subsubsection{Random 2D cellular flow}\label{sec:KPPc*in2DSpOUflow}
\noindent
Here we change the velocity field into 2D cellular flow with random perturbation. The approximation to the random perturbation is generated by the random Fourier method, i.e. $v(x, y)$ is now changed into:
\begin{equation}
    v(x, y) = \delta \cdot(-\sin(x)\cos(y),\ \cos(x) \sin(y) + \epsilon\cdot \xi(x)).
    \label{2DRanCellFlow}
\end{equation}
Here $\xi(x)$ is defined the same as \eqref{xi(x)def} and follows the same energy spectral density $E(k)$ defined in \eqref{SpecEneDen05_1}. This is a 2D cellular flow with a random perturbation. In this numerical experiment, we set 
the iteration number $n = 256,\ \Delta t = 2^{-9}$, monte carlo number $N = 100,000$ (sufficient large). For the random velocity field with $\Delta k = \frac{1}{20 \pi}$, we set $N_F = 400$, $k_j = j\Delta k$. Then we compute the KPP front speeds with different coefficients of the random perturbation. 

\begin{figure}[h]
    \centering
    \includegraphics[width = 0.7\textwidth]{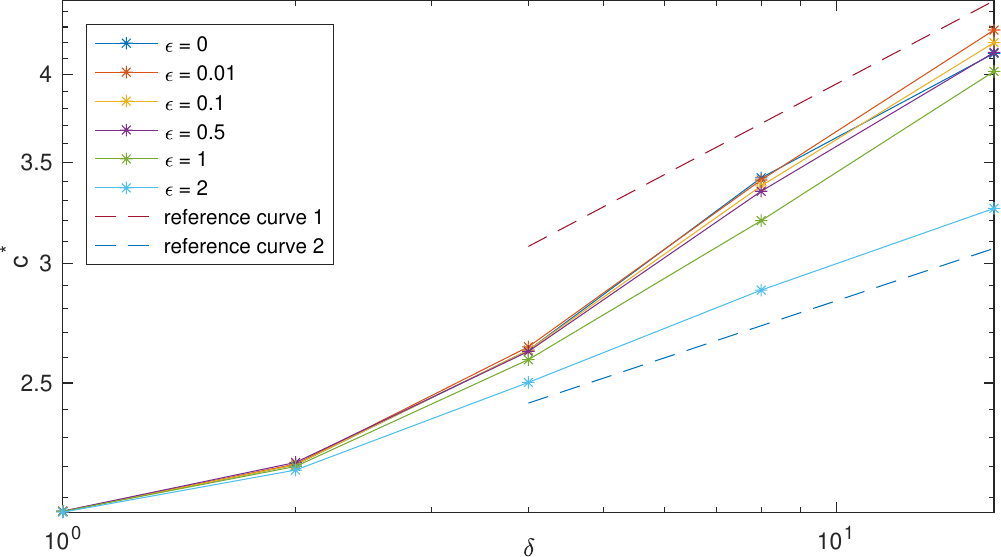}
    \caption{Log-log plot of $c^*$ vs. $\delta$ in 2D random cellular flow at different $\epsilon$; $z = (1,0)$.}
    \label{2DCellSpaceOUcxeps}
\end{figure} 
\begin{figure}[h]
    \centering
    \includegraphics[width = 0.7\textwidth]{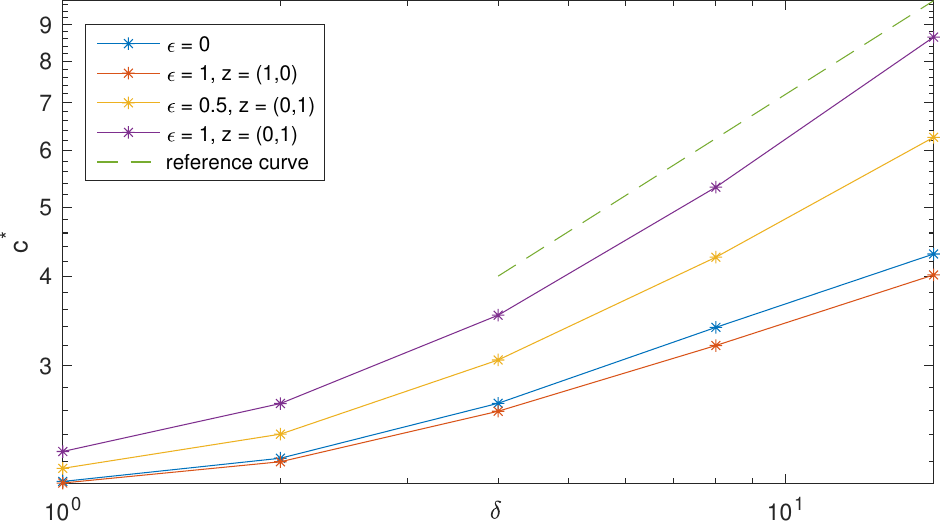}
    \caption{Log-log plot of $c^*$ vs. $\delta$ in 2D random cellular flow with different $\epsilon$ and $z$.}
    \label{2DCellSpaceOUcyeps}
\end{figure}

In Figure \ref{2DCellSpaceOUcxeps}, the two reference curves are with slopes of 0.26 and 0.17. Here we compute 30 realizations and take their averages. During the computation process, we found that the variance of these different realization samples is very small. Then from Figure \ref{2DCellSpaceOUcxeps}, we can see that: (1) As the coefficient of the random term $\epsilon$ is not equal to zero and increases, the front speed with direction $\bm z = (1,0)$, This term demonstrate the power-law scaling of $c^*$ concerning $\delta$, the fitted curves' slope is getting smaller and smaller. (2) As the coefficient of the random term $\epsilon$ decreasing and converging to zero, the front speed $c^*$ converges to the case with no random term (i.e. $\epsilon = 0$).

Then we changed the direction $\bm z = (0,1)$ and kept other parameters the same and got Figure \ref{2DCellSpaceOUcyeps}. In Figure \ref{2DCellSpaceOUcyeps}, the reference curve's slope is 0.64. We can see that if we change the direction $\bm z$ into $(0,1)$, then as the random coefficient $\epsilon$ increases, the front speed $c^*$ will also increase.

Then we compare the invariant measure of 2D Cellular flow with and without the addition of random perturbation term on 2-dimensional torus $\mathcal{D}$. From Figure \ref{2DCellEICompx}, by comparing with the sub-figures without random term (Left: \ref{0x2piEI} and \ref{0y2piEI}) and the invariant measure of the sub-figures with the random term (Right: \ref{1x2piEI} and \ref{05y2piEI}), we can see that the addition of the random term creates a degree of disruption to original invariant measure.

\begin{figure}[h]
\centering
         \begin{subfigure}{0.4\textwidth}
         \includegraphics[width = \linewidth,height = 3.5cm]{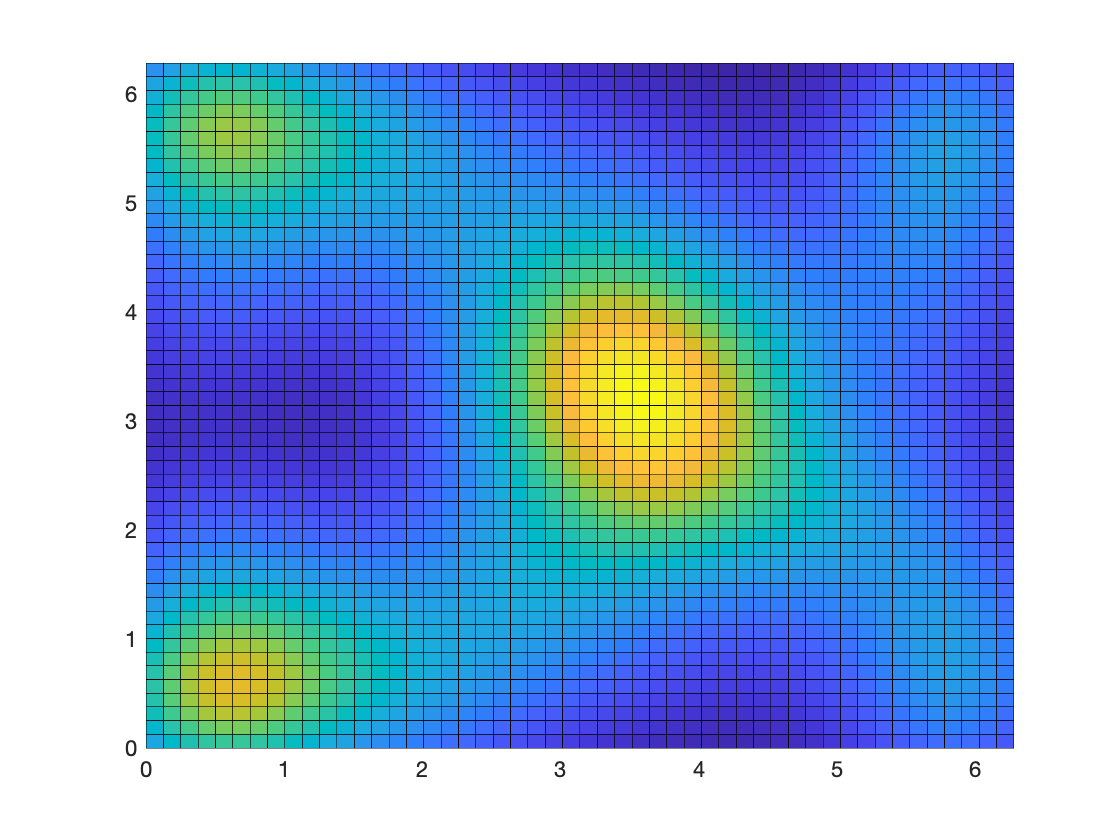}
         \caption{$\epsilon = 0,\ z = (1, 0)$}
         \label{0x2piEI}
         \end{subfigure}
         \begin{subfigure}{0.4\textwidth}
         \includegraphics[width = \linewidth,height = 3.5cm]{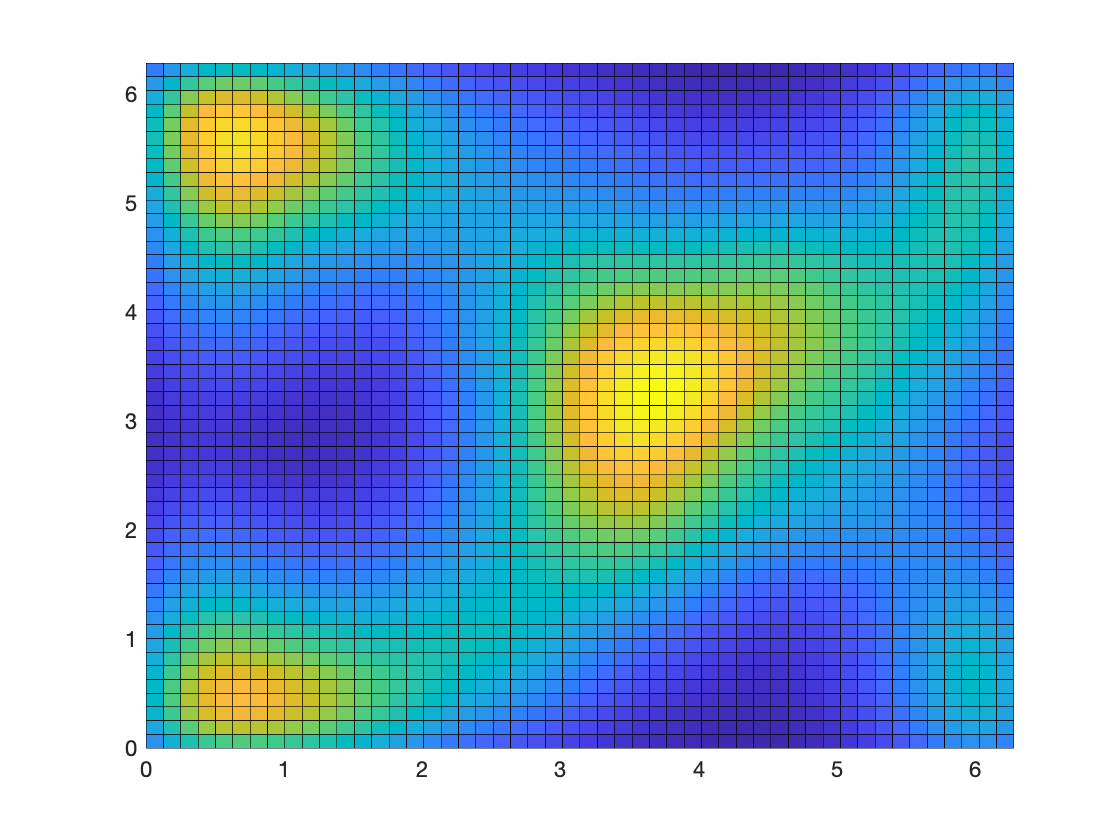}
         \caption{$\epsilon = 1,\ z = (1, 0)$}
         \label{1x2piEI}
         \end{subfigure}

         \begin{subfigure}{0.4\textwidth}
         \includegraphics[width = \linewidth,height = 3.5cm]{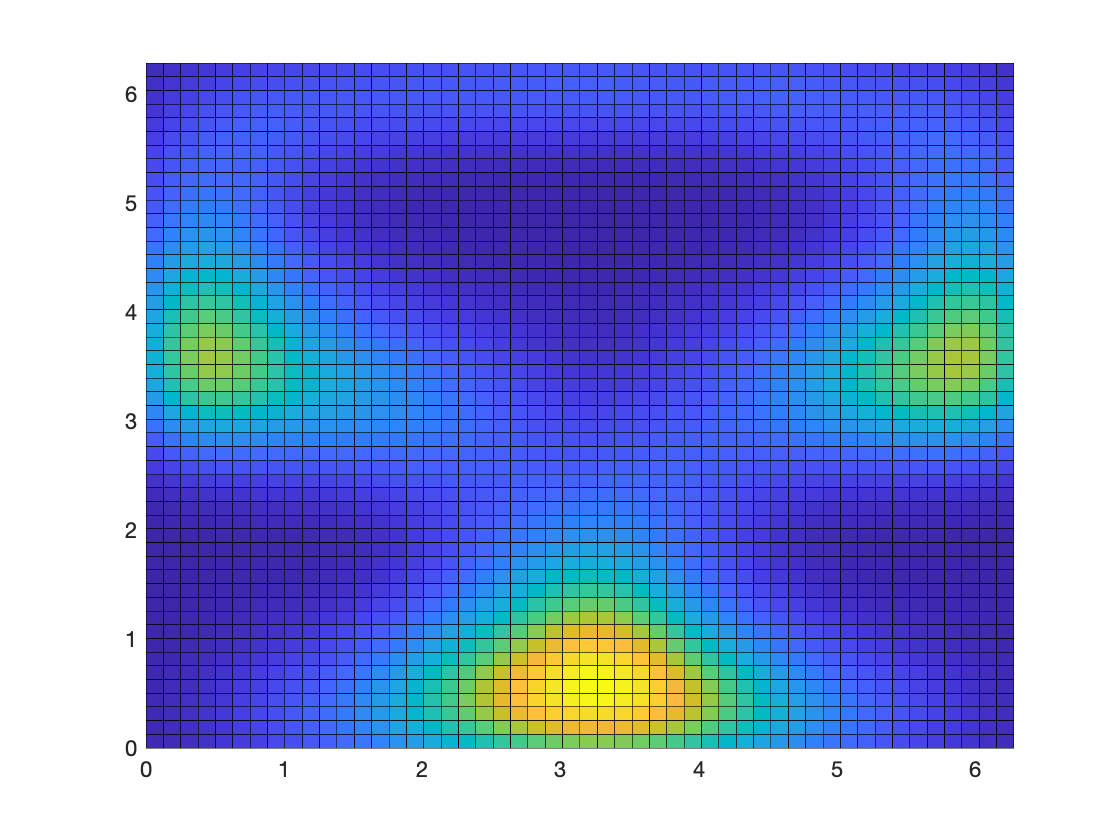}
         \caption{$\epsilon = 0,\ z = (0, 1)$}
         \label{0y2piEI}
         \end{subfigure}
         \begin{subfigure}{0.4\textwidth}
         \includegraphics[width = \linewidth,height = 3.5cm]{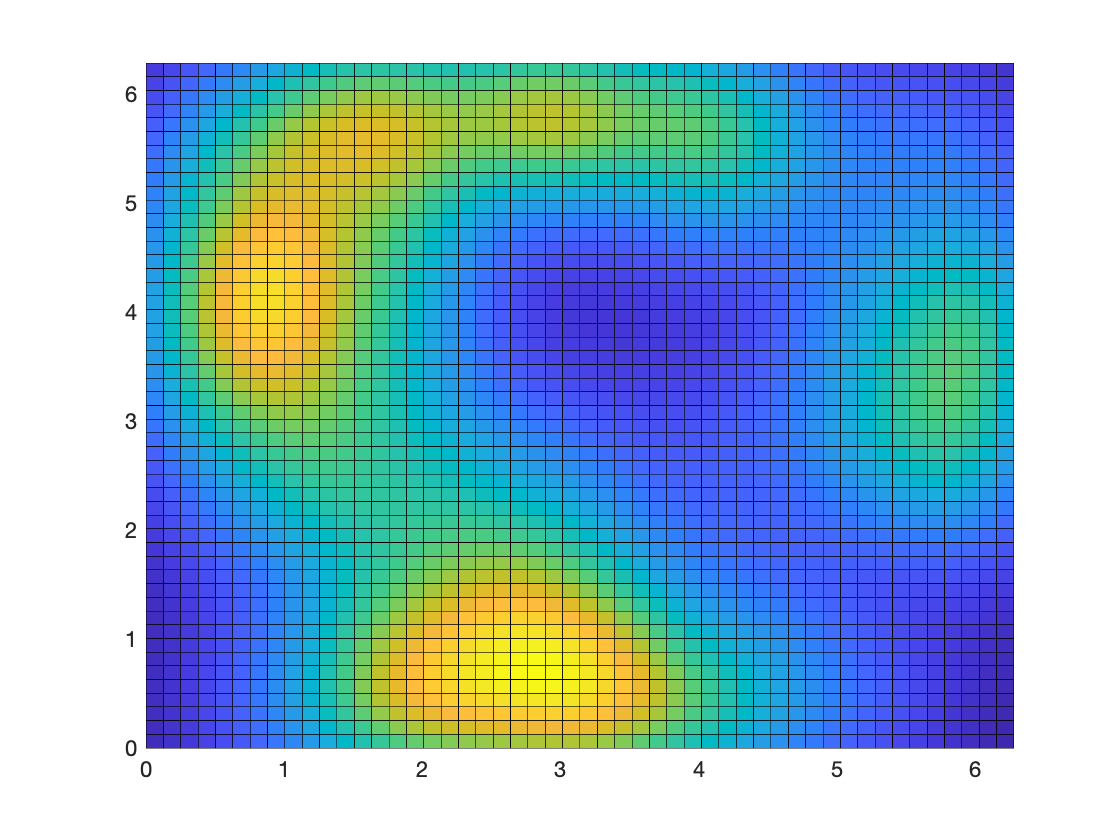}
         \caption{$\epsilon = 0.5,\ z = (0, 1)$}
         \label{05y2piEI}
         \end{subfigure}
	\caption{Comparison of invariant measure in 2D (random) cellular flow. }
		\label{2DCellEICompx}
\end{figure} 
\subsubsection{2D random shear flow}
\noindent
In this part, we changed the velocity field into the generated random shear flow, i.e. 
\begin{equation}
    \bm v = v(x, y) = \delta \cdot (0,\  \xi(x)).
\end{equation}

With this random shear velocity field flow setting, we consider the front speed $c^*$ with direction $\bm z = (0, 1)$. After searching for the direction $\bm e$, we found that $\bm e = (1,0)$ can reach the front speeds. For $\bm e = (1, 0)$, we can analyze this case from the system
\begin{equation}
    q_t = \mathcal{A}q \overset{\text{def}}{=} \kappa \Delta_{\bm x} q + (-2\kappa \lambda \bm e + \bm v)\cdot \nabla_{\bm x} q + (\kappa \lambda^2 - \lambda \bm v \cdot \bm e + f'(0))q.
\end{equation}

Decompose the operator $\mathcal{A}$ into $\mathcal{A = L + C}$ where,
\begin{equation}
    \mathcal{L} = \kappa \Delta_{\bm x} + (-2\kappa \lambda \bm e + \bm v)\cdot \nabla_{\bm x}\ \text{and}\ 
    \mathcal{C} =(\kappa \lambda^2 - \lambda \bm v \cdot \bm e + f'(0)).
\end{equation}

When $\bm e = (1,0)$, we can see that $\bm v \cdot \bm e = 0$, the potential is always a constant. Only has a shift effect on particles and reflects on the front speed $c^*$. So no matter how we change the amplitude of the velocity field $\delta$, the front speed $c^*$ will always remain at a constant. 

So here we only consider the front speed $c^*$ with $\bm z = (0, 1)$ in this random shear flow setting. Additionally, at this time we do not limit the domain $\mathcal{D}$, but enlarging it into infinity ($\mathbb{R}^2$). In the following experiments with an unbounded domain setting, we initialize the particles with uniform distribution on $\mathcal{D}$. Then the particles can spread among unbounded domains and we give them sufficient time to move and spread out. By realizing the random shear velocity field with different random seeds, we compute the KPP front speeds $c^*$, and the results are shown in Figure \ref{2DshycRComplogc}.

\begin{figure}[h]
    \centering
    \includegraphics[width = 0.75\textwidth]{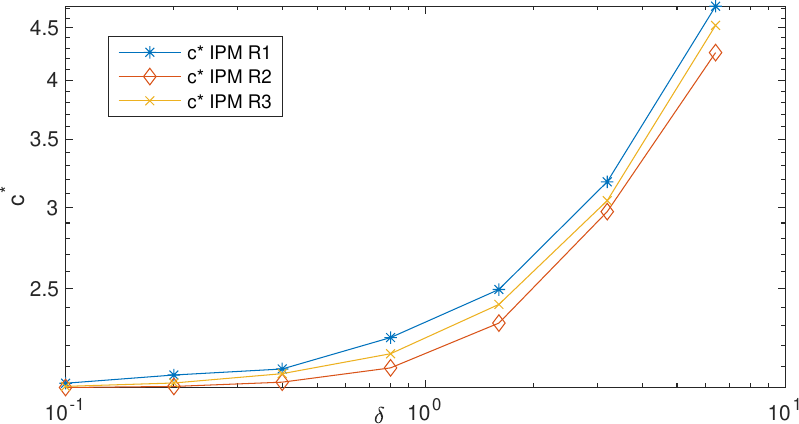}
    \caption{Log-log plot of $c^*$ vs. $\delta$ , $z = (0,1)$, 2D random shear flow by IPM.}
    \label{2DshycRComplogc}
\end{figure}

\subsubsection{3D random ABC flow}
\noindent
Here we change the velocity field into 3D ABC flow with random perturbation, i.e. the velocity field is changed into 
\begin{equation}
    v(x,y,z) = \delta \cdot \big(\sin(z) + \cos(y),\ \sin(x) + \cos(z) + \epsilon \cdot \xi(x),\  \sin(y) + \cos(x)\big).
\end{equation}

And here we set $n = 512,\ \Delta t = 2^{-11}$, $N = 100,000$ and $ \epsilon = 0$ firstly. Then we compute the front speed $c^*$ with three realizations and take their average. 
\begin{figure}[h]
    \centering
    \includegraphics[width = 0.7\textwidth]{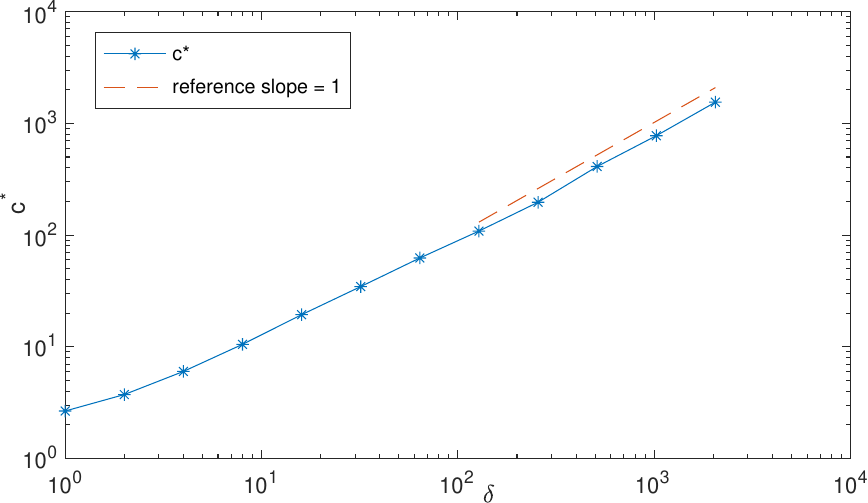}
    \caption{Log-log plot of $c^*$ vs. $\delta$ , $\epsilon = 0, z = (0,1,0)$, 3D ABC flow.}
    \label{0y3DABClogc}
\end{figure}

From Figure \ref{0y3DABClogc} we can see the result that the fitted curve's slope in log scale at relative large $\delta$ is 0.9842, i.e. $c^* = O(\delta^{0.9842}) \approx O(\delta^1)$. This means the fitted curve's slope is close to the reference curve with slope $= 1$ in log scale at relatively large $\delta$. This result is consistent with the result in \cite{shen2013finite3d}. 

Now if we change the random coefficient $\epsilon$ to make it not equal to 0 and compute the front speed $c^*$ corresponding to these cases. Then we will find in Figure \ref{3DABCyEpsComp} that, with the addition of a random term and increasing its amplitude, $c^*$ in 3D ABC flow will get smaller and smaller. This is because the random term gradually destroys the principal vortex structure \cite{DFGHMS} (ballistic orbits, \cite{XYZ,KLX_21,MXYZ_16}) in the ABC flow.

\begin{figure}[h]
    \centering
    \includegraphics[width = 0.75\textwidth]{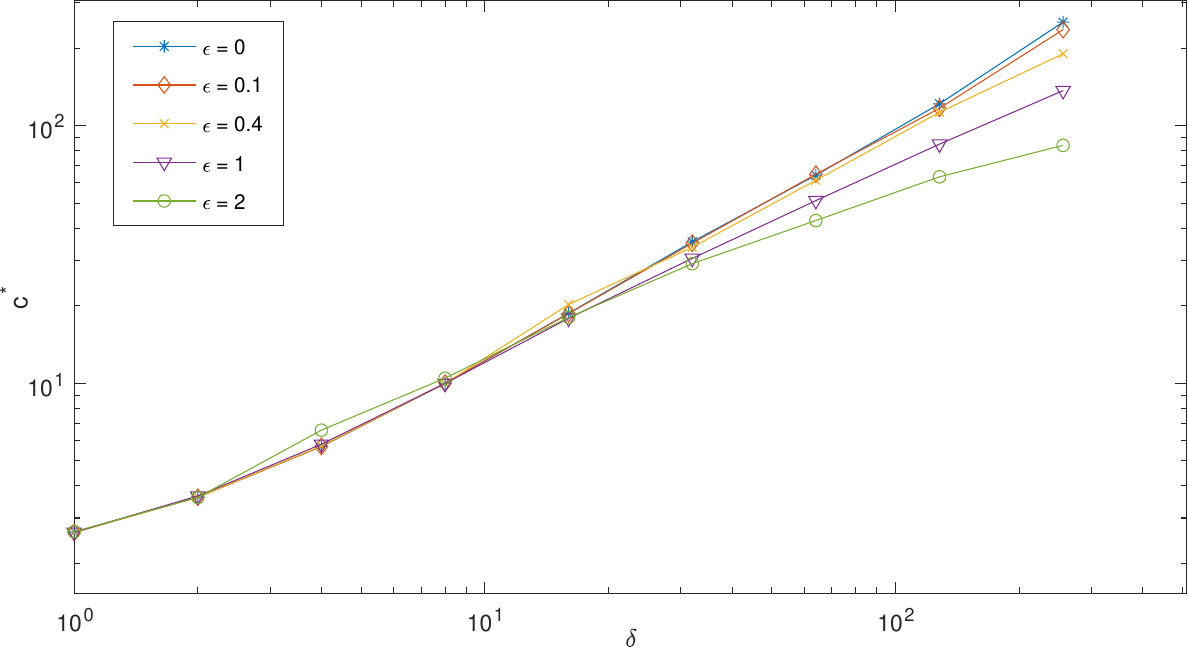}
    \caption{Log-log plot of $c^*$ vs. $\delta$ with different $\epsilon$, 3D random ABC flow.}
    \label{3DABCyEpsComp}
\end{figure}

\subsubsection{3D random cellular flow}
\noindent
Here we consider the 3D cellular flow with random perturbation, the velocity field is changed into $v(x,y,z) =$
\begin{equation}
     \delta \cdot \big(-\sin(x)\cos(y)\cos(z),\ -\sin(y)\cos(x)\cos(z) + \epsilon \cdot \xi(x),\ 2\sin(z)\cos(x)\cos(y)\big).
\end{equation}

Here we set $n = 512,\ \Delta t = 2^{-11}$, $N = 100,000,\ \epsilon = 0, 0.1, 0.4, 1, 2$ and compute their corresponding $c^*$.

\begin{figure}[h]
    \centering
    \includegraphics[width = 0.75\textwidth]{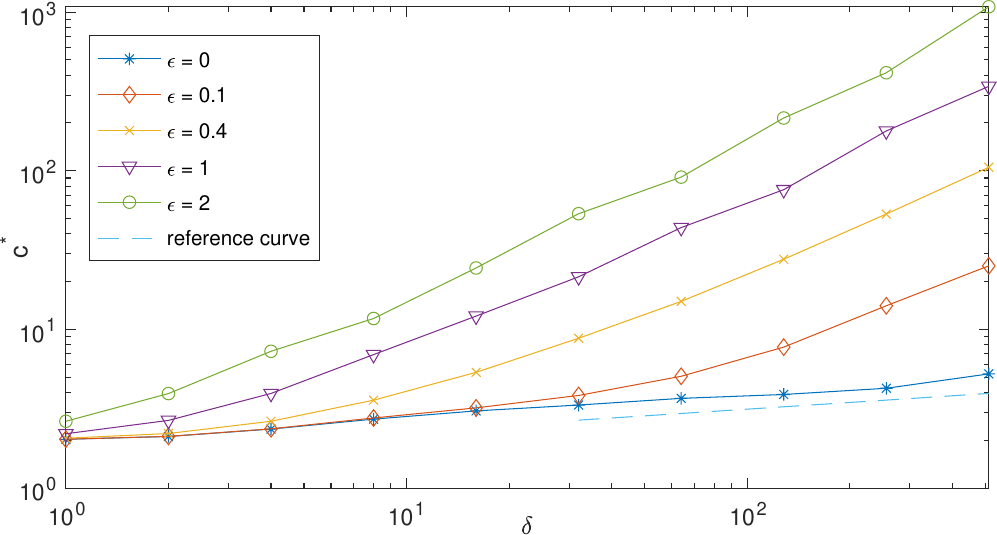}
    \caption{Log-log plot of $c^*$ vs. $\delta$ at different $\epsilon$, 3D random cellular flow.}
    \label{3DCellyEpsComp}
\end{figure}

In Figure \ref{3DCellyEpsComp}, the reference curve has a slope of 0.13. We can see that when $\epsilon = 0$, the fitted curve has the slope 0.1351 and this result is consistent with the result in \cite{shen2013finite3d}. As the random perturbation strengthens by tuning up $\epsilon$, the corresponding 
slopes of the fitted curves and 
$c^*$ values are getting larger. The addition of the random perturbation term destroys the closed streamline structures in 3D cellular flow (reduces trapping) and enhances particle transport thereby $c^*$ increases.

\subsection{IPM with dynamic shift for particles on $\mathbb{R}^d$}
By removing the boundary restriction step in Algorithm \ref{Algo::ForTimeindependent}, IPM can be readily implemented for computing $c^*$ on $\mathbb{R}^d$. In order to investigate whether the $c^*$ computed by IPM with gradually increasing domain sizes will converge to unbounded domain results, we conduct experiments on the randomly perturbed 2D cellular flow\eqref{2DRanCellFlow}. Fig. \ref{2DSpDomcon05yP} shows that as the domain size becomes larger, the $c^*$ computed with the same random realization converges to the direct result on the unbounded domain. 

\begin{figure}[h]
    \centering
    \includegraphics[width = 0.75\textwidth, height=5.5cm]{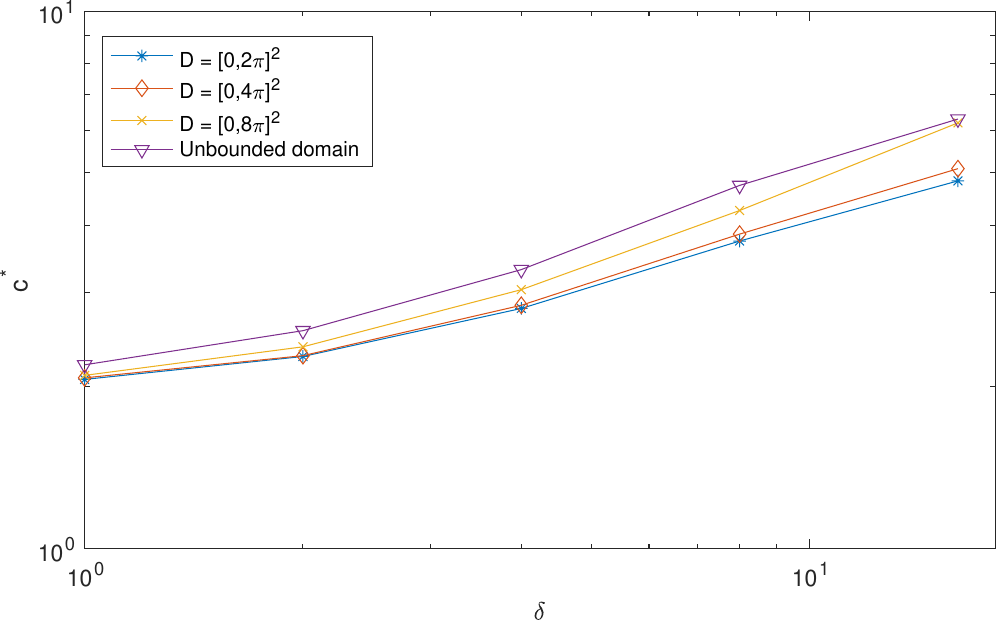}
    \caption{Log-log plot of $c^*$ vs. $\delta$ as domain size increases, $\epsilon = 0.5$, 2D random cellular flow.}
    \label{2DSpDomcon05yP}
\end{figure}

In the convergence analysis of section \ref{sec:ConvergenceAnalysis}, we require 
a bounded domain (e.g. $\mathcal{D}$) to guarantee the existence of the principal eigenvalue and carry out the subsequent $c^*$ error analysis. In a future
theoretical analysis on unbounded domains,  we will work with the Lyapunov exponent as a limit of principal eigenvalue in the infinite volume limit of $\mathcal{D}$ 
(as $\text{vol}\,(\mathcal{D}) \uparrow \infty$)
similar to the periodization approach to stochastic  homogenization \cite{Otto2012,Otto2015}
. Our IPM can be adapted (dropping domain restriction in step 4) to compute particle distribution on unbounded domains ($\mathbb{R}^d$) directly. Different from the $\mathcal{D}$ case, the convergence to an invariant measure requires a dynamic shift to remove a mean flow-induced drift.

The SDE system \eqref{SDEsys}, due to $\bm{b} = -2\kappa \lambda\bm{e} + \bm{v}$ and $\bm{v}$ being mean zero, has a nonzero mean drift $-2\kappa \lambda \bm{e}$. If we do not have particle motion limited on a bounded domain, the particles spread due to nonzero mean drift and do not converge to an invariant measure at large times. To make sense of convergence, we introduce a dynamic shift to center the particle ensemble 
over $j$ generations (each of duration $T$).  
At each $j$, we shift all particles in the ensemble by 
$2 j \kappa \lambda T$. Due to the stationarity of the random flows, 
the statistics of the particle evolution are invariant. 
After the principal eigenvalue approximations have converged, we post-process and map the particle ensemble back 
to a bounded domain for visualization. For example $\mathbb{T}^d = [0, 2\pi]^d$, then map with 
 function $mod(\cdot)$,
the multi-dimensional modulo operation.
To summarize, the particle distribution at the output of the dynamically shifted IPM without domain restriction on $\mathbb{R}^d$ converges to an invariant measure subject to deterministic mapping.

\subsubsection{2D random shear flow}
\begin{equation}
    \bm v = v(x, y) = \delta \cdot (0,\  \xi(x)).
\end{equation}

We analyze IPM performance in the two cases: $\bm e = (1,0)$ and $\bm e =(0,1)$. When $\bm e = (0, 1)$, $\bm v \cdot \bm e = \delta  \cdot \xi(x)$, the eigen-function only depends on $x$. In operator $\mathcal{L}$, $(-2\kappa \lambda \bm e + \bm v) = (0,\ \delta \cdot \xi(x) - 2\kappa \lambda)$ and $\nabla_{\bm x} = (\cdot, 0)$. The inner product of the two terms is zero. This operator is then a 1D random Schrodinger operator in $x$ with well-known localization property. And when $\bm e = (1, 0)$, $\bm v \cdot \bm e = 0$, the potential is always a constant, adding a translation on the particle. To see what is the effect of adding space random perturbation term on the 2D cellular flow, we investigate how the invariant measure (particle distribution) changes after we dynamically shift them. Here we set $\delta = 2^2$, $N = 200,000$, and other parameters remain the same.

\begin{figure}[h]
    \centering
    \includegraphics[width = \textwidth,height=5.8cm]{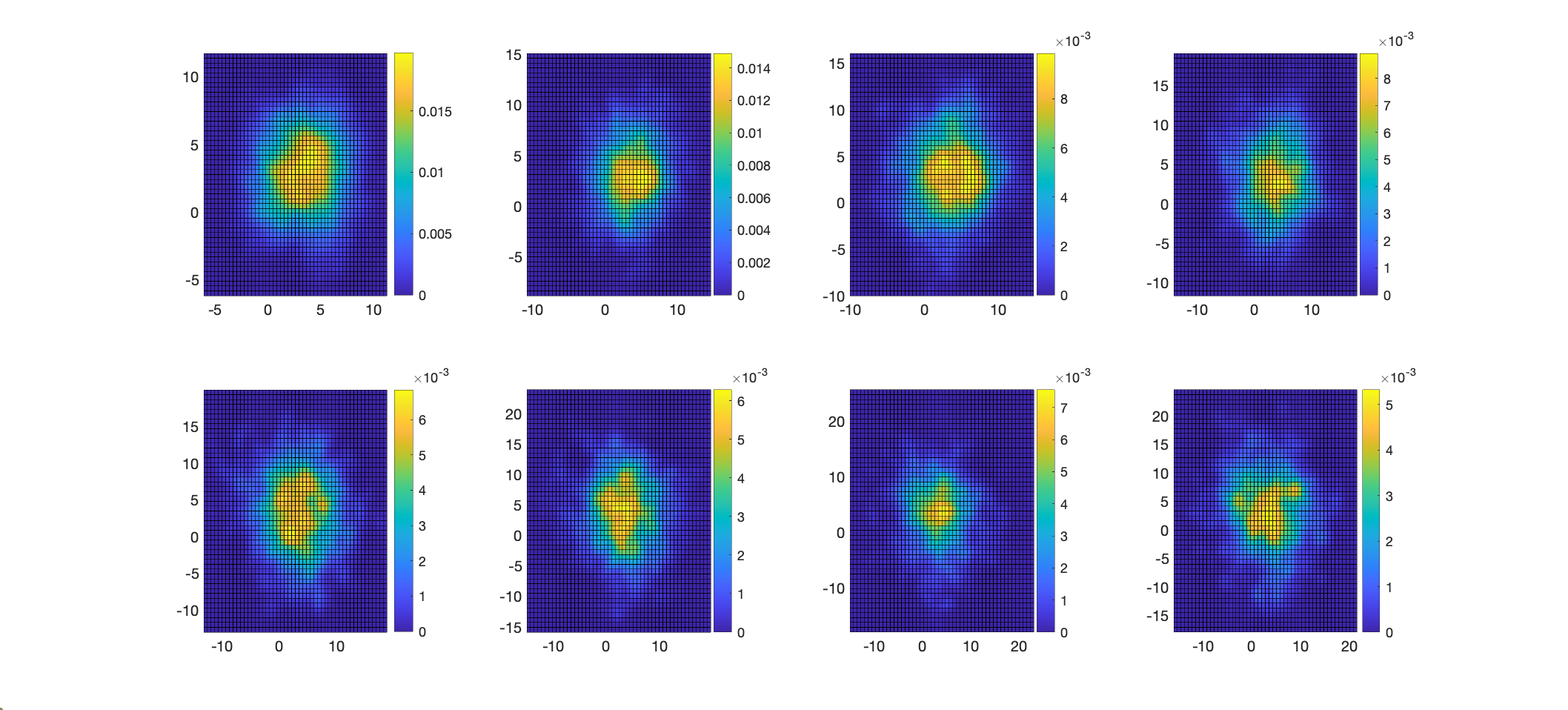}
    \caption{Particles distribution under dynamic shift, $t = 8$ to $t = 64$, $e = (1,0)$}
    \label{2Dsh01xEIT64}
\end{figure}

Here we can see from Figure \ref{2Dsh01xEIT64} 
that, at the early stage of the experiment, the particles are continually spreading outwards from the initial area. A similar case occurs with $\bm e=(0,1)$. For the performance of the particle distribution at long times, we use the re-scaled center $E(\cdot)$ and the second-order momentum $D(\cdot)$ of the particles to express them. We take their average value over ten realizations.
In Figure \ref{2Dsh01xCovMatT1024}, the order of the diffusion in direction $x$ and in direction $y$ are respectively $D(x) = O(t^{0.850})$ and $D(y) = O(t^{0.945})$. 
In Figure \ref{2Dsh01yCovMatT1024}, the order of the diffusion in direction $x$ and in direction $y$ are respectively $D(x) = O(t^{-0.013}) \approx O(t^0)$ and $D(y) = O(t^{0.742})$. 
We also perform common statistical tests of the particle distribution at $t = 1024$ (e.g. QQplot and KS test), and the results all show that the particle distribution is non-Gaussian. In fact, according to the variation of the variance $D(x), D(y)$ with time $t$, it can be seen that the diffusion of the particles after the re-center process exhibits sub-diffusion phenomena. 

\begin{figure}[h]
    \centering
    \begin{subfigure}{0.46\textwidth}
    \includegraphics[width =\textwidth]{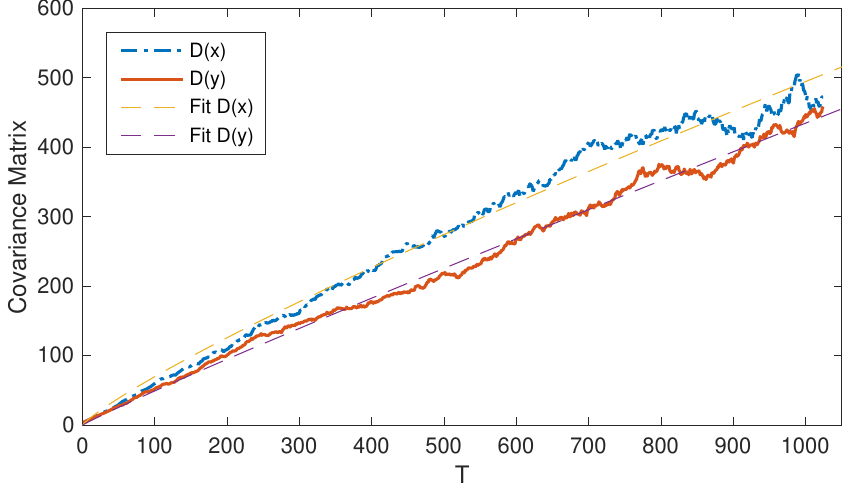}
\caption{$\bm e = (1,0)$}
\label{2Dsh01xCovMatT1024}
\end{subfigure}
\begin{subfigure}{0.46\textwidth}
    \includegraphics[width = \textwidth]{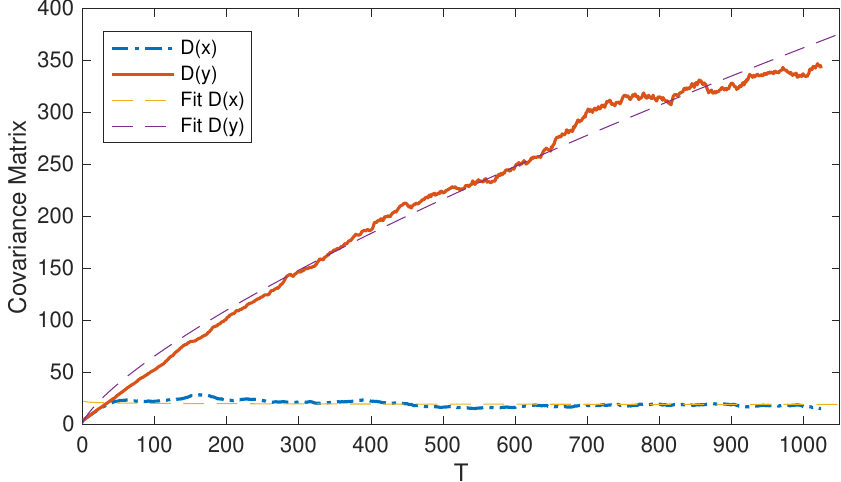}
    \caption{$\bm e = (0,1)$}
    \label{2Dsh01yCovMatT1024}
\end{subfigure}
\caption{Covariance matrix evolution with different $\bm e$, from $t = 0$ to $t = 1024$ }
\end{figure}

\begin{figure}[h]
\centering
\begin{subfigure}{0.46\textwidth}
    \includegraphics[width = \linewidth,height=3.6cm]{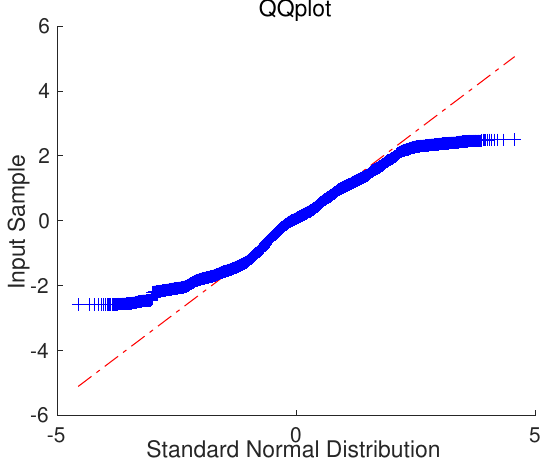}
    \caption{QQ plot in $x$, $\bm e = (1, 0)$}
    \label{qqplotxe10}
\end{subfigure}
\begin{subfigure}{0.46\textwidth}
    \includegraphics[width = \linewidth,height=3.6cm]{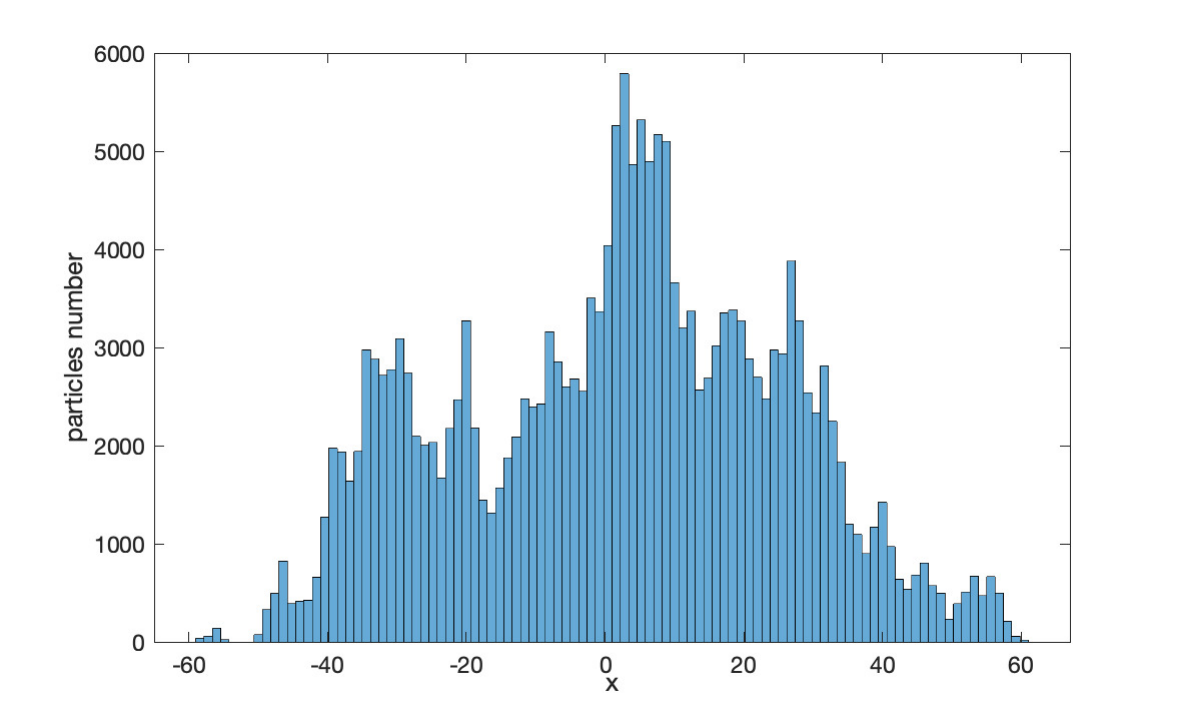}
    \caption{histogram in $x$, $\bm e = (1, 0)$}
    \label{histogram01xcx}
\end{subfigure}
\caption{QQ plot and histogram plot for random shear flow, $\bm e = (1,0), t = 1024$}
\label{qqplote10}
\end{figure}

\subsubsection{3D randomly perturbed ABC flow}
\begin{equation}
    v(x,y,z) = \delta \cdot \big(\sin(z) + \cos(y),\ \sin(x) + \cos(z) + \epsilon \cdot \xi(x),\  \sin(y) + \cos(x)\big).
\end{equation}

Here we plot the distribution of the particles with different random coefficients in $\mathbb{T}^3$. It can be shown from Figure \ref{3DABCyNB3viewCompP}, especially the $xz$-plane that as the coefficient of the random term gradually increases, the radius of the exhibited tube structure becomes larger and larger with the same total number of particles, which means that the distribution of particles becomes more and more dispersed and the tube structure is more and more severely damaged. A similar case also occurs if we change the system with an unbounded domain setting and at the final plot step we restrict the distribution of the particles on $\mathbb{T}^3$.

\begin{figure}[tbph]
    \centering
    \begin{subfigure}{\textwidth}
        \includegraphics[width = \linewidth]{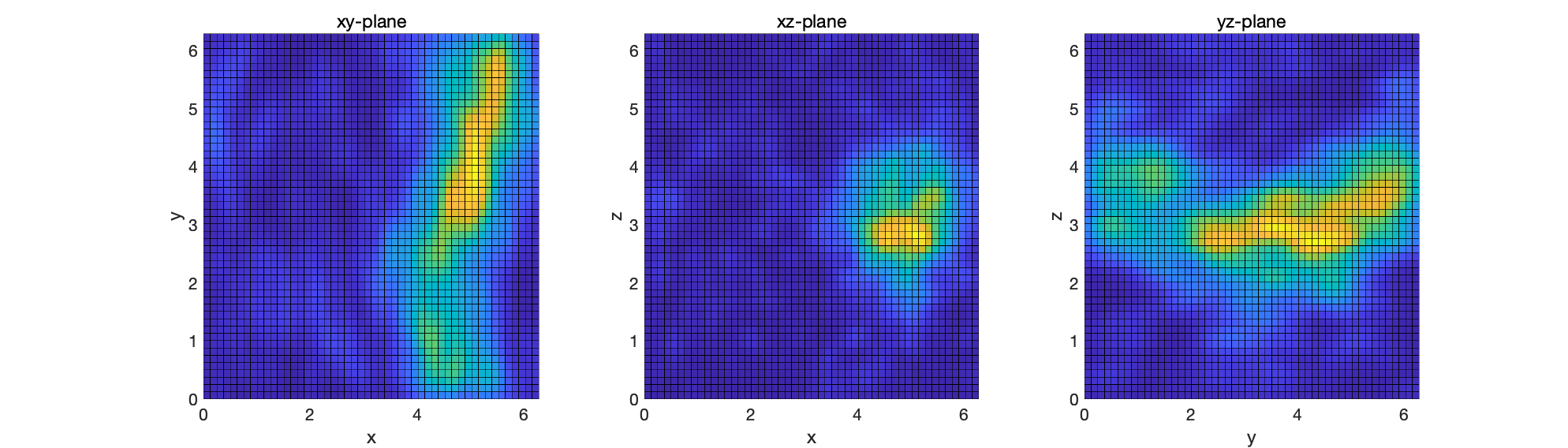}
        \caption{$\epsilon = 0$}
        \label{3DABC0y4pi3view}    
    \end{subfigure}

    \begin{subfigure}{\textwidth}
        \includegraphics[width = \linewidth]{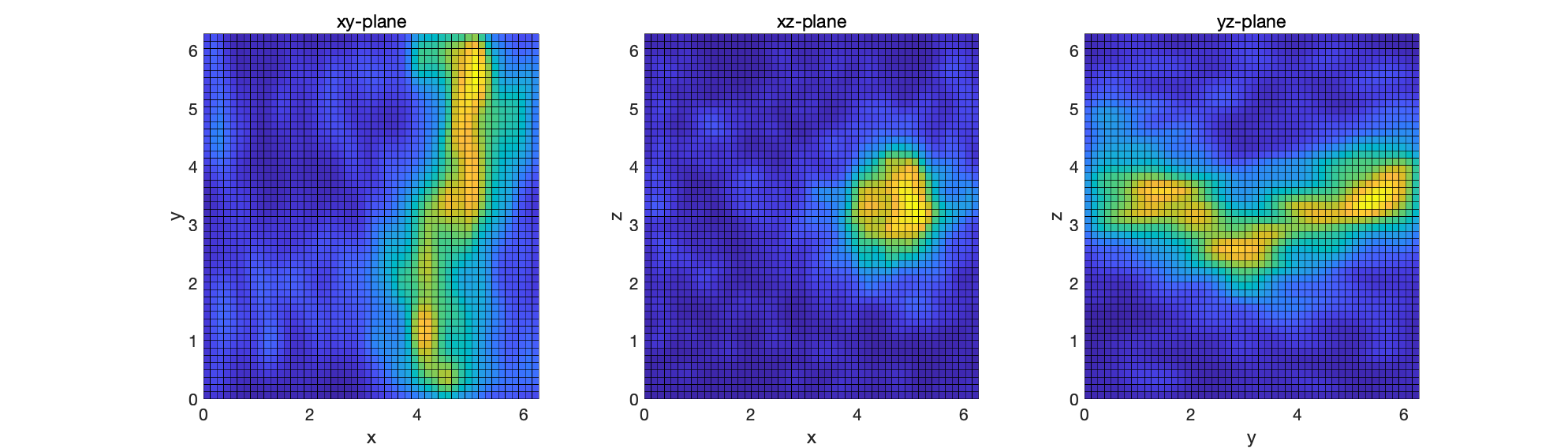}
        \caption{$\epsilon = 0.4$}
        \label{3DABC04y4pi3view}    
    \end{subfigure}

    \begin{subfigure}{\textwidth}
        \includegraphics[width = \linewidth]{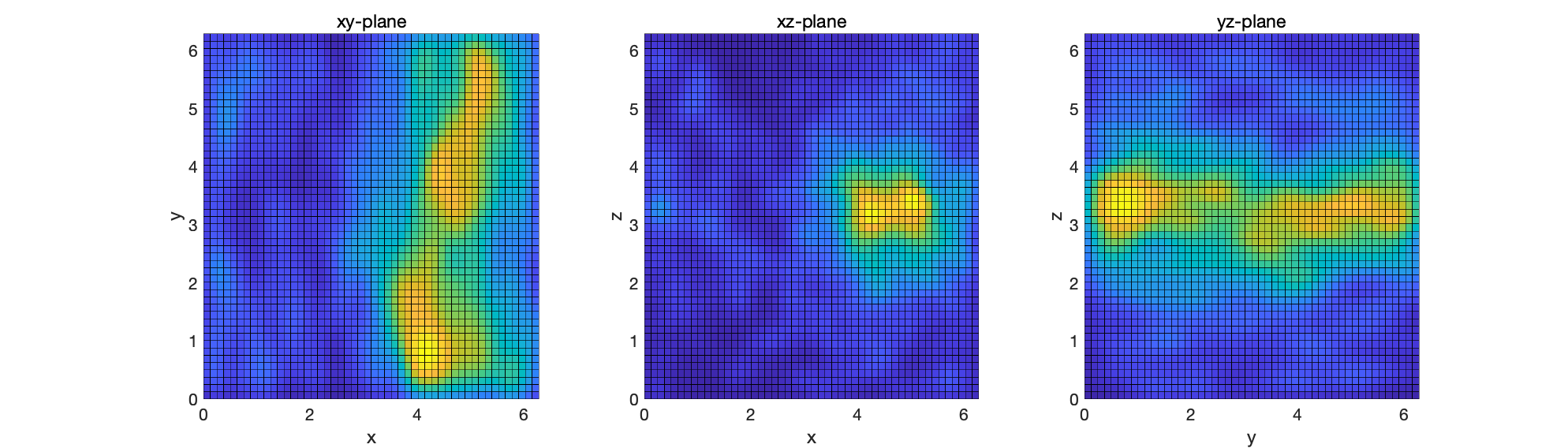}
        \caption{$\epsilon = 2$}
        \label{3DABC2y4pi3view}    
    \end{subfigure}
    
    \caption{Three planar views of particle distributions for randomly perturbed ABC flow at  $\epsilon=0,0.4,2$, showing increased blurring/broadening effect of perturbations.} 
    \label{3DABCyNB3viewCompP}
\end{figure}

\section{Conclusions}\label{sec:Conclusion}
\noindent In this paper, we developed an efficient IPM to compute KPP front speeds $c^*$ in 2D and 3D stochastic flows, with convergence analysis based on the random Fourier method. We first computed $c^*$
in periodic flows, including 2D cellular flow and 3D ABC flow, to reach agreement with semi-Lagrangian and spectral methods. 
Then we computed and analyzed the effect of random perturbations on $c^*$.
As the perturbation strengths increase, 
we observed that the vortical structures (ballistic orbits) \cite{DFGHMS,XYZ,MXYZ_16,KLX_21}  in 3D ABC flow are collapsing, 
resulting in smaller $c^*$. 
Compared with mesh-based methods, IPM 
has several striking advantages: (1) it is scalable  in spatial dimension and can handle stochastic 3D problems at ease; 
(2) it is scalable in domain size and can handle infinite domain as well as Lyapunov exponent approximation directly. In experiments on $\mathbb{R}^3$, the empirical particle distributions from IPM show convergence under dynamic normalization while the average particle fitness approximates the Lyapunov exponent.

\section*{Acknowledgements}
\noindent  JX was partially supported by NSF grant DMS-2309520. ZZ was supported by the National Natural Science Foundation of China (Project 12171406), the Hong Kong RGC grant (Projects G-HKU702/18, 17307921 and 173 04324), the Outstanding Young Researcher award of HKU (2020-21), Seed Funding for Basic Research, and Seed Funding for Strategic Interdisciplinary Research Scheme 2021/22 (HKU).

\appendix  

\section{Convergence of perturbed $\mu$}
\label{Sec:AppendixE3proof}
We review the perturbation theory of linear operators with the following notation:
\begin{equation}
    T(a) = T + a\cdot T'.
\end{equation}
Here $T$ is the original operator, $T'$ is the perturbation operator and $a$ is the coefficient of the perturbation term. Then we have the following property 
\begin{pro}
    Let $T(a)$ be continuous at $a=0$, then the eigenvalues of $T(a)$ are continuous at $a = 0$. If $\lambda$ is an eigenvalue of $T(0) = T$, then the $\lambda$-group is well-defined for sufficiently small $|a|$ and the total projection $P(a)$ for the $\lambda$-group is continuous at $a=0$. 
    \label{PerbCtsEigTx}
\end{pro}

Next, we would like to introduce some notations we use to describe the convergence between operators in the generalized sense.
\begin{dfn}
\label{Sec:ApendixDefdist}
Consider closed linear manifolds $M$ and $N$ of a Banach space $Z$. We use $S_M$ to denote the unit sphere of $M$ (the set of all $v \in M$ which satisfy $||v|| = 1$). For any two closed linear manifolds $M, N$ of $Z$, we define
\begin{equation}
    \delta(M, N) = \sup_{u \in S_M} dist(u,M),
\end{equation}
\begin{equation}
    \tilde{\delta}(M,N) = \max [\delta(M, N),  \delta(N, M)]^1.
\end{equation}
\end{dfn}

\begin{dfn}
    We denote that $T_n$ converge to $T$ in the generalized sense if $\tilde{\delta} (T_n, T) \to 0$.
    \label{defGenCon}
\end{dfn}

There is a proposition that is a sufficient condition for the generalized convergence and its proof can be found in \cite{kato2013perturbation}.
\begin{pro}
    Let $T \in \mathcal{C}(X, Y)$. Let $A_n, n=0,1,2,\cdots$, be $T$-bounded so that $||A_n v|| \leq a_n ||v|| + b_n ||Tv||$ for $v \in D(T) \subset D(A_n)$. If $a_n \to 0$ and $b_n \to 0$, then $T_n = T + A_n \in \mathcal{C}(X,Y)$ for sufficiently large $n$ and $T_n \to T$ in the generalized sense. 
    \label{SuffCondGenCon}
\end{pro}

Next, we would like to review some results which can make the above propositions applicable to an infinite system. Here we use $\Sigma(T)$ to denote the spectrum of the operator $T \in \mathcal{C}(X)$ and we suppose that the spectrum $\Sigma(T)$ has an isolated point $\lambda$. Then obviously we can divide $\Sigma(T)$ into two parts $\Sigma'(T)$ and $\Sigma''(T)$. Here $\Sigma'(T)$ consists of the single point $\lambda$ and any simple closed curve $\Gamma$ enclosing $\lambda$ but no other point in $\Sigma(T)$ may be chosen. Following this idea, we can divide the spectrum of an infinity system and focus on the finite system $\Sigma'(T)$ which contains the principal eigenvalue and we have the following proposition to analyze the continuity of a finite system of eigenvalues.
\begin{pro}
    The change of a finite system $\Sigma'(T)$ of eigenvalues of a closed operator $T$ is small (in the sense of \ref{PerbCtsEigTx}) when $T$ is subjected to a small perturbation in the sense of generalized convergence. 
    \label{RelaPerOpandEig}
\end{pro}

\section{Proof of Lemma \ref{Lemma3_2}}
\label{sec::AppendixUpperBound}
\noindent 
In the convergence analysis in \cite{Junlong2022}, the author assumed that the deterministic $\bm b_{0}(\bm x)$ and $c_{0}(\bm x)$ satisfy that $||\bm b_{0}(\bm x)||_{L^2} \leq M_1$, $|c_{0}(\bm x)| \leq M_2$, $||\nabla_{\bm x} c_{0}(\bm x)||_{L^2} \leq M_3$ and $||\Delta_{\bm x} c_{0}(\bm x)||_{L^2} \leq M_4$. Now we analyze these assumptions with the addition of random velocity field $\bm v$
and we use $\tilde{\bm v}$ to be an approximation to $\bm v$ in the practical experiments. 
\begin{proof}
    
By the triangle inequality, we can get that
\begin{align}
    ||\bm b(\bm x,\omega)||_{L^2} &\leq ||\bm b_{0}(\bm x)||_{L^2} + ||\bm b(\bm x,\omega) - \bm b_{0}(\bm x)||_{L^2}  \notag \\
    &\leq M_1 + ||\bm b(\bm x,\omega) - \bm b_{0}(\bm x)||_{L^2}.
    \label{b_rxL2normApp}
\end{align}

We give the analysis for 2D cases as an example, a similar analysis can be done for higher dimensions.  
We assume that deterministic part of $\bm v$ with period $1$ is denoted as $\bm v_0(\bm x)$. Thus, both $\bm v$ and $\tilde{\bm v}$ have a period $1/\Delta k \in \mathbb{N}^+$ in each spatial dimension. We define $\mathcal{D} = [0, \frac{1}{\Delta k} ]^2$ and assume that both $\bm v$ and $\tilde{\bm v}$ are locally bounded in the bounded domain $\mathcal{D}$. Next, for the term $||\bm b(\bm x,\omega) - \bm b_{0}(\bm x)||_{L^2}$, we can obtain
\begin{equation}
    \mathbb{E}_{\omega}[||\bm b(\bm x,\omega) - \bm b_{0}(\bm x)||^2_{L^2}] 
    = \mathbb{E}_{\omega}[\int_{\mathcal{D}} |\bm v(\bm x,\omega) - \bm v_0(\bm x)|^2 d\bm x]. 
    \label{equb2}
\end{equation}

We denote $\widehat{\bm v}(\bm k,\omega) \overset{\text{def}}{=} \mathcal{F}[\bm v](\bm k), \widehat{\bm v}_0(\bm k) \overset{\text{def}}{=} \mathcal{F}[\bm v_0](\bm k)$, the Fourier transform of $\bm v(\bm x,\omega)$ and $\bm v_0(\bm x)$. From the definition of energy spectral density, we know that $E(\bm k) = \mathbb{E}_{\omega}[|\widehat{\bm v}(\bm k,\omega) - \widehat{\bm v}_0(\bm k)|^2]$. Due to the spectrum we consider are all isotropic and have finite moments, by applying the Plancherel theorem, we can get the following relationship.
\begin{equation}
   \mathbb{E}_{\omega}[||\bm b(\bm x,\omega) - \bm b_{0}(\bm x)||^2_{L^2}] \leq 2\pi \sum_{j = 0}^{+\infty} k_j E(k_j) \Delta k \leq m_1,
   \label{equb3}
\end{equation}
where $m_1$ is a finite number due to $E(k)$ decay exponentially as $k \to \infty$. Together \eqref{equb2} and \eqref{equb3}, we can get $\mathbb{E}_{\omega}[||\bm b(\bm x,\omega)||_{L^2}] \leq M_1 + \sqrt{m_1}
$. The variance satisfies $Var(||\bm b(\bm x,\omega)||_{L^2})$ $ \leq \mathbb{E}_{\omega}[||\bm b(\bm x,\omega) - \bm b_{0}(\bm x)||_{L^2}^2] \leq 
m_1$. Then by applying Chebyshev's inequality, for a probability $p$ close to 1, we can find the corresponding upper bound $\tilde{M}_{1}$, so that $P(||\bm b(\bm x,\omega)||_{L^2} < \tilde{M}_{1}) > p$.
Similar analysis can be done for $||\tilde{\bm b}(\bm x, \omega)||_{L^2} = ||-2\kappa \lambda \bm e + \tilde{\bm v}||_{L^2}$ and we can get an upper bound $\tilde{M}_{1}'$. Then we take $M_1^* = \max(\tilde{M}_{1},$ $ \tilde{M}_{1}')$.  Then for the second assumption, we have 
\begin{equation}
    |c(\bm x, \omega)| 
    \leq |c_{0}(\bm x)| + \lambda \cdot ||\bm v(\bm x,\omega) - \bm v_0(\bm x)||_{L^2} \leq M_2 + \lambda \cdot ||\bm v(\bm x,\omega) - \bm v_0(\bm x)||_{L^2} .
\end{equation}

Then we can get $\mathbb{E}_{\omega}[|c(\bm x,\omega)|] \leq M_2 + \sqrt{m_1}$. The variance $Var(|c(\bm x,\omega)|) \leq \mathbb{E}_{\omega}[||\bm v(\bm x,\omega) - \bm v_0(\bm x)||_{L^2}^2] \leq m_1$. By applying Chebyshev's inequality, for a probability $p$ close to 1, we can find the corresponding upper bound $\tilde{M}_{2}$, so that $P(|c(\bm x,\omega)| < \tilde{M}_{2}) > p$. Similar analysis can be done for $|\tilde{c}(\bm x, \omega)|$, and we can get an upper bound $\tilde{M}_{2}'$. Then we take $M_2^* = \max(\tilde{M}_{2}, \tilde{M}_{2}')$. Then for the third and fourth assumptions, we have 
\begin{align}
    ||\nabla_{\bm x} c(\bm x,\omega)||_{L^2} &\leq ||\nabla_{\bm x} c_{0}(\bm x)||_{L^2} + \lambda \cdot ||\nabla_{\bm x} (\bm v(\bm x,\omega) - \bm v_0(\bm x))||_{L^2} \\ \notag
    &\leq M_3 + \lambda \cdot ||\nabla_{\bm x} (\bm v(\bm x,\omega) - \bm v_0(\bm x))||_{L^2}, 
\end{align}
\begin{align}
    ||\Delta_{\bm x} c(\bm x,\omega)||_{L^2} &\leq ||\Delta_{\bm x} c_{0}(\bm x)||_{L^2} + \lambda \cdot ||\Delta_{\bm x} (\bm v(\bm x,\omega) - \bm v_0(\bm x))||_{L^2} \\ \notag 
    &\leq M_4 + \lambda \cdot ||\Delta_{\bm x}(\bm v(\bm x,\omega) - \bm v_0(\bm x))||_{L^2} .
\end{align}

Then for the term $||\nabla_{\bm x} (\bm v(\bm x,\omega) - \bm v_0(\bm x))||_{L^2}$ and $||\Delta_{\bm x} (\bm v(\bm x,\omega) - \bm v_0(\bm x))||_{L^2}$ , we have 
\begin{equation}
    ||\nabla_{\bm x} (\bm v(\bm x,\omega) - \bm v_0(\bm x))||_{L^2}^2 
    = \int_{\mathcal{D}} |\nabla_{\bm x}(\bm v(\bm x,\omega) - \bm v_0(\bm x))|^2 d\bm x 
\end{equation}
\begin{equation}
    ||\Delta_{\bm x} (\bm v(\bm x,\omega) - \bm v_0(\bm x))||_{L^2}^2 
    = \int_{\mathcal{D}}|\Delta_{\bm x} (\bm v(\bm x,\omega) - \bm v_0(\bm x))|^2 d\bm x 
\end{equation}

By the properties of the Fourier transform and the Plancherel theorem, we can obtain that
\begin{align}
    \mathbb{E}_{\omega}[||\nabla_{\bm x} (\bm v(\bm x,\omega) - \bm v_0(\bm x))||_{L^2}^2 ] 
    &\leq 2\pi \sum_{j = 0}^{\infty} k_j^3 E(k_j) \Delta k < \infty,
\end{align}
\begin{align}
   \mathbb{E}_{\omega}[||\Delta_{\bm x} (\bm v(\bm x,\omega) - \bm v_0(\bm x))||_{L^2}^2 ] 
    &\leq 2\pi \sum_{j = 0}^{\infty} k_j^5 E(k_j) \Delta k < \infty,
\end{align}

Then we can obtain that 
\begin{equation}
    \mathbb{E}_{\omega}[||\nabla_{\bm x} c(\bm x, \omega)||_{L^2}] \leq M_3 + 2\pi \sum_{j = 0}^{\infty} k_j^3 E(k_j) \Delta k
\end{equation}
\begin{equation}
     \mathbb{E}_{\omega}[||\Delta_{\bm x} c(\bm x, \omega)||_{L^2}] \leq M_4 + 2\pi \sum_{j =0}^{\infty} k_j^5 E(k_j) \Delta k
    \label{UpperBoundforAss3Ass4}
\end{equation}

And we also know the upper bound of their variance. Following the same idea, we can get $\tilde{M}_{3}$ and $\tilde{M}_{4}$, the upper bound hold with probability greater than $p$. Following the same idea, we can get $\tilde{M}_{3}'$ and $\tilde{M}_{4}'$. Then we take $M_3^* = \max(\tilde{M}_{3}, \tilde{M}_{3}')$ and $M_4^*= \max(\tilde{M}_{4},\tilde{M}_{4}')$. We have already computed the assumptions' upper bound distributions after adding a random perturbation with specific energy spectral density. 
\end{proof}
\section{Proof for Theorem \ref{Thm::e3_5}}
\label{Sec::AppendixTheorem3_5}
\begin{proof}
By the operators defined in Lemma \ref{Lemmma:Defpandupperbound}, we can get that
\begin{equation}
    \mathcal{U} (T,0) - \prod_{h=1}^H e^{\Delta t_h \tilde{\mathcal{L}} (\bm x,\omega)} e^{\Delta t_h \tilde{\mathcal{C}} (\bm x,\omega)} = e^{\gamma T} \big( \mathcal{U}_{\gamma} (T,0) - \prod_{h=1}^H e^{\Delta t_h \tilde{\mathcal{L}}_{\gamma}} e^{\Delta t_h \tilde{\mathcal{C}}_{\gamma}}\big).
\end{equation}

By the telescopic sum argument, we obtain that for any $q \in \mathcal{D}$,

\begin{align}
    e_3 = &||\prod_{h=1}^H e^{\Delta t_h \mathcal{L} (\bm x,\omega)} e^{\Delta t_h \mathcal{C} (\bm x,\omega)}q - \prod_{h=1}^H e^{\Delta t_h \tilde{\mathcal{L}} (\bm x,\omega)} e^{\Delta t_h \tilde{\mathcal{C}} (\bm x,\omega)} q ||_{L^2} \notag \\
     = & e^{\gamma T} \big| \big| \sum_{j=1}^H \prod_{h = j+1}^H e^{\Delta t_h \mathcal{L}_{\gamma}} e^{\Delta t_h \mathcal{C}_{\gamma}} (e^{\Delta t_j \mathcal{L}_{\gamma}} e^{\Delta t_j \mathcal{C}_{\gamma}} - e^{\Delta t_j \tilde{\mathcal{L}}_{\gamma}} e^{\Delta t_j \tilde{\mathcal{C}}_{\gamma}}) \prod_{i=1}^{j-1} e^{\Delta t_i \tilde{\mathcal{L}}_{\gamma}} e^{\Delta t_i \tilde{\mathcal{C}}_{\gamma}} q  \big| \big|_{L^2} \notag \\
    \leq & e^{\gamma T}||\sum_{j=1}^H (e^{\Delta t_j \mathcal{L}_{\gamma}} e^{\Delta t_j \mathcal{C}_{\gamma}} - e^{\Delta t_j \tilde{\mathcal{L}}_{\gamma}} e^{\Delta t_j \tilde{\mathcal{C}}_{\gamma}}) q ||_{L^2}.
    \label{ErrEst3part}
\end{align}

To simplify, we denote $\mathcal{L}_{\gamma} =\mathcal{L}_{\gamma}(\bm x,\omega)= \mathcal{L}(\bm x,\omega) - \tilde \gamma$ and $\mathcal{C}_{\gamma} = \mathcal{C}_{\gamma}(\bm x, \omega) = \mathcal{C}(\bm x,\omega) - M_2^*$. This step is to guarantee that $||e^{\Delta t \mathcal{L}_{\gamma}(\bm x,\omega)}||_{L^2} \leq 1$ and $||e^{\Delta t \mathcal{C}_{\gamma}(\bm x,\omega)}||_{L^2}\leq 1$ are both satisfied with probability greater than $p$. Same for $\tilde{\mathcal{L}}_{\gamma}$ and $\tilde{\mathcal{C}}_{\gamma}$. By expressing the term $e^{\Delta t \mathcal{C}_{\gamma}}$ and $e^{\Delta t \tilde{\mathcal{C}}_{\gamma}}$ into exponential series, we can obtain
\begin{equation}
    e^{\Delta t \mathcal{L}_{\gamma}} e^{\Delta t \mathcal{C}_{\gamma}} q =  e^{\Delta t \mathcal{L}_{\gamma}}q  +  (\Delta t) e^{\Delta t \mathcal{L}_{\gamma}} \mathcal{C}_{\gamma} q + r_1(\bm x,\omega) q,
    \label{TaylorexpL'}
\end{equation}
\begin{equation}
    e^{\Delta t \tilde{\mathcal{L}}_{\gamma}} e^{\Delta t \tilde{\mathcal{C}}_{\gamma}} q =  e^{\Delta t \tilde{\mathcal{L}}_{\gamma}} q +  (\Delta t)e^{\Delta t \tilde{\mathcal{L}}_{\gamma}} \tilde{\mathcal{C}}_{\gamma} q + r_{1}'(\bm x, \omega) q.
\end{equation}
Here the residual terms $||r_1(\bm x,\omega)||_{L^2} \leq \frac{1}{2} (\Delta t)^2 ||\mathcal{C}_{\gamma}(\bm x,\omega)||^2_{L^2}, ||r_{1}'(\bm x, \omega)||_{L^2} $ $\leq \frac{1}{2} (\Delta t)^2$ $||\tilde{\mathcal{C}}_{\gamma}(\bm x,\omega)||^2_{L^2}$. Then we denote $r_1 = r_1(\bm x,\omega),\ r_{1}'=r_{1}'(\bm x, \omega)$ and let the two equations above subtract
\begin{align}
    LHS & = (e^{\Delta t \mathcal{L}_{\gamma}} - e^{\Delta t \tilde{\mathcal{L}}_{\gamma}})q + (\Delta t)e^{\Delta t \mathcal{L}_{\gamma}} \mathcal{C}_{\gamma} q - (\Delta t)e^{\Delta t \tilde{\mathcal{L}}_{\gamma}} \tilde{\mathcal{C}}_{\gamma} q + r_1 q - r_{1}' q \notag \\
    & = (e^{\Delta t \mathcal{L}_{\gamma}} - e^{\Delta t \tilde{\mathcal{L}}_{\gamma}})\big(1 + \Delta t\mathcal{C}_{\gamma}\big) q + (\Delta t)e^{\Delta t \tilde{\mathcal{L}}_{\gamma}} (\mathcal{C}_{\gamma} - \tilde{\mathcal{C}}_{\gamma}) q + (r_1 - r_{1}') q.
    \label{eLFeCFest}
\end{align}

Then we define $\mathcal{G}(\bm x,\omega) = \mathcal{L}_{\gamma}(\bm x,\omega) - \tilde{\mathcal{L}}_{\gamma}(\bm x,\omega)$ and $\mathcal{H}(\bm x,\omega) = \mathcal{C}_{\gamma}(\bm x,\omega) - \tilde{\mathcal{C}}_{\gamma}(\bm x,\omega)$. By the freezing coefficient formula, we obtain
\begin{align}
    e^{\Delta t \mathcal{L}_{\gamma}(\bm x,\omega)} - e^{\Delta t \tilde{\mathcal{L}}_{\gamma}(\bm x,\omega)} = & e^{\Delta t \tilde{\mathcal{L}}_{\gamma}(\bm x,\omega)} + \int_{0}^{\Delta t} e^{s \tilde{\mathcal{L}}_{\gamma}(\bm x,\omega)} \mathcal{G}(\bm x,\omega) e^{(\Delta t - s) \mathcal{L}_{\gamma}(\bm x,\omega)} ds  \notag \\
    & - e^{\Delta t \tilde{\mathcal{L}}_{\gamma}(\bm x,\omega)} = \int_{0}^{\Delta t} e^{s \tilde{\mathcal{L}}_{\gamma}(\bm x,\omega)} \mathcal{G}(\bm x,\omega) e^{(\Delta t - s) \mathcal{L}_{\gamma}(\bm x,\omega)} ds.
    \label{FreeCoeffofL_LF}
\end{align}

We denote $\mathcal{G} = \mathcal{G}(\bm x,\omega),\ \mathcal{H} = \mathcal{H}(\bm x, \omega)$. Then we substitute the above equation \eqref{FreeCoeffofL_LF} into \eqref{eLFeCFest} and compute its $L^2$ norm. 
\begin{align}
    & ||(e^{\Delta t \mathcal{L}_{\gamma}} - e^{\Delta t \tilde{\mathcal{L}}_{\gamma}})\big(1 + \Delta t\mathcal{C}_{\gamma}\big) q + (\Delta t)e^{\Delta t \tilde{\mathcal{L}}_{\gamma}} (\mathcal{C}_{\gamma} - \tilde{\mathcal{C}}_{\gamma}) q + (r_1 - r_{1}') q ||_{L^2} \notag \\
     \leq\ &||(e^{\Delta t \mathcal{L}_{\gamma}} - e^{\Delta t \tilde{\mathcal{L}}_{\gamma}})\big(1 + \Delta t\mathcal{C}_{\gamma}\big) q||_{L^2} + ||(\Delta t)e^{\Delta t \tilde{\mathcal{L}}_{\gamma}} \mathcal{H} q||_{L^2} + ||(r_1 - r_{1}') q ||_{L^2} \notag \\
    \leq \ & ||(e^{\Delta t \mathcal{L}_{\gamma}} - e^{\Delta t \tilde{\mathcal{L}}_{\gamma}})\big(1 + \Delta t\mathcal{C}_{\gamma}\big) q||_{L^2} + (\Delta t)|| \mathcal{H} q||_{L^2}  \notag \\
    & + \frac{1}{2}(\Delta t)^2 (||\mathcal{C}_{\gamma}||^2_{L^2} + ||\tilde{\mathcal{C}}_{\gamma}||^2_{L^2})||q||_{L^2} \notag \\
    \leq \ & (\Delta t)||\mathcal{G} q||_{L^2} + (\Delta t)^2 ||\mathcal{G} \mathcal{C}_{\gamma}q||_{L^2} + (\Delta t)|| \mathcal{H} q||_{L^2} \notag \\
    &+ \frac{1}{2}(\Delta t)^2 (||\mathcal{C}_{\gamma}||^2_{L^2} + ||\tilde{\mathcal{C}}_{\gamma}||^2_{L^2}) ||q||_{L^2}.
    \label{erreste3partdeltat}
\end{align}

Then substitute this result back to \eqref{ErrEst3part}, and we can get the error estimate 
\begin{align}
    e_3 &= || \prod_{h=1}^H e^{\Delta t_h \mathcal{L} (\bm x,\omega)} e^{\Delta t_h \mathcal{C} (\bm x,\omega)}q - \prod_{h=1}^H e^{\Delta t_h \tilde{\mathcal{L}} (\bm x,\omega)} e^{\Delta t_h \tilde{\mathcal{C}} (\bm x,\omega)} q ||_{L^2} \notag \\
    &\leq  C_1(q,\omega) \Delta t + C_2(||\mathcal{G}q||_{L^2} + ||\mathcal{H}q||_{L^2}) .
\end{align}
Here the coefficients $C_1(q,\omega) =T e^{\gamma T} \cdot \Big(||\mathcal{G}(\bm x,\omega) \mathcal{C}_{\gamma}(\bm x,\omega)q||_{L^2} +$$ \frac{1}{2} (||\mathcal{C}_{\gamma}(\bm x,\omega)||^2_{L^2} + ||\tilde{\mathcal{C}}_{\gamma}(\bm x, \omega) ||^2_{L^2})$ $||q||_{L^2}\Big)$,\ $C_2 = Te^{\gamma T}$. Then due to $||\tilde{\mathcal{C}}_{\gamma}(\bm x, \omega) ||^2_{L^2}$,  $||\mathcal{C}_{\gamma}(\bm x, \omega) ||^2_{L^2} $ $<1$ are both satisfied with probability greater than $p$, we can obtain that $prob(B_1) > p$, where $B_1=$
\begin{equation}
    \Big\{ e_3 \leq Te^{\gamma T} \cdot \big(||\bm v - \tilde{\bm v}||_{L^2}\cdot \big(||\nabla_{\bm x} q||_{L^2} + \lambda ||q||_{L^2}\big) + (\Delta t) \big(||\bm v - \tilde{\bm v}||_{L^2}\cdot||\nabla_{\bm x} q||_{L^2} + ||q||_{L^2}\big)\big) \Big\}.
\end{equation}

\end{proof}

\section{Proof of Theorem \ref{Thm::Errofe3ande}}
\label{sec::AppendixProofThm3_7}
\begin{proof}

\begin{align}
    e^{\Delta t \mathcal{L}_{\gamma}^{j} } e^{\Delta t \mathcal{C}_{\gamma}^{j} } - & e^{\Delta t \mathcal{L}_{\gamma} }e^{\Delta t \mathcal{C}_{\gamma} } 
    = e^{\Delta t \mathcal{L}_{\gamma}}e^{\Delta t \mathcal{C}_{\gamma}} [e^{-\Delta t \mathcal{C}_{\gamma}} e^{-\Delta t \mathcal{L}_{\gamma}} - e^{-\Delta t \mathcal{C}_{\gamma}^{j}} e^{-\Delta t \mathcal{L}_{\gamma}^{j}}]e^{\Delta t \mathcal{L}_{\gamma}^{j}} e^{\Delta t \mathcal{C}_{\gamma}^{j}} \notag \\
     = &e^{\Delta t \mathcal{L}_{\gamma}}e^{\Delta t \mathcal{C}_{\gamma}} [e^{-\Delta t \mathcal{C}_{\gamma}} e^{-\Delta t \mathcal{L}_{\gamma}} - e^{-\Delta t \mathcal{C}_{\gamma}^{j}} e^{-\Delta t \mathcal{L}_{\gamma}^{j}}]e^{\Delta t \mathcal{L}_{\gamma}} e^{\Delta t \mathcal{C}_{\gamma}}\ + \notag\\
    & e^{\Delta t \mathcal{L}_{\gamma}}e^{\Delta t \mathcal{C}_{\gamma} } [e^{-\Delta t \mathcal{C}_{\gamma} } e^{-\Delta t \mathcal{L}_{\gamma} } - e^{-\Delta t \mathcal{C}_{\gamma}^{j} } e^{-\Delta t \mathcal{L}_{\gamma}^{j} }][e^{\Delta t \mathcal{L}_{\gamma}^{j} } e^{\Delta t \mathcal{C}_{\gamma}^{j} } - e^{\Delta t \mathcal{L}_{\gamma} } e^{\Delta t \mathcal{C}_{\gamma} }].
\end{align}

Now applying $e^{\Delta t \mathcal{L}_{\gamma}^{j} } e^{\Delta t \mathcal{C}_{\gamma}^{j} }  - e^{\Delta t \mathcal{L}_{\gamma} }e^{\Delta t \mathcal{C}_{\gamma} }$ on $q$, we can get
\begin{align}
    ||e^{\Delta t \mathcal{L}_{\gamma}^{j} } & e^{\Delta t \mathcal{C}_{\gamma}^{j} }q  - e^{\Delta t \mathcal{L}_{\gamma} }e^{\Delta t \mathcal{C}_{\gamma} } q||_{L^2} \leq  ||e^{-\Delta t \mathcal{C}_{\gamma}} e^{-\Delta t \mathcal{L}_{\gamma}} - e^{-\Delta t \mathcal{C}_{\gamma}^{j}} e^{-\Delta t \mathcal{L}_{\gamma}^{j}}||_{L^2} \cdot \notag \\ &||e^{\Delta t \mathcal{L}_{\gamma} }e^{\Delta t \mathcal{C}_{\gamma}}||_{L^2}^2 \cdot ||q||_{L^2} 
    + ||e^{\Delta t \mathcal{L}_{\gamma} }e^{\Delta t \mathcal{C}_{\gamma}}||_{L^2} \cdot \notag \\ 
    &||e^{-\Delta t \mathcal{C}_{\gamma}} e^{-\Delta t \mathcal{L}_{\gamma}} - e^{-\Delta t \mathcal{C}_{\gamma}^{j}} e^{-\Delta t \mathcal{L}_{\gamma}^{j}}||_{L^2} \cdot ||e^{\Delta t \mathcal{L}_{\gamma}^{j} } e^{\Delta t \mathcal{C}_{\gamma}^{j} }q  - e^{\Delta t \mathcal{L}_{\gamma} }e^{\Delta t \mathcal{C}_{\gamma} } q||_{L^2}.
    \label{Convergencerateofe3}
\end{align}

Then for the term $||e^{-\Delta t \mathcal{C}_{\gamma}} e^{-\Delta t \mathcal{L}_{\gamma}} - e^{-\Delta t \mathcal{C}_{\gamma}^{j}} e^{-\Delta t \mathcal{L}_{\gamma}^{j}}||_{L^2}$, we can define $C_{\gamma} = e^{2\gamma\Delta t}, \mathcal{C}_{\bar{\gamma}} = \mathcal{C} + M_2^*,\ \mathcal{L}_{\bar{\gamma}} = \mathcal{L} + \tilde \gamma$ and similarly get $\mathcal{C}_{\bar{\gamma}}^{j}$ and $\mathcal{L}_{\bar{\gamma}}^{j}$. Here $M_2^*$ and $\tilde \gamma$ are the upper bound defined in Lemma \ref{Lemmma:Defpandupperbound}. Then following the idea above, we can get 
\begin{align}
    ||e^{-\Delta t \mathcal{C}_{\gamma}} e^{-\Delta t \mathcal{L}_{\gamma}} - e^{-\Delta t \mathcal{C}_{\gamma}^{j}} e^{-\Delta t \mathcal{L}_{\gamma}^{j}}||_{L^2} = C_{\gamma} ||e^{-\Delta t \mathcal{C}_{\bar{\gamma}}} e^{-\Delta t \mathcal{L}_{\bar{\gamma}}} - e^{-\Delta t \mathcal{C}_{\bar{\gamma}}^{j}} e^{-\Delta t \mathcal{L}_{\bar{\gamma}}^{j}}||_{L^2}.
\end{align}

Then we can estimate $||e^{-\Delta t \mathcal{C}_{\bar{\gamma}}} e^{-\Delta t \mathcal{L}_{\bar{\gamma}}} - e^{-\Delta t \mathcal{C}_{\bar{\gamma}}^{j}} e^{-\Delta t \mathcal{L}_{\bar{\gamma}}^{j}}||_{L^2}$ following the idea mentioned from \eqref{TaylorexpL'} to \eqref{erreste3partdeltat}. Here we need to emphasize that for higher order terms of $\Delta t$ in the expansion, we still end up with the factor $||\bm v - \bm v_{j}||_{L^2}$. So we can get that 
$\lim_{j \to \infty} ||e^{-\Delta t \mathcal{C}_{\bar{\gamma}}} e^{-\Delta t \mathcal{L}_{\bar{\gamma}}} - e^{-\Delta t \mathcal{C}_{\bar{\gamma}}^{j}} e^{-\Delta t \mathcal{L}_{\bar{\gamma}}^{j}}||_{L^2} = 0$. Then we can assume for $j \geq N^*$, 
\begin{equation}
    ||e^{\Delta t \mathcal{L}_{\gamma} }e^{\Delta t \mathcal{C}_{\gamma}}||_{L^2} \cdot ||e^{-\Delta t \mathcal{C}_{\gamma}} e^{-\Delta t \mathcal{L}_{\gamma}} - e^{-\Delta t \mathcal{C}_{\gamma}^{j}} e^{-\Delta t \mathcal{L}_{\gamma}^{j}}||_{L^2} \leq \frac{1}{2}.
\end{equation}

Then substitute back into \eqref{Convergencerateofe3}, we can get for $\forall j \geq N^*$
\begin{align}
    ||e^{\Delta t \mathcal{L}_{\gamma}^{j} } e^{\Delta t \mathcal{C}_{\gamma}^{j} }q - e^{\Delta t \mathcal{L}_{\gamma} }e^{\Delta t \mathcal{C}_{\gamma} } q||_{L^2} \leq\ & 2\ ||e^{-\Delta t \mathcal{C}_{\gamma}} e^{-\Delta t \mathcal{L}_{\gamma}} - e^{-\Delta t \mathcal{C}_{\gamma}^{j}} e^{-\Delta t \mathcal{L}_{\gamma}^{j}}||_{L^2} \cdot \notag \\ &||e^{\Delta t \mathcal{L}_{\gamma} }e^{\Delta t \mathcal{C}_{\gamma}}||_{L^2}^2 \cdot ||q||_{L^2}.
\end{align}

Then we can get for $\forall j \geq N^*$
\begin{equation}
    ||e^{\Delta t \mathcal{L}_{\gamma}^{j} } e^{\Delta t \mathcal{C}_{\gamma}^{j} }q - e^{\Delta t \mathcal{L}_{\gamma} }e^{\Delta t \mathcal{C}_{\gamma} } q||_{L^2} \leq 2 ||e^{-\Delta t \mathcal{C}_{\gamma}} e^{-\Delta t \mathcal{L}_{\gamma}} - e^{-\Delta t \mathcal{C}_{\gamma}^{j}} e^{-\Delta t \mathcal{L}_{\gamma}^{j}}||_{L^2} \cdot ||q||_{L^2}.
\end{equation}

So that for  $||e^{\Delta t \mathcal{L}_{\gamma}^{j} } e^{\Delta t \mathcal{C}_{\gamma}^{j} }q - e^{\Delta t \mathcal{L}_{\gamma} }e^{\Delta t \mathcal{C}_{\gamma} } q||_{L^2}$, it will have a convergence rate $O\big(||\bm v - \bm v_{j}|| \cdot (\Delta t + (\Delta t)^2)\big)$. Then for $2 ||e^{-\Delta t \mathcal{C}_{\gamma}} e^{-\Delta t \mathcal{L}_{\gamma}} - e^{-\Delta t \mathcal{C}_{\gamma}^{j}} e^{-\Delta t \mathcal{L}_{\gamma}^{j}}||_{L^2} \cdot ||q||_{L^2}$, we have
\begin{align}
    LHS =&\ 2 e^{2\gamma \Delta t} ||e^{-\Delta t \mathcal{C}_{\bar{\gamma}}} e^{-\Delta t \mathcal{L}_{\bar{\gamma}}} - e^{-\Delta t \mathcal{C}_{\bar{\gamma}}^{j}} e^{-\Delta t \mathcal{L}_{\bar{\gamma}}^{j}}||_{L^2} \cdot ||q||_{L^2} \notag \\
    \leq &\ 2 e^{2\gamma \Delta t} \big[\Delta t ||\mathcal{G}||_{L^2}  + (\Delta t)^2 ||\mathcal{C}_{\bar{\gamma}} \mathcal{G}||_{L^2} + \Delta t ||\mathcal{H}||_{L^2} + \notag \\ 
    & \frac{1}{2} (\Delta t)^2 (||\mathcal{C}_{\bar{\gamma}}||_{L^2} + ||\mathcal{C}_{\bar{\gamma}}^{j}||_{L^2})\big] \cdot ||q||_{L^2} .
\end{align}
Here $\mathcal{G}$ and $\mathcal{H}$ are defined in \ref{Thm::e3_34}. Then for the entire evolution process, we can obtain that 
\begin{align}
    e_3 \leq &\ 2T e^{\gamma (T + 2\Delta t)} \big[||\mathcal{G}||_{L^2}  + (\Delta t) ||\mathcal{C}_{\bar{\gamma}} \mathcal{G}||_{L^2} + ||\mathcal{H}||_{L^2} +  \notag \\
    & \frac{\Delta t}{2} (||\mathcal{C}_{\bar{\gamma}}||_{L^2} + ||\mathcal{C}_{\bar{\gamma}}^{j}||_{L^2})\big] \cdot ||q||_{L^2} \notag \\
    \leq &\ 2T e^{\gamma (T + 2\Delta t)} \big[
    2M_2^* \Delta t + ||\bm v - \bm v_{j}||_{L^2} (C_D(p) + \lambda) +  \notag \\
    &\ 2M_2^* C_D(p) \Delta t ||\bm v - \bm v_{j}||_{L^2}\big] \cdot ||q||_{L^2}.
\end{align}
Here $C_D(p)$ is a constant that depends on the selection of $\gamma$ which means it will depend on $p$. Then we can get 
\begin{align}
    e_3 \leq 2T e^{\gamma (T + 2\Delta t)} \big[(M_2^*)^2 (\Delta t)^2 + 2M_2^* \Delta t + & (C_D(p)+\lambda) ||\bm v - \bm v_{j}||_{L^2} + \notag \\
    & C_D(p)^2 ||\bm v - \bm v_{j}||^2_{L^2} \big]\cdot ||q||_{L^2} .
    \label{e3ConBeforeCauchy} 
\end{align}

Since $\lim_{\Delta t \to 0} \frac{(M_2^*)^2 (\Delta t)^2}{2M_2^* \Delta t} = 0$ and $\lim_{j \to \infty} \frac{C_D(p)^2 ||\bm v - \bm v_{j}||^2_{L^2}}{(C_D(p)+\lambda) ||\bm v - \bm v_{j}||_{L^2}} = 0$, as $\Delta t \to 0$ and $j \to \infty$, we have

\begin{equation}
    e_3 \leq  2T e^{\gamma (T + 2\Delta t)} \big[2M_2^* \Delta t +  (C_D(p)+\lambda) ||\bm v - \bm v_{j}||_{L^2} \big]\cdot ||q||_{L^2} + O((\Delta t)^2) + O(||\bm v - \bm v_{j}||^2_{L^2}).
\end{equation}
By applying the Cauchy-Schwarz inequality, we can get 
\begin{equation}
    2M_2^* \Delta t + (C_D(p)+\lambda) ||\bm v - \bm v_{j}||_{L^2} \leq \sqrt{4(M_2^*)^2+(C_D(p) + \lambda)^2} \sqrt{(\Delta t)^2 + ||\bm v - \bm v_{j}||_{L^2}^2}.
\end{equation}

By denoting $C_3(\lambda, p) = 4Te^{\gamma(T+1)}\cdot \sqrt{4(M_2^*)^2+(C_D(p) + \lambda)^2}$, we can obtain \eqref{Theorem3.6e3Con}. Then due to the error estimate of $e_1$, $e_2$ and $O(\frac{4 C_s}{3}(\Delta t)^{\frac{1}{2}} + 2M_2^* \Delta t) = O(\frac{4 C_s}{3}(\Delta t)^{\frac{1}{2}})$, the total error estimate $e_{op}$ in approximating the solution operator $\mathcal{U}(T,0)$ will be 
\begin{equation}  
    e_{op} = O\big(\frac{4 C_s}{3}(\Delta t)^{\frac{1}{2}} \big) + O\big((C_D(p) + \lambda)||\bm v - \bm v_{j}||_{L^2}\big).
\end{equation}  
\end{proof}

\bibliographystyle{siamplain}
\bibliography{references}
\end{document}